\documentclass[onefignum]{siamart250211}

\usepackage{aliascnt}
\usepackage{hyperref}
\usepackage{cleveref}
\usepackage{amssymb}
\usepackage{amsmath}
\usepackage{graphicx}
\usepackage{epstopdf}
\usepackage[caption=false]{subfig}
\usepackage{bm}
\usepackage{lipsum}
\usepackage{amsfonts}
\usepackage{algorithmic}
\usepackage{tikz}
\usepackage{url}
\allowdisplaybreaks

\ifpdf
  \DeclareGraphicsExtensions{.eps,.pdf,.png,.jpg}
\else
  \DeclareGraphicsExtensions{.eps}
\fi

\providecommand{\abs}[1]{\left|#1\right|}
\providecommand{\norm}[1]{\left\Vert#1\right\Vert}

\crefname{theorem}{theorem}{theorems}
\Crefname{theorem}{Theorem}{Theorems}

\newaliascnt{remark}{theorem}
\newsiamremark{remark}{Remark}
\aliascntresetthe{remark}
\crefname{remark}{remark}{remarks}
\Crefname{remark}{Remark}{Remarks}

\newaliascnt{example}{theorem}
\newsiamremark{example}{Example}
\aliascntresetthe{example}
\crefname{example}{example}{examples}
\Crefname{example}{Example}{Examples}

\title{Magnus Methods for Stochastic Delay-Differential Equations\thanks{Submitted to the editors \protect\today.
\funding{This work was supported by the Queensland University of Technology Postgraduate Research Award (QUTPRA).}}}

\author{Mitchell T. Griggs\thanks{School of Mathematical Sciences, Queensland University of Technology (QUT), Brisbane, Australia
  (\email{mitchell.griggs@connect.qut.edu.au}, \email{kevin.burrage@qut.edu.au}, \email{pamela.burrage@qut.edu.au}).}
\and Kevin Burrage\footnotemark[2]
\and Pamela M. Burrage\footnotemark[2]}

\headers{Magnus Methods for SDDEs}{Mitchell Griggs, Kevin Burrage, and Pamela Burrage}

\begin{document}
\maketitle
\begin{abstract}
This paper introduces Magnus-based methods for solving stochastic delay-differential equations (SDDEs). We construct Magnus--Euler--Maruyama (MEM) and Magnus--Milstein (MM) schemes by combining stochastic Magnus integrators with Taylor methods for SDDEs. These schemes are applied incrementally between multiples of the delay times. We present proofs of their convergence orders and demonstrate these rates through numerical examples and error graphs. Among the examples, we apply the MEM and MM schemes to both linear and nonlinear problems. We also apply the MEM scheme to a stochastic partial delay-differential equation (SPDDE), comparing its performance with the traditional Euler--Maruyama (EM) method. Under fine spatial discretization, the MEM scheme remains numerically stable while the EM method becomes unstable, yielding a significant computational advantage.
\end{abstract}




\section{Introduction}\label{Section1}

This paper introduces Magnus numerical schemes for simulating solutions of the semilinear It\^o SDDE
\begin{align}
\mathrm{d}X(t)&=[A_0X(t)+f(t,X(t),X(t-\tau_1),\ldots,X(t-\tau_K))]\,\mathrm{d}t\nonumber\\
&\;+\sum_{j=1}^m[A_jX(t)+g_j(t,X(t),X(t-\tau_1),\ldots,X(t-\tau_K))]\,\mathrm{d}W_j(t),\quad t\in[0,T],\label{Equation1}\\
X(t)&=\phi(t),\quad t\in[-\tau,0],\nonumber
\end{align}
with constant matrices $A_0,\ldots,A_m\in\mathbb{R}^{d\times d}$, constant delays $\tau_1,\ldots,\tau_K>0$, largest delay $\tau=\max\{\tau_1,\ldots,\tau_K\}$, drift $f:[0,T]\times\mathbb{R}^d\times(\mathbb{R}^d)^K\to\mathbb{R}^d$ and volatility functions $g_1,\ldots,g_m:[0,T]\times\mathbb{R}^d\times(\mathbb{R}^d)^K\to\mathbb{R}^d$, independent Wiener processes $W_1,\ldots,W_m$, and history process $\phi$. Each of $f,g_1,\ldots,g_m$ is $C^1$ in the first component and $C^2$ in all other components, while the history process satisfies the moment bound $\sup_{t\in[-\tau,0]}\mathbb{E}[\abs{\phi(t)}^4]<\infty$. We stipulate $\phi$ as having bounded fourth moment in our derivation of the MM scheme, but this condition is not prohibitively restrictive as $\phi$ is often deterministic in practical applications.

Our schemes offer new tools for simulating SDDEs, which arise in a wide range of applications. These applications include modeling disease spread \cite{RihanAlsakaji2021,Shahid2024}, option pricing \cite{MaoSebanis2013,McWilliams2011,PangYong2019}, investment models \cite{Crandall2016,HeGuYao2024}, population models \cite{Mao2007,WangWangChen2019}, predator-prey models \cite{BaharMao2004,Rihan2020}, respiratory-control systems \cite{HartungTuri2013}, postural dynamics \cite{Boulet2010}, and also physical models such as Langevin equations and crossing-time problems \cite{Atay2010}. The Magnus schemes introduced in this paper also provide an efficient method to simulate certain SPDDE solutions, which we see in our final example.

The Magnus SDDE schemes are the combination of the Magnus schemes for stochastic differential equations (SDEs) by Burrage and Burrage \cite{BurrageBurrage1999} with the Taylor (EM and Milstein) SDDE schemes given by K\"uchler and Platen \cite{KuchlerPlaten2000}. K\"{u}chler and Platen produce EM and Milstein schemes for the single-delay ($K=1$) version of \cref{Equation1} by applying the EM and Milstein SDE schemes of Maruyama and Milstein \cite{Maruyama1955,Milstein1975} on each of the intervals $[0,\tau_1],[\tau_1,2\tau_1],\ldots$. This step-by-step approach is an extension from that first considered in 1963 by Bellman \cite{BellmanCooke1963}, for the deterministic case ($m=0$) of ordinary delay-differential equations. In the prior decade, Magnus \cite{Magnus1954} introduced numerical schemes for ordinary differential equations based on the exponential map, which Burrage and Burrage \cite{BurrageBurrage1999} extended to homogeneous, linear SDEs, developing the Magnus SDE schemes as part of Burrage's doctoral research \cite{Burrage1999}. Convergence analysis of these Magnus schemes (for linear, homogeneous SDEs) was conducted in Stratonovich setting, by Lord, Malham, and Wiese \cite{LordMalhamWiese2008}. Yang, Burrage, Komori, Burrage, and Ding \cite{BurrageBurrage2021} extended the Magnus SDE schemes further, to the case of semilinear SDEs.

We present the MEM and MM schemes, in \cref{Section2}. We then give the derivation of these schemes, in \cref{Section3}. We also show examples of the schemes, in \cref{Section4}.

\subsection{Motivation}

This paper focuses on the SDDE \cref{Equation1}, but the Magnus SDDE schemes that we introduce here provide a method to efficiently simulate SPDDEs. As an example, we consider the one-dimensional (non-delayed) stochastic heat equation with cooling,
\begin{align}
\mathrm{d}U(t,x)&=\left[D\frac{\partial^2U(t,x)}{\partial x^2}+C_U(t,x)\right]\,\mathrm{d}t\label{Equation2}\\
&\hspace{2cm}+cU(t,x)\,\mathrm{d}W^c(t,x),\quad(t,x)\in(0,\infty)\times(0,1),\nonumber\\
\vphantom{\int}U(t,0)&=U(t,1)=0,\quad U(t,x)=T_0(x),\quad x\in[0,1],\nonumber
\end{align}
where $U$ is the solution process (temperature), $T_0:[0,1]\to\mathbb{R}$ is the initial temperature, $D>0$ is the diffusion coefficient, $C_U:(0,\infty)\times\mathbb{R}\to\mathbb{R}$ is a cooling function applied to alter the temperature $U$, and $c>0$ is a scaling constant. The stochastic process $W^c$ is a $Q$-Wiener process, as defined by Lord, Powell, and Shardlow \cite{LordPowellShardlow2014}, but we also provide this definition in \cref{AppendixSPDE}. In order to account for delayed response time, the cooling is applied with a time delay of $\tau$. We model this by extending the domain of the cooling function, so that with $C_U:(-\tau,\infty)\times\mathbb{R}\to\mathbb{R}$, the stochastic partial differential \cref{Equation2} becomes the SPDDE
\begin{align}
\mathrm{d}U(t,x)&=\left[D\frac{\partial^2U(t,x)}{\partial x^2}+C_U(t-\tau,x)\right]\,\mathrm{d}t\label{Equation3}\\
&\hspace{2cm}+cU(t,x)\,\mathrm{d}W^c(t,x),\quad(t,x)\in(0,\infty)\times(0,1),\nonumber\\
\vphantom{\int}U(t,0)&=U(t,1)=0,\quad U(t,x)=T_0(x),\quad (t,x)\in[-\tau,0]\times[0,1],\nonumber
\end{align}
with the temperature held fixed on $[-\tau,0]$. To simulate a solution of \cref{Equation3}, we apply a forward time-centered space method from the work of Roache \cite{Roache1972}, where we select a spatial step size of $\Delta x=1/d$ for some $d\in\mathbb{N}$, and then approximate the solution $U(\cdot,x_j)$, at the cross sections $x=x_j=j\,\Delta x$ for $j=0,\ldots,d$, with the system of SDDEs
\begin{align}
\mathrm{d}U_0(t)&=0\,\mathrm{d}t,\label{Equation4}\\
\mathrm{d}U_j(t)&=\left[D\frac{U_{j-1}(t)+U_{j+1}(t)-2U_j(t)}{(\Delta x)^2}+C_U(t-\tau,x_j)\right]\,\mathrm{d}t+\frac{cU_j(t)}{\sqrt{\Delta x}}\,\mathrm{d}W_j^c(t),\nonumber\\
\mathrm{d}U_d(t)&=0\,\mathrm{d}t,\nonumber
\end{align}
for $j=1,\ldots,d-1$ and $t\geqslant0$, equipped with $(U_0(t),\ldots,U_d(t))=(T_0(x_0),\ldots,T_0(x_d))$ when $t\in[-\tau,0]$. We denote by $W_j^c$ the cross section $(W^c(t,x_j))_{t\in[0,\infty)}$ of $W^c$.

The simplest method for simulating \cref{Equation4} is the Euler--Maruyama scheme, but that is unfortunately restricted by the need for a small step size in time. If we ignore the cooling function and noise (supposing $C_U=0$ and $c=0$), then applying the EM method (which reduces to the regular Euler scheme) to the system \cref{Equation4} requires a temporal step size $h$ that satisfies the Neumann stability condition
\begin{equation}2D\,h<(\Delta x)^2.\label{StabilityCondition}\end{equation}
Details of this requirement may be found in the doctoral thesis by Hayman \cite{Hayman1988}. When a fine spatial discretization is required, the requirement that $h$ decreases quadratically with $\Delta x$ requires significant computation time for the Euler implementation. The Magnus schemes that we introduce in this paper are not hindered by this condition.

\section{Magnus SDDE schemes}\label{Section2}

This section presents the MEM and MM schemes, beginning with a review of existing SDDE Taylor and SDE Magnus methods. In this paper, $Y=(Y_n)_{n=0,\ldots,N}$ denotes a numerical approximation of the solution $X=(X(t))_{t\in[0,T]}$ of an SDE, and $Y=(Y_n)_{n=-p,\ldots,N}$ denotes a numerical approximation for the solution $X=(X(t))_{t\in[-\tau,T]}$ of \cref{Equation1}. That is, in each case, $Y$ is a finite sequence where each $Y_n=Y(t_n)$ is an approximation of $X(t_n)$ at the time point $t_n$, where $(t_n)_{n=-p,\ldots,N}$ is a time mesh satisfying $t_n<t_{n+1}$ for each $n=-p,\ldots,N$, as well as $t_{-p}=-\tau$, $t_0=0$, and $t_N=T$.

\subsection{Existing Taylor SDDE schemes and Magnus SDE methods}\label{Section2.1}

\begin{definition}
A stochastic process $X=(X(t))_{t\in[-\tau,T]}$ is a \emph{(strong) solution} of \cref{Equation1} if $X$ satisfies $X(t)=\phi(t)$ when $t\leqslant0$, and
\begin{align*}
X(t)&=X(0)+\int_0^t[A_0X(s)+f(s,X(s),X(s-\tau_1),\ldots,X(s-\tau_K))]\,\mathrm{d}s\\
&\quad\quad+\sum_{j=1}^m\int_0^t[A_jX(s)+g_j(s,X(s),X(s-\tau_1),\ldots,X(s-\tau_K))]\,\mathrm{d}W_j(s),
\end{align*}
for all $t\in[0,T]$. The solutions of \cref{Equation1} are \emph{(pathwise) unique} if whenever $X=(X(t))_{t\in[-\tau,T]}$ and $Y=(Y(t))_{t\in[-\tau,T]}$ are solutions then
\[\mathbb{P}\left(\sup_{-\tau\leqslant s\leqslant t}\abs{Y(s)-X(s)}>0\right)=0,\textrm{ for every }t\in[-\tau,T].\]
\end{definition}

In order to ensure (pathwise) existence and uniqueness of a (strong) solution for \cref{Equation1}, we assume that the following linear-growth and spatial Lipschitz bounds hold. Letting $g_0=f$, for each $j=0,\ldots,m$, there exist constants $L_j^{(1)},L_j^{(2)}>0$ such that
\begin{align}\abs{g_j(t,x_0,x_1,\ldots,x_K)}^2&\leqslant L_j^{(1)}\bigg(1+\sum_{k=0}^K\abs{x_k}^2\bigg)\quad\textrm{and}\label{LinearGrowthBound}\\
\abs{g_j(t,x_0,x_1,\ldots,x_K)-g_j(t,y_0,y_1,\ldots,y_K)}&\leqslant L_j^{(2)}\sum_{k=0}^K\abs{x_k-y_k},\label{LipschitzBound}\end{align}
for all $(t,x_0,x_1,\ldots,x_K),(t,y_0,y_1,\ldots,y_K)\in[0,T]\times\mathbb{R}^d\times(\mathbb{R}^d)^K$. With these conditions, the following existence-uniqueness result is similar to that of the single-delay case, as shown by K\"uchler and Platen \cite{KuchlerPlaten2000}, but with additional terms. We also provide an alternative proof of the following theorem, in the supplementary material \cref{AlternativeProofofPLS}, based on the work by Kloeden and Shardlow \cite{KloedenShardlow2012}. Further existence-uniqueness results for \cref{Equation1} may be found in the texts by Mohammed \cite{Mohammed1998} and Mao \cite{Mao2007}.

\begin{theorem}\label{PLS}
Suppose that $f:[0,T]\times\mathbb{R}^d\times(\mathbb{R}^d)^K\to\mathbb{R}^d$ is Lebesgue integrable, each $g_j:[0,T]\times\mathbb{R}^d\times(\mathbb{R}^d)^K\to\mathbb{R}^d$ is square integrable, and also that each $f,g_j$ satisfies the linear-growth and Lipschitz bounds \cref{LinearGrowthBound} and \cref{LipschitzBound}. Additionally, suppose that $\phi$ is a stochastic process with $\int_{-\tau}^0\mathbb{E}[\abs{\phi(s)}^2]\,\mathrm{d}s<\infty$. Under these assumptions, there exists a pathwise-unique, strong solution $(X(t))_{t\in[-\tau,T]}$ of \cref{Equation1}, satisfying $\sup_{t\in[-\tau,T]}\mathbb{E}[\abs{X(t)}^2]<\infty$.
\end{theorem}


The SDDE Taylor schemes by K\"uchler and Platen \cite{KuchlerPlaten2000} are defined for the equation
\begin{align}
\mathrm{d}X(t)&=f(t,X(t),X(t-\tau_1),\ldots,X(t-\tau_K))\,\mathrm{d}t\label{Equation5}\\
&\quad+\sum_{j=1}^mg_j(t,X(t),X(t-\tau_1),\ldots,X(t-\tau_K))\,\mathrm{d}W_j(t),\quad t\in[0,T],\nonumber\\
X(t)&=\phi(t),\quad t\in[-\tau,0].\nonumber
\end{align}
The numerical scheme $Y=(Y_n)_{n=-p,\ldots,N}$ is the Euler--Maruyama approximation for \cref{Equation5} when $Y_n=\phi(t_n)$ if $n\leqslant0$, and
\begin{align}
Y_{n+1}&=Y_n+f(t_n,Y_n,Y^{\tau_1}_n,\ldots,Y^{\tau_K}_n)\,h_n\label{EM}\\
&\quad+\sum_{j=1}^mg_j(t_n,Y_n,Y^{\tau_1}_n,\ldots,Y^{\tau_K}_n)\,\Delta W_j(t_n,t_{n+1}),\quad n=0,\ldots,N-1,\nonumber
\end{align}
where $h_n=t_{n+1}-t_n$ is the step size on $[t_n,t_{n+1}]$, while $Y^{\tau_k}_n=Y(t_n-\tau_k)$ and $\Delta W_j(t_n,t_{n+1})=W_j(t_{n+1})-W_j(t_n)$ for each $k=1,\ldots,K$ and $j=1,\ldots,m$. The Milstein approximation for \cref{Equation5} is defined similarly, but with \cref{EM} replaced by
\begin{align}Y_{n+1}&=Y_n\!+\!f(t_n,Y_n,Y^{\tau_1}_n,\ldots,Y^{\tau_K}_n)\,h_n\!+\!\sum_{j=1}^m\!g_j(t_n,Y_n,Y^{\tau_1}_n,\ldots,Y^{\tau_K}_n)\,\Delta W_j(t_n,t_{n+1})\nonumber\\
&\quad+\sum_{i=1}^m\sum_{j=1}^m\nabla_x g_j(t_n,Y_n,Y^{\tau_1}_n,\ldots,Y^{\tau_K}_n)\cdot g_i(t_n,Y_n,Y^{\tau_1}_n,\ldots,Y^{\tau_K}_n)\,I_{ij}(t_n,t_{n+1})\label{Milstein}\\
&\quad+\sum_{k=1}^K\mathbb{I}(t_n\geqslant\tau_k)\,\sum_{i=1}^m\sum_{j=1}^m\bigg(\nabla_{x_{\tau_k}} g_j(t_n,Y_n,Y^{\tau_1}_n,\ldots,Y^{\tau_K}_n) \nonumber\\
&\phantom{\quad+\sum_{k=1}^K\mathbb{I}(t_n\geqslant\tau_k)\,\sum_{i=1}^m\sum_{j=1}^m}\quad \cdot g_i(t_n^{\tau_k},Y^{\tau_k}_n,Y^{\tau_1,\tau_k}_n,\ldots,Y^{\tau_K,\tau_k}_n)\bigg)\,I_{ij}^{\tau_k}(t_n,t_{n+1}),\nonumber\end{align}
for $n=0,\ldots,N-1$, where $t_n^{\tau_k}=t_n-\tau_k$ and $Y_n^{\tau_l,\tau_k}=Y(t_n-\tau_l-\tau_k)$ for $k,l=1,\ldots,K$, $\mathbb{I}$ is the indicator function defined by
\[\mathbb{I}(A)=\begin{cases}1&\textrm{if }A\textrm{ is true},\\ 0&\textrm{otherwise},\end{cases}\]
$I_{ij}(t_n,t_{n+1})$ and $I_{ij}^{\tau_k}(t_n,t_{n+1})$ are iterated stochastic integrals, respectively defined by
\begin{align}
I_{ij}(t_n,t_{n+1})&=\int_{t_n}^{t_{n+1}}\int_{t_n}^s\,\mathrm{d}W_i(u)\,\mathrm{d}W_j(s)\quad\textrm{and}\label{IijandIijtauk}\\ I_{ij}^{\tau_k}(t_n,t_{n+1})&=\int_{t_n}^{t_{n+1}}\int_{t_n}^s\,\mathrm{d}W_i(u-\tau_k)\,\mathrm{d}W_j(s),\nonumber
\end{align}
and $\nabla_x$, $\nabla_{x_{\tau_k}}$ are spatial and delayed Jacobians, defined by
\[\nabla_{x_{\tau_k}}g_j(t,\!x,\!x_{\tau_1},\!\ldots,\!x_{\tau_K})=\left(
                          \begin{array}{ccc}
                            \frac{\partial g_{1,j}}{\partial x_1^{\tau_k}}(t,\!x,\!x_{\tau_1},\!\ldots,\!x_{\tau_K}) & \cdots & \frac{\partial g_{1,j}}{\partial x_d^{\tau_k}}(t,\!x,\!x_{\tau_1},\!\ldots,\!x_{\tau_K}) \\
                            \vdots & \ddots & \vdots \\
                            \frac{\partial g_{d,j}}{\partial x_1^{\tau_k}}(t,\!x,\!x_{\tau_1},\!\ldots,\!x_{\tau_K}) & \cdots & \frac{\partial g_{d,j}}{\partial x_d^{\tau_k}}(t,\!x,\!x_{\tau_1},\!\ldots,\!x_{\tau_K}) \\
                          \end{array}
                        \right),\]
for $k=0,1,\ldots,K$, writing $x_{\tau_0}=x=(x_1,\ldots,x_d)^\intercal\in\mathbb{R}^d$, $x_{\tau_k}=(x_1^{\tau_k},\ldots,x_d^{\tau_k})^\intercal\in\mathbb{R}^d$, and $g_j=(g_{1,j},\ldots,g_{d,j})^\intercal\in\mathbb{R}^d$, for each $j=1,\ldots,m$.

The Magnus schemes by Yang, Burrage, Komori, Burrage, and Ding \cite{BurrageBurrage2021} numerically solve the equation
\begin{align}
\mathrm{d}X(t)&=[A_0X(t)\!+\!f(t,X(t))]\,\mathrm{d}t+\sum_{j=1}^m[A_jX(t)\!+\!g_j(t,X(t))]\,\mathrm{d}W_j(t),\;t\in[0,T],\label{Equation6}
\end{align}
where $X(0)=\xi$ is a random variable with a bounded fourth moment. These schemes are derived by first assuming that $X$ is of the form $X(t)=X_H(t)X_R(t)$ for solution $X_H$ of the homogeneous $d\times d$ matrix SDE
\begin{align}
\mathrm{d}X_\mathrm{H}(t)&=A_0X_\mathrm{H}(t)\,\mathrm{d}t+\sum_{j=1}^mA_jX_\mathrm{H}(t)\,\mathrm{d}W_j(t),\quad t\in[0,T],\label{dXH}
\end{align}
where $X_H(0)=I_d$ is the $d\times d$ identity matrix. By applying the vectorized It\^o expansion to the differential $\mathrm{d}X(t)=\mathrm{d}(X_H(t)X_R(t))$ (see \cite[page 6]{BurrageBurrage2021}), it follows that
\begin{align}
X_R(t)&=\xi+\int_0^tX_H^{-1}(s)\tilde{f}(s,X(s))\,\mathrm{d}s+\sum_{j=1}^m\int_0^tX_H^{-1}(s)g_j(s,X(s))\,\mathrm{d}W_j(t),\label{XR}
\end{align}
where $\tilde{f}(t,x)=f(t,x)-\sum_{j=1}^mA_jg_j(t,x)$. The Magnus schemes are produced by applying Magnus expansions to \cref{dXH} while also applying Taylor approximations to \cref{XR}. The MEM scheme is the numerical scheme $Y=(Y_n)_{n=0,\ldots,N}$ where $Y_n=Y(t_n)$ is defined at the time point $t_n$, with $t_0=0$ and $t_N=T$, by $Y_0=\xi$ and
\begin{align}
&\Omega^{[1]}(t_n,t_{n+1})=\bigg(A_0-\frac{1}{2}\sum_{j=1}^mA_j^2\bigg)\,h_n+\sum_{j=1}^mA_j\,\Delta W_j(t_n,t_{n+1}),\label{MEMOmega1}\\
&Y_{n+1}=\exp(\Omega^{[1]}(t_n,t_{n+1}))\bigg\{\!Y_n\!+\!\tilde{f}(t_n,Y_n)\,h_n\!+\!\sum_{j=1}^mg_j(t_n,Y_n)\,\Delta W_j(t_n,t_{n+1})\!\bigg\}\label{MEMYnplus1}
\end{align}
for $n=0,\ldots,N-1$. The MM scheme is defined similarly but with \cref{MEMOmega1} and \cref{MEMYnplus1} respectively replaced by
\begin{align}
&\hspace{-1cm}\Omega^{[2]}(t_n,t_{n+1})=\Omega^{[1]}(t_n,t_{n+1})+\frac{1}{2}\sum_{i=0}^m\sum_{j=i+1}^m[A_i,A_j](I_{ji}(t_n,t_{n+1})-I_{ij}(t_n,t_{n+1})),\nonumber\\
Y_{n+1}&=\exp(\Omega^{[2]}(t_n,t_{n+1}))\bigg\{Y_n+\tilde{f}(t_n,Y_n)\,h_n+\sum_{j=1}^mg_j(t_n,Y_n)\,\Delta W_j(t_n,t_{n+1})\nonumber\\
&\;\;+\sum_{i=1}^m\sum_{j=1}^m\left(\vphantom{\sum}\nabla_xg_j(t_n,\!Y_n)\cdot\left[A_iY_n+g_i(t_n,\!Y_n)\right]-A_ig_j(t_n,\!Y_n)\vphantom{\sum}\right)\,I_{ij}(t_n,\!t_{n+1})\bigg\},\label{MMYnplus1}
\end{align}
where $[\cdot,\cdot]$ denotes the Lie bracket, defined by $[A,B]=AB-BA$ for square matrices $A$ and $B$, while
\begin{equation*}I_{0j}(t_n,t_{n+1})=\int_{t_n}^{t_{n+1}}\int_{t_n}^s\,\mathrm{d}u\,\mathrm{d}W_j(s)\quad\textrm{and}\quad I_{j0}(t_n,t_{n+1})=\int_{t_n}^{t_{n+1}}\int_{t_n}^s\,\mathrm{d}W_j(u)\,\mathrm{d}s\end{equation*}
are iterated stochastic integrals that satisfy the relation $I_{0j}(t_n,t_{n+1})+I_{j0}(t_n,t_{n+1})=h_n\,\Delta W_j(t_n,t_{n+1})$, for each $j=1,\ldots,m$. The Magnus schemes may be generalised further. Kamm, Pagliarani, and Pascucci \cite{Kamm2021} give a general theorem that may be used to iteratively find the $q^\mathrm{th}$-order Magnus approximation $\exp(\Omega^{[q]}(0,t))$ of the solution $X_H$ to the linear homogeneous SDE \cref{dXH}
 as well as extend this result to the case of each $A_j$ being a bounded stochastic process independent of the Wiener processes. Derivation of the Magnus schemes, as well as the stochastic Taylor schemes, is contingent on the following results.


\begin{theorem}\label{FubiniTonelliIto}
If $y:[0,T]\times\mathbb{R}^d\to\mathbb{R}^d$ is an adapted stochastic process with bounded first moment, then
\begin{equation}\mathbb{E}\bigg[\int_{t_n}^{t_{n+1}}y(s,X(s))\,\mathrm{d}s\bigg]=\int_{t_n}^{t_{n+1}}\mathbb{E}[y(s,X(s))]\,\mathrm{d}s.\label{FubiniTonelliResult}\end{equation}
Further to this, if $y$ has a bounded second moment, then for each $j=1,\ldots,m$,
\begin{equation}\mathbb{E}\bigg[\abs{\int_{t_n}^{t_{n+1}}y(s,X(s))\,\mathrm{d}W_j(s)}^2\bigg]=\int_{t_n}^{t_{n+1}}\mathbb{E}\left[\abs{y(s,X(s))}^2\right]\,\mathrm{d}s.\label{ItoIsometry}\end{equation}
\end{theorem}

The equality \cref{FubiniTonelliResult} is the Fubini--Tonelli Theorem, applied in the particular case of the measures being Lebesgue and probability measures. The second equality, \cref{ItoIsometry}, is the It\^o isometry. Proofs of these are available in standard texts on measure theory and probability, such as the text by Billingsley \cite{Billingsley1995}.

\begin{definition}
Suppose $Y=(Y_n)_{n=-p,\ldots,N}$ is a numerical approximation for the solution $X=(X(t))_{t\in[-\tau,T]}$ of \cref{Equation1}, with fixed step size $h=h_n=t_{n+1}-t_n$, for each $n=-p,\ldots,N-1$. The scheme $Y$ has (strong) \emph{order of convergence} (OoC) $q$ if there is a constant $c>0$ such that $\mathbb{E}\left[\abs{Y_n-X(t_n)}\right]\leqslant ch^q$, for all $n=-p,\ldots,N$. Similarly, $Y$ has \emph{mean-square order of convergence} (MS OoC) $q$ if there is a constant $c>0$ satisfying $\left(\mathbb{E}\left[\abs{Y_n-X(t_n)}^2\right]\right)^{1/2}\leqslant ch^q$, for all $n=-p,\ldots,N$.
\end{definition}

Generally, if $Y$ has MS OoC $q$ then it also has strong OoC $q$ (see \cite[page 5]{MilsteinTretyakov2021}), which can be shown with the H\"older inequality (for probability spaces). Another measurement of the accuracy for a numerical scheme is its weak error, which is the error in the moments of (smooth functions of) the scheme, but we do not study these errors in this paper. However, we do employ the following result by Milstein \cite{Milstein1995}, to prove the convergence orders for our schemes. This result relates the weak and mean-square errors of each single (local) step of a numerical scheme to its total (global) convergence order. That is, this result relates local errors of $Y_{n+1}$, conditional on the value of $Y_n$, to the global error in $Y$. Yang, Burrage, Komori, Burrage, and Ding \cite{BurrageBurrage2021} also use this result to prove the orders of convergence for their Magnus schemes \cref{MEMYnplus1} and \cref{MMYnplus1}. We later restate this result for SDDEs, before we then derive the MS OoC for each of our schemes.

\begin{theorem}\label{GeneralTheoremofMilstein}
Let $Y=(Y_n)_{n=0,1,\ldots,N}$ be a numerical scheme with fixed step size $h=h_n=t_{n+1}-t_n$, approximating the solution $X=(X(t))_{t\in[0,T]}$ of
\[\mathrm{d}X(t)=f(t,X(t))\,\mathrm{d}t+\sum_{j=1}^mg_j(t,X(t))\,\mathrm{d}W_j(t),\quad t\in[0,T],\]
where $f$ and each $g_j$ are spatially Lipschitz with linear-growth bound. Furthermore, suppose $X$ is adapted, $X(0)$ is independent of the Wiener processes, and $\mathbb{E}[\abs{X(0)}^2]<\infty$. If there is a constant $M>0$ such that
\begin{align*}
\abs{\mathbb{E}[Y_{n+1}|Y_n=x]-\mathbb{E}[X(t_{n+1})|X(t_n)=x]}&\leqslant M\left(1+\abs{x}^2\right)^{1/2}h^{q_1}\quad\textrm{and}\\
\left(\mathbb{E}\left[\abs{Y_{n+1}-X(t_{n+1})}^2|Y_n=X(t_n)=x\right]\right)^{1/2}&\leqslant M\left(1+\abs{x}^2\right)^{1/2}h^{q_2},
\end{align*}
for any $n=0,1,\ldots,N-1$ and $x\in\mathbb{R}^d$, where $q_2\geqslant 1/2$ and $q_1\geqslant q_2+1/2$, then
\begin{equation}\left(\mathbb{E}\left[\abs{Y_n-X(t_n)}^2|Y_0=X(0)\right]\right)^{1/2}\leqslant M\left(1+\mathbb{E}\left[\abs{X(0)}^2\right]\right)^{1/2}h^{q_2-1/2}\label{MSBound}\end{equation}
holds for every $n=0,1,\ldots,N$.
\end{theorem}

To construct Magnus–Taylor schemes for SDDEs, we combine Magnus methods for the linear homogeneous part with stochastic Taylor schemes for the remaining nonlinear terms, applying them piecewise on Bellman intervals between delay multiples.

\subsection{Combined schemes}\label{Section2.2}

Suppose $(t_n)_{n=-p,\ldots,N}$ is a time mesh that includes times $t_{p_k}=\tau_k$, for every $k=1,\ldots,K$, where $t_{p}=\tau$. In order to combine the Taylor and Magnus schemes for the SDDE \cref{Equation1}, we solve the equation between the multiples of the delay times. That is, we list the delay multiples $\tau_1,2\tau_1,\ldots,N_1\tau_1$, $\tau_2,2\tau_2,\ldots,N_2\tau_2$, $\ldots,$ $\tau_K,2\tau_K,\ldots,N_K\tau_K$, where
\[N_k=\inf\{n\in\mathbb{N}:n\tau_k\geqslant T\},\quad k=1,\ldots,K,\]
sort these $N_{\mathrm{total}}=\sum_{k=1}^KN_k$ numbers into ascending order, label these ordered delay multiples $\sigma_1,\allowbreak\sigma_2,\allowbreak\ldots,\allowbreak\sigma_{N_{\mathrm{total}}}$, and then solve \cref{Equation1} on the \emph{Bellman intervals} $[0,\sigma_1],\allowbreak[\sigma_1,\sigma_2],\allowbreak\ldots,\allowbreak[\sigma_{{N_{\mathrm{total}}}-1},\sigma_{N_{\mathrm{total}}}]$. Without loss of generality, we assume $T=\sigma_{N_{\mathrm{total}}}$ (or else we may assume an additional delay $\tau_{K+1}=T$ is included, and then remove any delay multiples above $T$).

We give further details of this derivation in the next section, but to define the Magnus--Taylor schemes for \cref{Equation1}, we apply Magnus approximations to the linear homogeneous part
\begin{equation}\mathrm{d}X_H(t)=A_0X_H(t)\,\mathrm{d}t+\sum_{j=1}^mA_jX_H(t)\,\mathrm{d}W_j(t),\label{LHEquation}\end{equation}
and we apply stochastic Taylor approximations to the inhomogeneous equation
\begin{align}\mathrm{d}X_R(t)&=f(t,X_R(t),X_R(t-\tau_1),\ldots,X_R(t-\tau_K))\,\mathrm{d}t\label{IHEquation}\\
&\quad+\sum_{j=1}^mg_j(t,X_R(t),X_R(t-\tau_1),\ldots,X_R(t-\tau_K))\,\mathrm{d}W_j(t).\nonumber\end{align}
The MEM scheme is produced by numerically solving \cref{LHEquation} with the first-order Magnus approximation \cref{MEMYnplus1} while the SDDE EM scheme \cref{EM} is applied to \cref{IHEquation}. Similarly, the MM scheme approximates \cref{LHEquation} with the second-order Magnus scheme \cref{MMYnplus1} while applying the SDDE Milstein scheme \cref{Milstein} to \cref{IHEquation}. These schemes are applied iteratively on each Bellman interval, to produce the following numerical schemes.


\begin{definition}\label{MEMandMMSDDESchemes}
Suppose $X=(X(t))_{t\in[-\tau,T]}$ is the unique strong solution of \cref{Equation1}, being approximated by the numerical scheme $Y=(Y_n)_{n=-p,\ldots,N}$, with $Y_n=Y(t_n)=\phi(t_n)$ when $t_n\leqslant0$. For each $n=0,\ldots,N-1$, define $t_n^{\tau_k}=t_n-\tau_k$, $Y_n^{\tau_k}=Y(t_n-\tau_k)$, and $Y_n^{\tau_l,\tau_k}=Y(t_n-\tau_l-\tau_k)$, for $k,l=1,\ldots,K$, along with step size $h_n=t_{n+1}-t_n$. Define the function
\[\tilde{f}(t,x,x_{\tau_1},\ldots,x_{\tau_K})=f(t,x,x_{\tau_1},\ldots,x_{\tau_K})-\sum_{j=1}^mA_jg_j(t,x,x_{\tau_1},\ldots,x_{\tau_K}),\]
for $(t,x,x_{\tau_1},\ldots,x_{\tau_K})\in[0,T]\times\mathbb{R}^d\times(\mathbb{R}^d)^K$. We call $Y$ the \emph{Magnus--Euler--Maruyama (MEM)} numerical approximation of \cref{Equation1} if it is given by first calculating the one-step, first-order Magnus approximation
\[M_n^{[1]}(t_n,t_{n+1})=\exp\bigg(\bigg(A_0-\frac{1}{2}\sum_{i=1}^mA_i^2\bigg)\,h_n+\sum_{j=1}^mA_j\,\Delta W_j(t_n,t_{n+1})\bigg),\]
and this is then used to evaluate
\begin{align}Y_{n+1}=M_n^{[1]}(t_n,t_{n+1})\cdot\bigg\{&Y_n+\tilde{f}(t_n,Y_n,Y_n^{\tau_1},\ldots,Y_n^{\tau_K})\,h_n\label{SDDEMEM}\\
&+\sum_{j=1}^mg_j(t_n,Y_n,Y_n^{\tau_1},\ldots,Y_n^{\tau_K})\,\Delta W_j(t_n,t_{n+1})\bigg\},\nonumber\end{align}
for $n=0,\ldots,N-1$. Similarly, $Y$ is the \emph{Magnus--Milstein (MM)} numerical approximation if it is given by finding the one-step, second-order Magnus approximation
\begin{align*}
M_n^{[2]}(t_n,t_{n+1})=\exp\bigg(&\bigg(A_0-\frac{1}{2}\sum_{i=1}^mA_i^2\bigg)\,h_n+\sum_{j=1}^mA_j\,\Delta W_j(t_n,t_{n+1})\\
&+\frac{1}{2}\sum_{i=0}^m\sum_{j=i+1}^m[A_i,A_j](I_{ji}(t_n,t_{n+1})-I_{ij}(t_n,t_{n+1}))\bigg),
\end{align*}
finding the one-step, second-order Taylor approximation
\begin{align}
\tilde{Y}_n=\bigg\{Y_n&+\tilde{f}(t_n,Y_n,Y_n^{\tau_1},\ldots,Y_n^{\tau_K})\,h_n+\sum_{j=1}^mg_j(t_n,Y_n,Y_n^{\tau_1},\ldots,Y_n^{\tau_K})\,\Delta W_j(t_n,t_{n+1})\nonumber\\
&+\sum_{i=1}^m\sum_{j=1}^m\left(\vphantom{\sum}\right.\nabla_xg_j(t_n,Y_n,Y_n^{\tau_1},\ldots,Y_n^{\tau_K})\cdot\left[A_iY_n+g_i(t_n,Y_n,Y_n^{\tau_1},\ldots,Y_n^{\tau_K})\right]\nonumber\\
&\phantom{+\sum_{i=1}^m\sum_{j=1}^m\quad}\hspace{2.5cm}-A_ig_j(t_n,Y_n,Y_n^{\tau_1},\ldots,Y_n^{\tau_K})\left.\vphantom{\sum}\right)\,I_{ij}(t_n,t_{n+1})\nonumber\\
&+\sum_{k=1}^K\mathbb{I}(t_n\geqslant\tau_k)\sum_{i=1}^m\sum_{j=1}^m\bigg(\nabla_{x_{\tau_k}}g_j(t_n,Y_n,Y_n^{\tau_1},\ldots,Y_n^{\tau_K})\nonumber\\
&\phantom{+\sum_{k=1}^K\mathbb{I}(t_n\geqslant\tau_k)}\!\!\cdot\left[A_iY_n^{\tau_k}+g_i(t_n^{\tau_k},Y_n^{\tau_k},Y_n^{\tau_1,\tau_k},\ldots,Y_n^{\tau_K,\tau_k})\right]\bigg)\,I_{ij}^{\tau_k}(t_n,t_{n+1})\bigg\},\nonumber
\end{align}
and then using this to evaluate
\begin{equation}
Y_{n+1}=M_n^{[2]}(t_n,t_{n+1})\cdot\tilde{Y}_n,\label{SDDEMM}
\end{equation}
for each $n=0,\ldots,N-1$.
\end{definition}

\section{Derivation of schemes}\label{Section3}

Here we derive the MEM and MM schemes, primarily focusing on the MM scheme as the MEM scheme is simpler to express. We give this derivation directly, applying numerous inequalities to bound the schemes. We then retain only the terms whose mean-square expectations are multiples of the step size $h$ or $h^{1/2}$, while omitting all higher-order terms. In doing this, we aim to present an intuitive derivation of the MEM and MM schemes. An alternative, concise derivation for the case of no delays is available in Yang, Burrage, Komori, Burrage, and Ding \cite{BurrageBurrage2021}, who use the It\^o chain rule (also known as integration by parts for It\^o SDEs).

We consider equation \cref{Equation1} on each of the Bellman intervals $[\sigma_0,\sigma_1],[\sigma_1,\sigma_2],\ldots,\allowbreak[\sigma_{{N_{\mathrm{total}}}-1},\sigma_{N_{\mathrm{total}}}]$, where $\sigma_0=0$. On each fixed interval $[\sigma_b,\sigma_{b+1}]$ (for some $b=0,\ldots,{N_{\mathrm{total}}}-1$), we apply the decomposition by Yang et al.\ \cite{BurrageBurrage2021}, writing the solution $X$ as the product $X(t)=X_\mathrm{H}(t)X_\mathrm{R}(t)$, for solution $X_\mathrm{H}$ of the homogeneous $d\times d$ matrix SDE
\begin{align}
\mathrm{d}X_\mathrm{H}(t)&=A_0X_\mathrm{H}(t)\,\mathrm{d}t+\sum_{j=1}^mA_jX_\mathrm{H}(t)\,\mathrm{d}W_j(t),\quad t\in[\sigma_b,\sigma_{b+1}],\label{dXHd}
\end{align}
where $X_\mathrm{H}(\sigma_b)=I_d$. \Cref{dXHd} is \cref{dXH} applied on the present Bellman interval rather than on $[0,T]$. In the interval $[\sigma_b,\sigma_{b+1}]$, the past solution has been determined from \cref{Equation1} on prior Bellman intervals. Letting $\phi_{\sigma_b}:[0,\sigma_b]\to(\mathbb{R}^d)^K$ denote the collection of the $K$ delayed solutions as a function of $s$, given by $\phi_{\sigma_b}(s)=(X(s-\tau_1),\ldots,X(s-\tau_K))$, it follows from \cref{XR} that the remaining process $X_\mathrm{R}$ is given by
\begin{align}
X_\mathrm{R}(t)&=X(\sigma_b)+\int_{\sigma_b}^tX_\mathrm{H}^{-1}(s)\tilde{f}(s,X(s),\phi_{\sigma_b}(s))\,\mathrm{d}s\label{XRd}\\
&\quad+\sum_{j=1}^m\int_{\sigma_b}^tX_\mathrm{H}^{-1}(s)g_j(s,X(s),\phi_{\sigma_b}(s))\,\mathrm{d}W_j(s),\quad t\in[\sigma_b,\sigma_{b+1}].\nonumber
\end{align}

\subsection{Magnus approximation for homogeneous part}

Burrage and Burrage \cite{BurrageBurrage1999} derive Magnus expansions for Stratonovich SDEs, while Kamm, Pagliarani, and Pascucci \cite{Kamm2021} extend this to the It\^o setting. We use these expansions to express $X_H(t)=\exp(\Omega(\sigma_b,t))$, where
\begin{align}
\Omega(\sigma_b,t)&=\bigg(A_0-\frac{1}{2}\sum_{j=1}^mA_j^2\bigg)\,(t-\sigma_b)+\sum_{j=1}^mA_j\,\Delta W_j(\sigma_b,t)\label{Omega}\\
&\quad+\frac{1}{2}\sum_{i=0}^m\sum_{j=i+1}^m[A_i,A_j](I_{ji}(\sigma_b,t)-I_{ij}(\sigma_b,t))+\cdots.\nonumber
\end{align}
In defining the MEM scheme, \cref{Omega} is truncated and approximated by
\begin{equation}
\Omega^{[1]}(\sigma_b,t)=\bigg(A_0-\frac{1}{2}\sum_{j=1}^mA_j^2\bigg)\,(t-\sigma_b)+\sum_{j=1}^mA_j\,\Delta W_j(\sigma_b,t),\label{Omega1}
\end{equation}
while in the MM scheme, \cref{Omega} is approximated by
\begin{align}
\Omega^{[2]}(\sigma_b,t)&=\bigg(A_0-\frac{1}{2}\sum_{j=1}^mA_j^2\bigg)\,(t-\sigma_b)+\sum_{j=1}^mA_j\,\Delta W_j(\sigma_b,t)\label{Omega2}\\
&\quad+\frac{1}{2}\sum_{i=0}^m\sum_{j=i+1}^m[A_i,A_j](I_{ji}(\sigma_b,t)-I_{ij}(\sigma_b,t)).\nonumber
\end{align}
All unseen terms of \cref{Omega} are multiples of triple and higher-order integrals. The general $q^\mathrm{th}$-order terms of \cref{Omega} may be found by applying the iterative derivation of $\Omega^{[q]}\approx\Omega$ by Kamm et al.\ \cite{Kamm2021}. We also note that $X_H^{-1}(t)=\exp(-\Omega(\sigma_b,t))$, since this satisfies $X_H^{-1}(t)X_H(t)=I_d=X_H(t)X_H^{-1}(t)$.


\subsection{Taylor approximation for inhomogeneous part}

We make use of the Taylor theorem for the SDDE \cref{Equation1}, so we now give this (\cref{ItoTaylorKuchlerPlaten}, below). This theorem is given by Kloeden and Platen \cite{KloedenPlaten1999} in the case of no delay, but extended to the single-delay case by K\"uchler and Platen \cite{KuchlerPlaten2000}. We give proof of the extension to $K$ delays, in the supplementary material \cref{ProofofItoTaylorKuchlerPlaten}. We also require expression for the differential $\mathrm{d}X_H^{-1}$, as well as several inequalities, in order to derive the Magnus schemes, so we give these results in \cref{AppendixProofsforConvergenceAnalysis}.

\begin{theorem}\label{ItoTaylorKuchlerPlaten}
Let $(X(t))_{t\in[-\tau,T]}$ be the strong solution of \cref{Equation5} and let $\sigma_b\leqslant t_n\leqslant t_{n+1}\leqslant \sigma_{b+1}$ be times in the Bellman interval $[\sigma_b,\sigma_{b+1}]$, for given $b=0,\ldots,{N_{\mathrm{total}}}-1$, and suppose $h_n=t_{n+1}-t_n<\tau_\mathrm{min}$. It follows that
\begin{align}
X(t_{n+1})&=X(t_n)+f(t_n,X(t_n),X(t_n^{\tau_1}),\ldots,X(t_n^{\tau_K}))\,h_n\nonumber\\
&\quad+\sum_{j=1}^mg_j(t_n,X(t_n),X(t_n^{\tau_1}),\ldots,X(t_n^{\tau_K}))\,\Delta W_j(t_n,t_{n+1})\label{DelayTaylorExpansion}\\
&\quad+\sum_{i=1}^m\sum_{j=1}^m\bigg(\nabla_xg_j(t_n,X(t_n),X(t_n^{\tau_1}),\ldots,X(t_n^{\tau_K}))\nonumber\\
&\quad\quad\hspace{1cm}\cdot g_i(t_n,X(t_n),X(t_n^{\tau_1}),\ldots,X(t_n^{\tau_K}))\bigg)\,I_{ij}(t_n,t_{n+1})\nonumber\\
&\quad+\sum_{k=1}^K\mathbb{I}(t_n\geqslant\tau_k)\sum_{i=1}^m\sum_{j=1}^m\bigg(\nabla_{x_{\tau_k}}g_j(t_n,X(t_n),X(t_n^{\tau_1}),\ldots,X(t_n^{\tau_K}))\nonumber\\
&\quad\quad\hspace{1cm}\cdot g_i(t_n^{\tau_k},X(t_n^{\tau_k}),X(t_n^{\tau_1,\tau_k}),\ldots,X(t_n^{\tau_K,\tau_k}))\bigg)\,I_{ij}^{\tau_k}(t_n,t_{n+1})+R,\nonumber
\end{align}
where $t_n^{\tau_k}=t_n-\tau_k$ and $t_n^{\tau_l,\tau_k}=t_n-\tau_l-\tau_k$ for $k,l=1,\ldots,m$, while $R$ contains only terms that are multiples of integrals which have three terms or more, and satisfy the bound $\mathbb{E}[\abs{R}^2]<ch_n^{3/2}$ for constant $c>0$.
\end{theorem}

Returning to the expansion $X(t)=X_H(t)X_R(t)$, for $t\in[\sigma_b,\sigma_{b+1}]$, $X_R(t)$ is simulated using \cref{XRd}, by applying a quadrature approximation to the Lebesgue integral while also using stochastic Taylor approximations for the It\^o integrals. In particular, we use approximations for the integral sum
\begin{equation}
I(\sigma_b,t)=\sum_{j=1}^m\int_{\sigma_b}^tX_H^{-1}(s)g_j(s,X(s),\phi_{\sigma_b}(s))\,\mathrm{d}W_j(s).\label{I}
\end{equation}
Letting $t_\alpha$ and $t_\beta$ be mesh points where $\sigma_b=t_\alpha\leqslant t_n<t_{n+1}\leqslant t_\beta=\sigma_{b+1}$, we suppose we are simulating $X$ at the time point $t=t_{n+1}$. The Euler--Maruyama approximation of \cref{I} is
\begin{align*}
I(\sigma_b,t_{n+1})&=\sum_{j=1}^m\sum_{l=\alpha}^n\int_{t_l}^{t_{l+1}}\exp(-\Omega(t_l,s))g_j(s,X(s),\phi_{\sigma_b}(s))\,\mathrm{d}W_j(s)\\
&\approx\sum_{j=1}^m\sum_{l=\alpha}^ng_j(t_l,X(t_l),\phi_{\sigma_b}(t_l))\,\Delta W_j(t_l,t_{l+1}),
\end{align*}
which is used to produce the MEM scheme. In order to produce the MM scheme, we use the Milstein approximation of \cref{I}. Let $\underline{A}_0=\left(-A_0+\sum_{j=1}^mA_j^2\right)$ and $\underline{A}_j=-A_j$, for each $j=1,\ldots,m$, to simplify notation. We calculate
\begin{align}
I(\sigma_b,t_{n+1})&=\sum_{j=1}^m\sum_{l=\alpha}^n\int_{t_l}^{t_{l+1}}\exp(-\Omega(t_l,s))g_j(s,X(s),\phi_{\sigma_b}(s))\,\mathrm{d}W_j(s)\nonumber\\
&=\sum_{j=1}^m\sum_{l=\alpha}^n\int_{t_l}^{t_{l+1}}\bigg(I_d+\int_{t_l}^s\underline{A}_0Z_{t_l}(u)\,\mathrm{d}u+\sum_{i=1}^m\int_{t_l}^s\underline{A}_iZ_{t_l}(u)\,\mathrm{d}W_i(u)\bigg)\label{Isum}\\
&\quad\cdot g_j(s,X(s),\phi_{\sigma_b}(s))\,\mathrm{d}W_j(s)\nonumber\\
&=\sum_{j=1}^m\sum_{l=\alpha}^n\bigg\{\int_{t_l}^{t_{l+1}}g_j(s,X(s),\phi_{\sigma_b}(s))\,\mathrm{d}W_j(s)\nonumber\\
&\quad+\int_{t_l}^{t_{l+1}}\bigg(\int_{t_l}^s\underline{A}_0Z_{t_l}(u)\,\mathrm{d}u\bigg)\cdot g_j(s,X(s),\phi_{\sigma_b}(s))\,\mathrm{d}W_j(s)\nonumber\\
&\quad+\int_{t_l}^{t_{l+1}}\bigg(\sum_{i=1}^m\int_{t_l}^s-A_iZ_{t_l}(u)\,\mathrm{d}W_i(u)\bigg)\cdot g_j(s,X(s),\phi_{\sigma_b}(s))\,\mathrm{d}W_j(s)\bigg\}\nonumber\\
&=\sum_{l=\alpha}^n\left(I^{(1)}(t_l,t_{l+1})+I^{(2)}(t_l,t_{l+1})+I^{(3)}(t_l,t_{l+1})\right),\nonumber
\end{align}
for a matrix-valued geometric Brownain motion $Z_{t_l}$ (see \cref{LemmaPhiInverse}), and integral terms $I^{(1)}(t_l,t_{l+1})$, $I^{(2)}(t_l,t_{l+1})$, and $I^{(3)}(t_l,t_{l+1})$. We further calculate these integrals (see \cref{NewLemma2}, \cref{I1,I2,I3}) to find
\begin{align}
I(\sigma_b,t_{n+1})&=\sum_{l=\alpha}^n\bigg(\sum_{j=1}^mg_j(t_l,X(t_l),\phi_{\sigma_b}(t_l))\,\Delta W_j\nonumber\\
&\quad\quad+\sum_{i=1}^m\sum_{j=1}^m\bigg(\nabla_xg_j(t_l,X(t_l),\phi_{\sigma_b}(t_l))\nonumber\\
&\quad\quad\quad\hspace{1cm}\cdot[A_iX(t_l)+g_i(t_l,X(t_l),\phi_{\sigma_b}(t_l))]\bigg)\,I_{ij}(t_l,t_{l+1})\nonumber\\
&\quad\quad+\sum_{k=1}^K\mathbb{I}(t_l\geqslant\tau_k)\sum_{i=1}^m\sum_{j=1}^m\bigg(\nabla_{x_{\tau_k}}g_j(t_l,X(t_l),\phi_{\sigma_b}(t_l))\nonumber\\
&\quad\quad\quad\hspace{1cm}\cdot[A_iX(t_l^{\tau_k})+g_i(t_l^{\tau_l},X(t_l^{\tau_l}),\phi_{\sigma_b}(t_l^{\tau_l}))]\bigg)\,I_{ij}^{\tau_k}(t_l,t_{l+1})\nonumber\\
&\quad\quad-\sum_{j=1}^m\sum_{i=1}^mA_iZ(t_l)g_j(t_l,X(t_l),\phi_{\sigma_b}(t_l))\,I_{ij}(t_l,t_{l+1})+R\bigg),\nonumber
\end{align}
where $\mathbb{E}[\abs{R}^2]\leqslant c(t_{l+1}-t_l)^{3/2}$, for some constant $c>0$.

\subsection{Magnus--Taylor expansion}

From the above, the Magnus--Taylor expansion for the solution of \cref{Equation1} is
\begin{align}
X(t)&=\exp(\Omega(\sigma_b,t))\bigg\{R+X(\sigma_b)+\sum_{l=\alpha}^n\int_{t_l}^{t_{l+1}}X_\mathrm{H}^{-1}(s)\tilde{f}(s,X(s),\phi_{\sigma_b}(s))\,\mathrm{d}s\nonumber\\
&\quad+\sum_{l=\alpha}^n\sum_{j=1}^mg_j(t_l,X(t_l),\phi_{\sigma_b}(t_l))\,\Delta W_j\label{MTExpansion}\\
&\quad+\sum_{l=\alpha}^n\sum_{i=1}^m\sum_{j=1}^m\left(\vphantom{\sum}\right.\nabla_xg_j(t_l,X(t_l),\phi_{\sigma_b}(t_l))\cdot[A_iX(t_l)+g_i(t_l,X(t_l),\phi_{\sigma_b}(t_l))]\nonumber\\
&\quad\quad\hspace{3.75cm}-A_iZ(t_l)g_j(t_l,X(t_l),\phi_{\sigma_b}(t_l))\left.\vphantom{\sum}\right)\,I_{ij}(t_l,t_{l+1})\nonumber\\
&\quad+\sum_{l=\alpha}^n\sum_{k=1}^K\mathbb{I}(t_l\geqslant\tau_k)\sum_{i=1}^m\sum_{j=1}^m\bigg(\nabla_{x_{\tau_k}}g_j(t_l,X(t_l),\phi_{\sigma_b}(t_l))\nonumber\\
&\quad\quad\hspace{2.75cm}\cdot[A_iX(t_l^{\tau_k})+g_i(t_l^{\tau_l},X(t_l^{\tau_l}),\phi_{\sigma_b}(t_l^{\tau_l}))]\bigg)\,I_{ij}^{\tau_k}(t_l,t_{l+1})\bigg\},\nonumber
\end{align}
where $R$ consists of matrix multiples (formed from commutators) of third-degree (and higher) iterated integrals and there is a constant $c>0$ such that $\mathbb{E}[\abs{R}^2]<c(t-\sigma_b)^{3/2}$, for all $t\in[\sigma_b,\sigma_{b+1}]$. From \cref{MTExpansion}, the MEM scheme \cref{SDDEMEM} is found by truncating all factors of second-degree (or higher) integrals, and retaining only the multiples of the first-degree integrals. Similarly, the MM scheme \cref{SDDEMM} is found by omitting all multiples of third-degree (or higher) integrals.

\subsection{Convergence analysis}

We now give a proof that the MEM scheme has MS OoC $1/2$, while the MM scheme has MS OoC $1$. These proofs apply \cref{GeneralTheoremofMilstein} on each Bellman interval $[\sigma_b,\sigma_{b+1}]$, for $b=0,\ldots,{N_{\mathrm{total}}}-1$. Using the notation of \cref{Section3}, we first rephrase \cref{GeneralTheoremofMilstein} in terms of the SDDE \cref{Equation1}. We also use the shorthand notation $y(t)=y(t,X(t),\phi_{\sigma_b}(t))$, for each of $y=f,\tilde{f},g_j$, $j=1,\ldots,m$.

\begin{theorem}\label{GeneralTheoremofMilsteinforSDDEs}
Let $Y=(Y_n)_{n=-p,\ldots,N}$ be a numerical scheme with fixed step size $h=t_{n+1}-t_n$, approximating the solution $X=(X(t))_{t\in[-\tau,T]}$ of
\cref{Equation1}, and suppose each delay $\tau_k=t_{p_k}$ for some mesh time $t_{p_k}$. If there are constants $M>0$, $q_2\geqslant 1/2$, and $q_1\geqslant q_2+1/2$, such that
\begin{align}
\abs{\mathbb{E}[Y_{n+1}|Y_n=x]-\mathbb{E}[X(t_{n+1})|X(t_n)=x]}&\leqslant M\left(1+\abs{x}^2\right)^{1/2}h^{q_1}\quad\textrm{and}\label{WeakBound}\\
\left(\mathbb{E}\left[\abs{Y_{n+1}-X(t_{n+1})}^2|Y_n=X(t_n)=x\right]\right)^{1/2}&\leqslant M\left(1+\abs{x}^2\right)^{1/2}h^{q_2},\label{MSPart}
\end{align}
for any $n=-p,\ldots,N-1$ and $x\in\mathbb{R}^d$, then
\begin{equation}\left(\mathbb{E}\left[\abs{Y_n-X(t_n)}^2|Y_0=X(0)\right]\right)^{1/2}\leqslant M\left(1+\mathbb{E}\left[\abs{X(0)}^2\right]\right)^{1/2}h^{q_2-1/2}\label{MSBoundSDDE}\end{equation}
holds for every $n=-p,\ldots,N$.
\end{theorem}

For each of the MEM and MM schemes, in order to prove the bounds \cref{WeakBound} and \cref{MSPart}, we express the solution $X$ on $[\sigma_b,\sigma_{b+1}]$, using the decomposition $X(t)=X_H(t)X_R(t)$, with $X_H$ and $X_R$ satisfying \cref{dXHd} and \cref{XRd}, respectively. This involves numerous calculations, which we show in the supplementary material \cref{ProofofMEMConditions,ProofofMMConditions}, and summarise in \Cref{NewLemma2}.

\begin{lemma}\label{MEMConditions}
The weak bound \cref{WeakBound} holds for the MEM scheme, with $q_1=2$, and the MS bound \cref{MSPart} holds for the MEM scheme, with $q_2=1$.
\end{lemma}

\begin{lemma}\label{MMConditions}
The weak bound \cref{WeakBound} holds for the MM scheme, with $q_1=2$, and the MS bound \cref{MSPart} holds for the MM scheme, with $q_2=1.5$.
\end{lemma}

By combining \cref{MEMConditions} and \cref{MMConditions} with \cref{GeneralTheoremofMilsteinforSDDEs}, we establish the following OoC result for the MEM and MM schemes.

\begin{theorem}\label{OoCTheorem}
The MEM scheme, for the semilinear SDDE \cref{Equation1}, has mean-square order of convergence $1/2$, and the MM scheme, for the semilinear SDDE \cref{Equation1}, has mean-square order of convergence $1$.
\end{theorem}

%
%
%
%
%

\subsection{Simulating second-degree It\^o integrals}

In order to simulate the numerical schemes \cref{Milstein} and \cref{SDDEMM}, we require methods to approximate the integrals \cref{IijandIijtauk}. We first consider the non-delayed integral $I_{ij}(t_n,t_{n+1})$, and investigate the delayed case, below. In the case $i=j$, we apply the well-known identity
\[I_{jj}(t_n,t_{n+1})=\frac{1}{2}\left[(\Delta W_j(t_n,t_{n+1}))^2-h_n\right],\]
simulating $I_{jj}(t_n,t_{n+1})$ with the value $\Delta W_j(t_n,t_{n+1})$. When $i\neq j$, then $I_{ij}(t_n,t_{n+1})$ is more complicated to express. One method to simulate $I_{ij}(t_n,t_{n+1})$, developed by Kloeden, Platen, and Wright \cite{KloedenPlatenWright1992}, consists of expanding the inner process as a Fourier series, truncated to finitely many terms. Wiktorsson \cite{Wiktorsson2001} improves this Fourier expansion by giving expression for the remainder term, rather than truncating the Fourier series. An improvement on this remainder term is given by Mrongowius and R\"{o}{\ss}ler \cite{MrongowiusRosler2021}. Alternative expansions are given by Kuznetsov \cite{Kuznetsov2018,Kuznetsov2019}, using different Fourier coefficients.

One simple method to approximate $I_{ij}(t_n,t_{n+1})$ is the trapezium method by Milstein \cite{Milstein1995} (later called the trapezoidal method by Milstein and Tretyakov \cite{MilsteinTretyakov2021}). This method approximates $\tilde{I}_{ij}\approx I_{ij}(t_n,t_{n+1})$ with
\begin{align}
\tilde{I}_{ij}&=\sum_{l=0}^{F_n-1}\frac{\Delta W_i\left(t_n^{(l)},t_n^{(l+1)}\right)\,\Delta W_j\left(t_n^{(l)},t_n^{(l+1)}\right)}{2}\label{RefinedIij}\\
&\quad+\sum_{l=0}^{F_n-1}\Delta W_i\left(t_n^{(l)},t_n^{(l+1)}\right)\,\Delta W_j\left(t_n^{(l+1)},t_n^{(F_n)}\right),\nonumber
\end{align}
dividing $[t_n,t_{n+1}]$ into $F_n$ subintervals of common length, with intermittent time points $t_n^{(l)}=t_n+lh/F_n$ for $l=0,1,\ldots,F_n$.

\begin{remark}\label{SimpleRefinedRemark}
Removing the first summation from the approximation \cref{RefinedIij} produces the rectangle approximation, which is the simplest approximation of $I_{ij}(t_n,\!t_{n+1})$ that maintains the theoretical OoC (1) of the Milstein scheme, as described in \cite[\textbf{Theorem 1.5.4}]{MilsteinTretyakov2021}.
\end{remark}

In the case of the delayed integral, the problem of expressing $I_{ij}^{\tau_k}(t_n,t_{n+1})$ is similar to expressing $I_{ij}(t_n,t_{n+1})$, but with the delayed process values $W_i^{\tau_k}(u)=W_i(u-\tau_k)$ replacing $W_i(u)$, for $u\in[t_n,t_{n+1}]$. Immediately, $W_i^{\tau_k}$ and $W_j$ are independent if $i\neq j$. However, when $i=j$, if the step sizes of the numerical scheme are no larger than the smallest delay value, then the processes $W_j^{\tau_k}$ and $W_j$ are independent on the intervals $[t_n-\tau_k,t_{n+1}-\tau_k]$ and $[t_n,t_{n+1}]$. That is, by simulating the schemes with the condition
\[\max_{n=0,\ldots,N}\{h_n\}<\min_{k=1,\ldots,K}\{\tau_k\},\]
then we may treat every instance of $I_{ij}^{\tau_k}(t_n,t_{n+1})$ similarly to $I_{ij}(t_n,t_{n+1})$, and use similar approximation methods for both.

In this paper, we simulate the double integrals \cref{IijandIijtauk} by approximating both $\tilde{I}_{ij}\approx I_{ij}(t_n,t_{n+1})$ with \cref{RefinedIij} and $\tilde{I}_{ij}^{\tau_k}\approx I_{ij}^{\tau_k}(t_n,t_{n+1})$ with
\begin{align}
\tilde{I}_{ij}^{\tau_k}  &= \sum_{l=0}^{F_n-1}\frac{\Delta W_i^{\tau_k}\left(t_n^{(l)},t_n^{(l+1)}\right)\,\Delta W_j\left(t_n^{(l)},t_n^{(l+1)}\right)}{2}\label{RefinedIijtau}\\
&\quad+\sum_{l=0}^{F_n-1}\Delta W_i^{\tau_k}\left(t_n^{(l)},t_n^{(l+1)}\right)\,\Delta W_j\left(t_n^{(l+1)},t_n^{(F_n)}\right),\quad k=1,\ldots,K,\nonumber
\end{align}
where $\Delta W_i^{\tau_k}(s,t)=W_i(t-\tau_k)-W_i(s-\tau_k)$ denotes a delayed Wiener increment, for $i=1,\ldots,m$. As we noted in \hyperref[SimpleRefinedRemark]{Remark~\ref*{SimpleRefinedRemark}}, simulating the integrals \cref{IijandIijtauk} with \cref{RefinedIij} and \cref{RefinedIijtau} maintains the theoretical OoC of the Milstein scheme \cref{Milstein}. A direct proof that this Milstein SDDE scheme has OoC $1$ is also given by Cao, Zhang, and Karniadakis \cite{CaoZhangKarniadakis2015}. The requirement to simulate the Wiener processes at the intermediate time points $t_n^{(l)}$ becomes computationally cumbersome as the number $m$ of independent Wiener processes increases. Unfortunately, however, using more efficient expressions for the integrals reduces the OoC of the Milstein scheme. For example, Hofmann and M\"uller-Gronbach \cite{Hofmann2006} simulate \cref{Milstein} with the approximations
\begin{align}I_{ij}(t_n,t_{n+1})&\approx\frac{\Delta W_i(t_n,t_{n+1})\,\Delta W_j(t_n,t_{n+1})}{2}\quad\textrm{and}\label{SimpleApproximations}\\
I_{ij}^{\tau_k}(t_n,t_{n+1})&\approx\frac{\Delta W_i^{\tau_k}(t_n,t_{n+1})\,\Delta W_j(t_n,t_{n+1})}{2},\nonumber\end{align}
which results in a numerical scheme of OoC $1/2$. This loss of accuracy, using the Riemann approximations \cref{SimpleApproximations}, is caused by the unbounded variation of the inner process $\int_{t_n}^s\,\mathrm{d}W_i(u)=W_i(s)-W_i(t_n)$.


\section{Examples}\label{Section4} In each of the following examples, we simulate \cref{Equation1} with our numerical schemes, and present the error graphs showing the mean-square errors, simulated with $n_\mathrm{t}$ trials. Specifically, at time $T$, we approximate the solution $X(T)$ using a reference solution with a small step size $h_\mathrm{R}$, along with the scheme approximation $Y_N=Y(t_N)$ simulated with a corresponding step size $h$, for each scheme. In every example, we simulate the Wiener processes and the reference solution on a fine time mesh of the form $(nh_\mathrm{R})_{n=0,\ldots,T/h_\mathrm{R}}$, while the scheme mesh is of the form $(nh)_{n=0,\ldots,T/h}$, where $N=T/h$. Every step size $h$ is an integer multiple of $h_\mathrm{R}$. We use the reference step size to define the number $F_n=h/h_\mathrm{R}$ of subintervals used to simulate the integral approximations \cref{RefinedIij} and \cref{RefinedIijtau}. \hyperref[Example1]{Example~\ref*{Example1}} demonstrates the case of a linear SDDE with one delay, \hyperref[Example2]{Example~\ref*{Example2}} includes a semilinear SDDE with one delay, \hyperref[Example3]{Example~\ref*{Example3}} uses a semilinear SDDE with two delays, and \hyperref[Example4]{Example~\ref*{Example4}} includes an SPDDE.

Letting $X^{(i)}(T)$ and $Y^{(i)}_N$ respectively denote the reference solution and numerical scheme for the $i^\mathrm{th}$ simulation of \cref{Equation1}, the \emph{mean-square error} (MSE) is calculated as
\begin{equation}
\mathrm{MSE}(T)=\left(\mathbb{E}\abs{Y_N-X(T)}^2\right)^{1/2}=\bigg(\frac{1}{n_\mathrm{t}}\sum_{i=1}^{n_\mathrm{t}}\abs{Y^{(i)}_N-X^{(i)}(T)}^2\bigg)^{1/2}.
\end{equation}
In the error graphs, we show the mean-square convergence orders of the EM, Milstein, MEM, and MM schemes, along with reference slopes of orders of convergence $1/2$ and $1$.

\begin{example}\label{Example1}
This example demonstrates the case of the single-delay ($K=1$), linear (in $X(t)$) \cref{Equation1}, where each function $g_j$ does not depend on $X(t)$. In this example, we select spatial dimension $d=2$, delay time $\tau_1=1$, terminal time $T=6$, $m=2$ Wiener processes, and constant initial process $\phi(t)=(0.8,\,0.2)^\intercal$ for all $t\in[-1,0]$. We also select parameters
\begin{align}
A_0&=\begin{pmatrix}
         -0.1 & 0.4 \\
         -0.3 & 0.2 \\
       \end{pmatrix},&&g_0\left(t,\begin{pmatrix} x_1 \\ x_2 \\ \end{pmatrix},\begin{pmatrix}
           y_1 \\
           y_2 \\
         \end{pmatrix}
       \right)=\dfrac{1}{10}\begin{pmatrix}
                        \cos(y_1+y_2) \\
                        y_2-y_1^2 \\
                      \end{pmatrix},\label{Example1Parameters}\\
A_1&=\begin{pmatrix}
           0.3 & 0.1 \\
           0 & 0.2 \\
         \end{pmatrix},&&g_1\left(t,\begin{pmatrix} x_1 \\ x_2 \\ \end{pmatrix},\begin{pmatrix}
           y_1 \\
           y_2 \\
         \end{pmatrix}
       \right)=\dfrac{1}{3}\begin{pmatrix}
                        \sin(y_1)+\exp\left(-y_2^2\right) \\
                        \arctan(y_1)+\cos(y_2) \\
                      \end{pmatrix},\nonumber\\
A_2&=\begin{pmatrix}
             0.1 & 0 \\
             0.3 & 0.1 \\
           \end{pmatrix},&&g_2\left(t,\begin{pmatrix} x_1 \\ x_2 \\ \end{pmatrix},\begin{pmatrix}
           y_1 \\
           y_2 \\
         \end{pmatrix}
       \right)=\begin{pmatrix}
                     0.18 & 0.04 \\
                     0.21 & 0.03 \\
                   \end{pmatrix}\begin{pmatrix}
           y_1 \\
           \arctan(y_2) \\
         \end{pmatrix},\nonumber
\end{align}
for \cref{Equation1}. The reference solution is formed with the Milstein scheme of step size $h_\mathrm{R}=2^{-14}$ and the numerical schemes are simulated using step sizes $h=2^0,2^{-1},\ldots,2^{-10}$, with the resulting error graphs shown in \Cref{A} (a), using $n_\mathrm{t}=5,000$ trials.
\end{example}

\begin{example}\label{Example2} We now simulate the Magnus schemes for a semilinear \cref{Equation1}, where the values for $K,d,\tau,T,m,A_0,A_1,A_2$, and $\phi(t)$, are as in \hyperref[Example1]{Example~\ref*{Example1}}, but with nonlinear functions $g_0,g_1,g_2$, given by
\begin{align}
g_0\left(t,\begin{pmatrix}x_1\\x_2\\ \end{pmatrix},\begin{pmatrix}y_1\\y_2\\ \end{pmatrix}\right)&=\frac{1}{3}\begin{pmatrix}\cos(x_1+x_2+y_1+y_2) \\ \sin(x_1+x_2+y_1+y_2) \\ \end{pmatrix},\label{Example2Parameters}\\
g_1\left(t,\begin{pmatrix}x_1\\x_2\\ \end{pmatrix},\begin{pmatrix}y_1\\y_2\\ \end{pmatrix}\right)&=\frac{1}{9}\begin{pmatrix}\cos(x_1) \\ \sin(x_2) \\ \end{pmatrix}+\frac{1}{5}\begin{pmatrix}\sin(y_1)+\exp\left(-y_2^2\right) \\ \arctan(y_1)+\cos(y_2) \\ \end{pmatrix},\nonumber\\
g_2\left(t,\begin{pmatrix}x_1\\x_2\\ \end{pmatrix},\begin{pmatrix}y_1\\y_2\\ \end{pmatrix}\right)&=\frac{1}{7}\begin{pmatrix}\sin(x_2) \\ \cos(x_1) \\ \end{pmatrix}+\begin{pmatrix}0.04 & 0.05 \\0.06 & 0.04 \\ \end{pmatrix}\begin{pmatrix}1/\left(1+y_1^2\right) \\ \arctan(y_2) \\ \end{pmatrix}.\nonumber
\end{align}
We again simulate the reference solution with the Milstein scheme, using step size $h_\mathrm{R}=2^{-14}$, along with the numerical schemes for step sizes $h=2^0,2^{-1},\ldots,2^{-10}$. The error graphs for this example are shown in \Cref{A} (b), produced with $n_\mathrm{t}=5,000$ trials.
\end{example}

\begin{example}\label{Example3} We show in \Cref{A} (c) the error graphs of the numerical approximations for \cref{Equation1} when there are $K=2$ delays $\tau_1=1$ and $\tau_2=1/4$, while $d,T,m$, and $\phi(t)$, are as in the above examples, but the matrices $A_0,A_1,A_2$ and functions $g_0,g_1,g_2$ are given by
\begin{align}
A_0&=\begin{pmatrix}-0.1&0.03\\-0.2&-0.04\\ \end{pmatrix},&&g_0\left(t,\begin{pmatrix}x_1\\x_2\\ \end{pmatrix},\begin{pmatrix}y_1\\y_2\\ \end{pmatrix},\begin{pmatrix}z_1\\z_2\\ \end{pmatrix}\right)=\frac{1}{5}\begin{pmatrix} \sin(x_1) \\ \cos(x_2) \\ \end{pmatrix},\label{Example3Parameters}\\
A_1&=\begin{pmatrix}0.15&0.1\\0.2&0.1\\ \end{pmatrix},&&g_1\left(t,\begin{pmatrix}x_1\\x_2\\ \end{pmatrix},\begin{pmatrix}y_1\\y_2\\ \end{pmatrix},\begin{pmatrix}z_1\\z_2\\ \end{pmatrix}\right)=\frac{1}{10}\begin{pmatrix} z_1-y_1 \\ y_2-z_2 \\ \end{pmatrix},\nonumber\\
A_2&=\begin{pmatrix}0.05&0.03\\0.04&0.01\\ \end{pmatrix},&&g_2\left(t,\begin{pmatrix}x_1\\x_2\\ \end{pmatrix},\begin{pmatrix}y_1\\y_2\\ \end{pmatrix},\begin{pmatrix}z_1\\z_2\\ \end{pmatrix}\right)=\frac{1}{5}\begin{pmatrix} \sin(x_2y_2z_2) \\ \cos(x_1y_1z_1) \\ \end{pmatrix}.\nonumber
\end{align}
These error graphs are produced with $n_\mathrm{t}=10,000$ trials. The reference solution is formed with the Milstein scheme of step size $h_\mathrm{R}=2^{-14}$, and the numerical schemes are simulated for step sizes $h=2^{-2},2^{-3},\ldots,2^{-10}$, so that the largest step size does not exceed the smallest delay value.
\end{example}

\begin{figure}[tbhp]
\centering
\subfloat[error graphs, linear SDDE, one delay]{\includegraphics[width=0.3\linewidth,trim=2 0 0 0,clip]{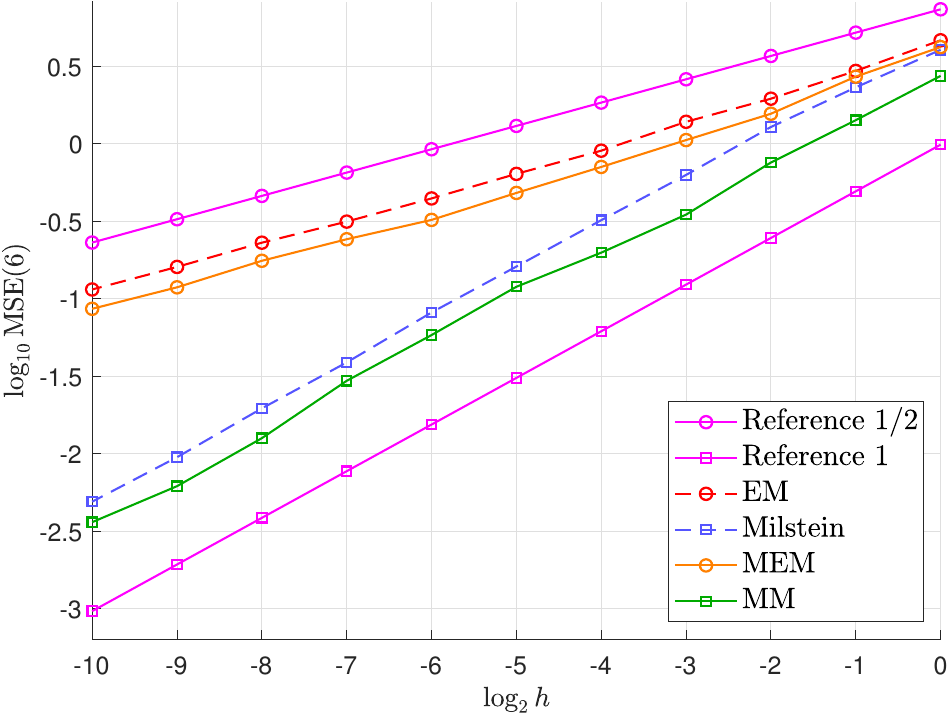}}\hspace{0.01\linewidth}
\subfloat[error graphs, semilinear SDDE, one delay]{\includegraphics[width=0.3\linewidth,trim=2 0 0 0,clip]{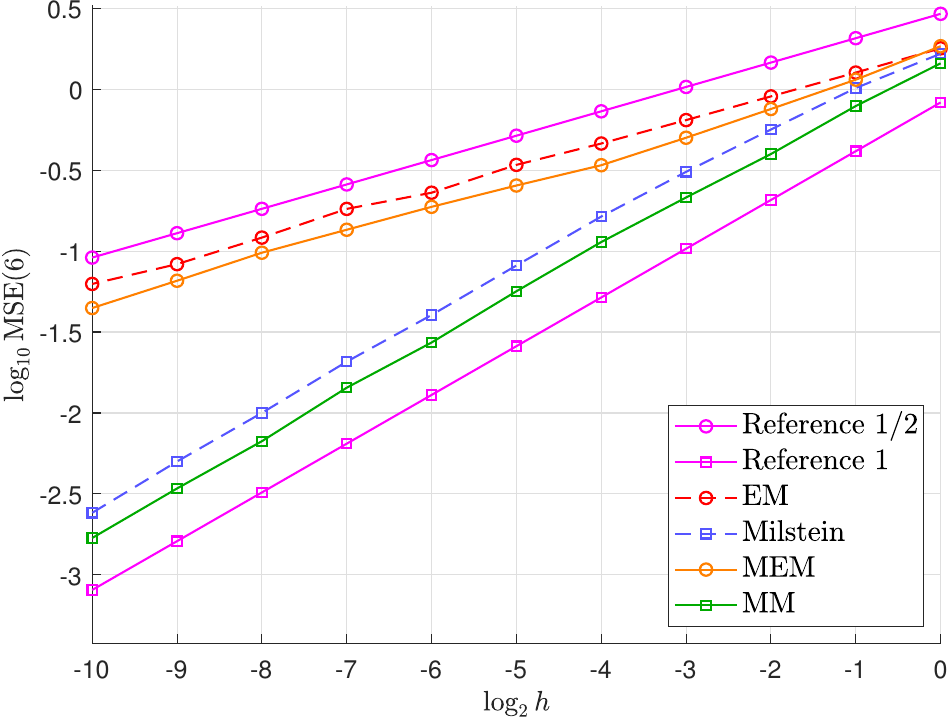}}\hspace{0.01\linewidth}
\subfloat[error graphs, semilinear SDDE, $K=2$ delays]{\includegraphics[width=0.3\linewidth,trim=2 0 0 0,clip]{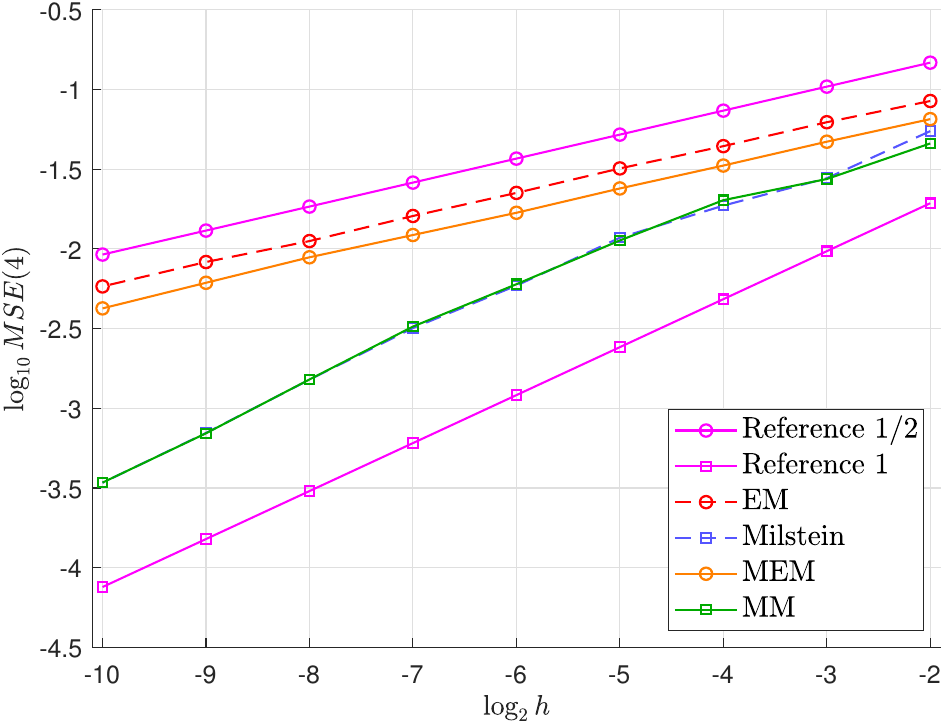}}
\caption{Error graphs of the Magnus-type approximations (shown in solid lines) for \cref{Equation1}, along with the regular Taylor approximations (dashed lines). Each simulation demonstrates the numerical schemes attaining their theoretical orders of convergence, where the EM and MEM schemes display OoC-$1/2$ behavior, while the Milstein and MM schemes display OoC-$1$ behavior.}
\label{A}
\end{figure}

\Cref{A} shows, in each of \hyperref[Example1]{Examples~\ref*{Example1}}, \ref{Example2}, and \ref{Example3}, the numerical schemes demonstrating their theoretical orders of convergence. That is, the EM and MEM schemes show OoC-$1/2$ behavior, while the Milstein and MM schemes show OoC-$1$ behavior. \Cref{A} shows this by plotting the $\log_2$ step sizes against the $\log_{10}$ errors, for each example. We also observe the magnitudes of the errors depending on the equation parameters. For example, at the finest step size $h=2^{-10}$, the strong-order-$1$ schemes have errors approximately $10^{-2.5}$, in \hyperref[Example1]{Examples~\ref*{Example1}} and \ref{Example2}, while these schemes have errors $10^{-3.5}$, in \hyperref[Example3]{Example~\ref*{Example3}}.

To demonstrate the applicability of our schemes to spatially extended systems, we now simulate the solutions of an SPDDE. This example highlights how the MEM scheme can outperform the EM scheme under spatial refinement.

\subsection{Application to the delayed stochastic heat equation}

\begin{example}\label{Example4} We numerically solve the SPDDE \cref{Equation3} by simulating the solutions of \cref{Equation4} with the EM and MEM schemes, and in doing this, we observe an advantage that the MEM scheme has over the EM scheme. Due to condition \cref{StabilityCondition}, the EM approximation requires a step size $h$ that decreases squarely proportionally with the spatial discretization $\Delta x$, while the MEM scheme is not restricted by this condition and achieves its order of convergence for larger $h$. We also present and compare sample solutions of \cref{Equation3} in the cases of no noise ($c=0$, so that \cref{Equation3} becomes a deterministic PDE), uncorrelated noise, and also correlated noise. Uncorrelated noise is the simpler to simulate, but we also demonstrate the simulation with correlated noise, to present a realistic scenario. To show how these situations affect the solution process, we show cross sections of the solutions, first for fixed $x$ and then for fixed $t$. On each of these cross sections, we superimpose the cross sections of the solution in the case of no delay ($\tau=0$), to demonstrate how the delay affects the model.

We select the delayed cooling function
\begin{equation}C_U(t-\tau,x)=\begin{cases}\left(\dfrac{x-b}{b-a}\right)r_aU(t-\tau,a)+\left(\dfrac{a-x}{b-a}\right)r_bU(t-\tau,b),&x\in(0,1),\\
0,&\textrm{otherwise},\end{cases}\label{CoolingFunction}\end{equation}
for constants $0<a<b<1$ and constant rates $r_a,r_b\geqslant0$. This cooling function represents the scenario where we observe the temperature at positions $x=a$ and $x=b$, and then apply external cooling in attempt to shift the total temperature towards its average, while the delay $\tau$ is the time between observing the temperature and applying the cooling. For example, if there is no delay ($\tau=0$) and we choose rates $r_a=r_b=1$, then the cooling $C_U$ is the linear segment from $U(t,a)$ to $U(t,b)$, multiplied by $-1$.

The $Q$-Wiener process is defined in terms of a spatial covariance function $Q:[0,1]\times[0,1]\to[0,\infty)$. In this example, we use the identity function $Q(x,y)=\mathbb{I}(x=y)$ in the case of uncorrelated noise, while using
\begin{equation}Q(x,y)=\min\{x,y\}\label{QFunction}\end{equation}
to model correlated noise. The function $Q$ given by \cref{QFunction} is known (see \cite[page 37]{GhanemSpanos1991Book}) to have eigenvalues and eigenfunctions, $\lambda_j$ and $\phi_j$, respectively given by
\[\lambda_j=\frac{4}{\pi^2(2j-1)^2}\quad\textrm{and}\quad\phi_j(x)=\sqrt{2}\sin\bigg(\frac{x}{\sqrt{\lambda_j}}\bigg),\quad j=1,\ldots,m,\ldots.\]

In order to simulate the system \cref{Equation4}, we convert to a system of the form of \cref{Equation1}. This is achieved by applying the Karhunen--Lo\`eve Theorem (\cref{AppendixSPDE}, \cref{QKLTheorem}), which expresses $W^c$ as a countable linear combination of independent Wiener processes. Truncated to $m$ terms, we approximate $W^c$ and $\int_0^tU(s,x)\,\mathrm{d}W^c(s)$ with
\begin{align}
W^c(t,x)&\approx\sum_{j=1}^m\sqrt{\lambda_j}\phi_j(x)W_j(t),\quad\label{TruncatedWcandIWc}\\
\int_0^tU(s,x)\,\mathrm{d}W^c(s)&\approx\sum_{j=1}^m\int_0^tU(s,x)\sqrt{\lambda_j}\phi_j(x)\,\mathrm{d}W_j(s),\nonumber
\end{align}
for each $(t,x)\in[0,\infty)\times[0,1]$. We also discretize space into $d=m$ intervals, with $\Delta x=1/d$, and we select the points $a=x_1$ and $b=x_{d-1}$ close to the spatial boundary.

Using \cref{TruncatedWcandIWc}, we simulate samples of $W^c$. These simulations are given in \Cref{D}, showing the difference between uncorrelated and correlated noise. We produce these figures by truncating the expansion \cref{QWienerFullExpansion} to $50$ terms, while time and space are discretized and simulated at points $(t,x)=(n2^{-16},j/50)$, for $n=0,1,\ldots,2^{18}$ and $j=0,\ldots,50$.

\begin{figure}[tbhp]
\centering
\subfloat[Heat map of a sample $Q$-Wiener process for uncorrelated noise, $Q(x,y)=\mathbb{I}(x=y)$.]{\includegraphics[width=0.3\linewidth,trim=2 0 0 0,clip]{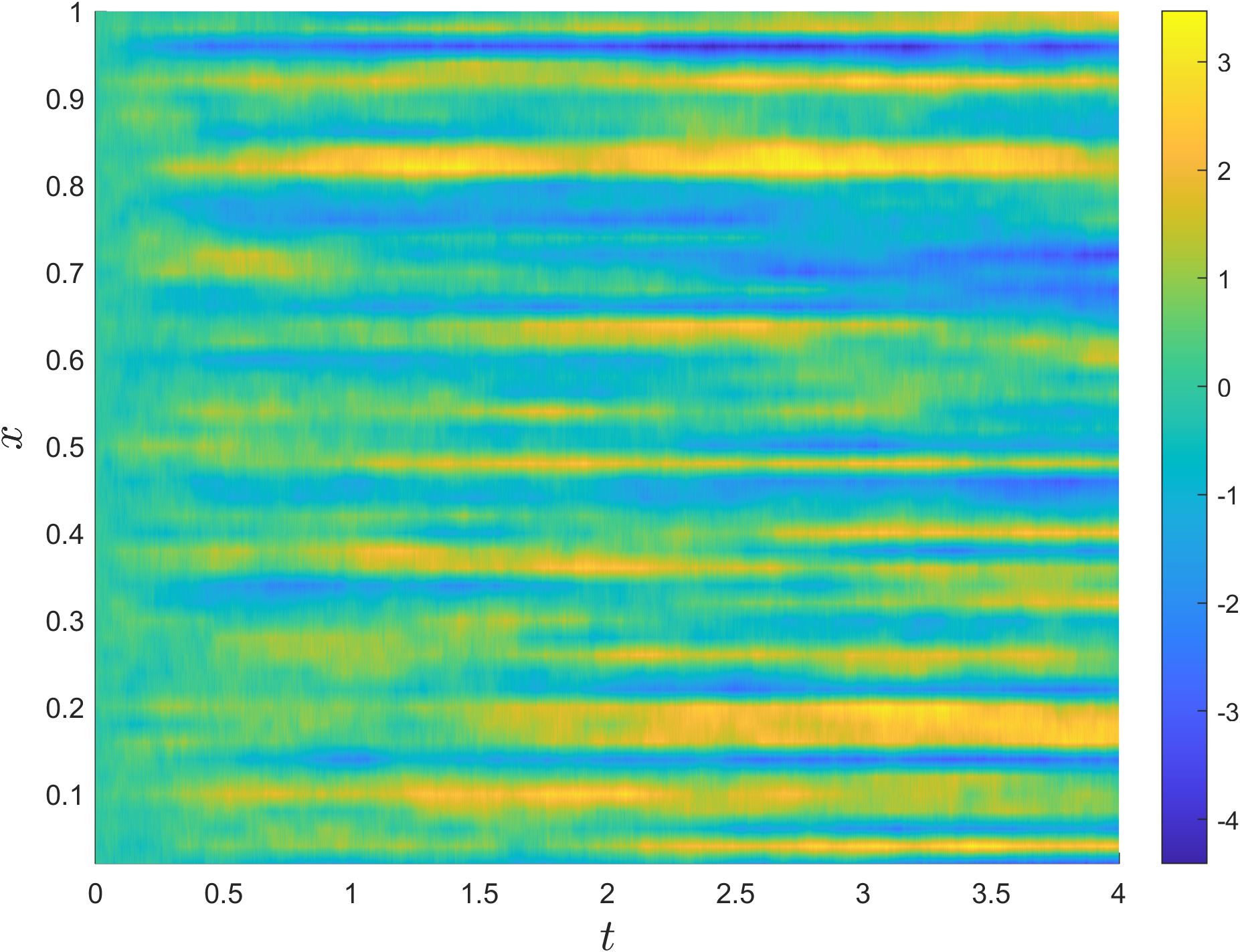}}\hspace{1cm}
\subfloat[Heat map of a sample $Q$-Wiener process for correlated noise, with $Q$ defined by \cref{QFunction}.]{\includegraphics[width=0.3\linewidth,trim=2 0 0 0,clip]{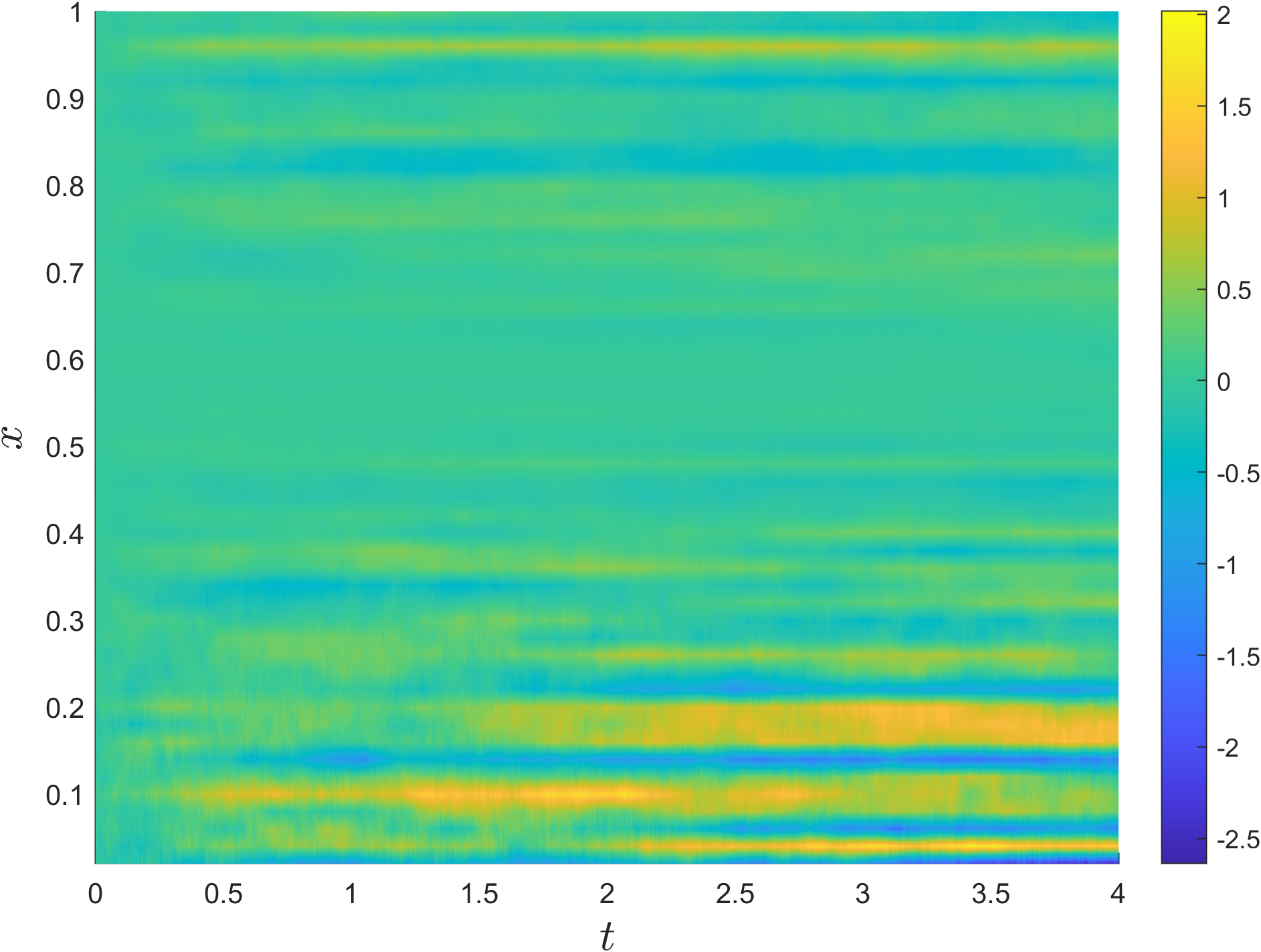}}
\caption{Sample $Q$-Wiener processes, for each case of uncorrelated and correlated noise.}
\label{D}
\end{figure}

Recalling $x_j=j\,\Delta x$ and defining both $U_j(t)=U(t,x_j)$ and $W^c_j(t)=W^c(t,x_j)$ for each $j=0,\ldots,d$, system \cref{Equation4} becomes
\begin{align}
\mathrm{d}U_0(t)&=0\,\mathrm{d}t,\label{BigSDDESystem}\\
\mathrm{d}U_j(t)&=\left[D\frac{U_{j-1}(t)+U_{j+1}(t)-2U_j(t)}{(\Delta x)^2}+\left(\frac{x_j-b}{b-a}\right)r_aU_1(t-\tau)\right.\nonumber\\
&\quad\quad\left.+\left(\frac{a-x_j}{b-a}\right)r_bU_{d-1}(t-\tau)\right]\,\mathrm{d}t+\frac{cU_j(t)}{\sqrt{\Delta x}}\,\mathrm{d}W_j^c(t),\nonumber\\
\mathrm{d}U_d(t)&=0\,\mathrm{d}t,\nonumber
\end{align}
for $j=1,\ldots,d-1$ and $t\geqslant0$, equipped with $(U_0(t),\ldots,U_d(t))=(T_0(x_0),\ldots,T_0(x_d))$ when $t\in[-\tau,0]$. Due to the boundary conditions, the boundary processes are fixed with $U_0(t)=U_d(t)=0$ for all $t\geqslant0$, so to simplify the indexing, we omit the cross section $(U_0(t,x_0))_{t\geqslant0}$ in the following construction. At the spatial point $x_i$, the truncation \cref{TruncatedWcandIWc} gives
\begin{align}
&\int_0^tU_i(s)\,\mathrm{d}W^c(s)=\int_0^tU(s,x_i)\,\mathrm{d}W^c(s)\label{IntApproxEqual}\\
&\hspace{1cm}\approx\sum_{j=1}^m\int_0^tU(s,x_i)\sqrt{\lambda_j}\phi_j(x_i)\,\mathrm{d}W_j(s)=\sum_{j=1}^m\int_0^tU_i(s)\sqrt{\lambda_j}\phi_j(x_i)\,\mathrm{d}W_j(s).\nonumber
\end{align}
\begin{remark}
In the case of uncorrelated noise, \cref{IntApproxEqual} reduces to
\[\int_0^tU_i(s)\,\mathrm{d}W^c(s)=\int_0^tU_i(s)\,\mathrm{d}W_i(s),\quad i=1,\ldots,m-1.\]
\end{remark}

We assume the approximation is equality in \cref{IntApproxEqual}, so that by writing $\bm{U}(t)=(U_1(t),\ldots,U_d(t))^\intercal$, the system \cref{BigSDDESystem} becomes the $d$-dimensional SDDE
\begin{align}
\mathrm{d}\bm{U}(t)&=\left[\frac{D}{(\Delta x)^2}A_0+\frac{r_a}{d-2}v_1\,U_1(t-\tau)+\frac{r_b}{d-2}v_2\,U_{d-1}(t-\tau)\right]\,\mathrm{d}t\label{BigMatrixSDDE}\\
&\quad+\sum_{j=1}^m\frac{c}{\sqrt{\Delta x}}A_j\bm{U}(t)\,\mathrm{d}W_j(t),\nonumber
\end{align}
where $A_0=(A_0^{ij})_{i,j=1,\ldots,d}$ is the tridiagonal matrix with $(i,j)$ entry
\begin{equation}A_0^{ij}=\begin{cases}-2&\textrm{if }j=i<d,\\ 1&\textrm{if }j=i\pm1\;\textrm{and}\;i<d,\\ 0&\textrm{otherwise},\end{cases}\label{SPDDEA0}\end{equation}
while $v_1=(v_1^1,\ldots,v_1^d)^\intercal$ and $v_2=(v_2^1,\ldots,v_2^d)^\intercal$ are vectors with entries
\[\left(v_1^j,v_2^j\right)=\begin{cases}\left(j+1-d,\,1-j\right)&\textrm{if }j=1,\ldots,d-1,\\ (0,\,0)&\textrm{if }j=d,\end{cases}\]
and $A_j$ is the diagonal matrix
\begin{equation}A_j=\frac{c}{\sqrt{\Delta x}}\sqrt{\lambda_j}\mathrm{diag}\left(\phi_j(x_1),\phi_j(x_2),\ldots,\phi_j(x_d)\right),\label{SPDDEAj}\end{equation}
for each $j=1,\ldots,m$. \Cref{BigMatrixSDDE} is a linear SDDE of the form
\begin{equation}\mathrm{d}\bm{U}(t)=[A_0\bm{U}(t)+f(t,\bm{U}(t-\tau))]\,\mathrm{d}t+\sum_{j=1}^mA_j\bm{U}(t)\,\mathrm{d}W_j(t),\label{LinearSDDEfromSPDDE}\end{equation}
which is a particular case of \cref{Equation1}.

\begin{remark}
In the case of uncorrelated noise, \cref{LinearSDDEfromSPDDE} holds similarly, but \cref{SPDDEAj} simplifies so that $A_j$ has $(j,j)$ entry $c/\sqrt{\Delta x}$, with all other entries being zero.
\end{remark}

We numerically solve \cref{LinearSDDEfromSPDDE} with the EM and MEM schemes, in each case of uncorrelated and correlated noise.  In these simulations, we select a spatial discretization with dimension $d=50$, delay value $\tau=1$, cooling rates $r_a=1$ and $r_b=10$, noise constant $c=0.15$, diffusion coefficient $D=1/25$, and initial temperature $T_0(x)=\sin(2\pi x)$. With this selection of parameters, condition \cref{StabilityCondition} reduces to $100h<1/2$. For $h$ of the form $h=2^{-i}$, for $i\in\mathbb{N}$, the Euler--Maruyama approximation is numerically stable for $i\geqslant8$, but \cref{StabilityCondition} is breached for $h>2^{-8}$. We observe this in the error graphs for this example, which we show in \Cref{B}. These error graphs are produced with $n_\mathrm{t}=1,000$ trials, for both the cases where \cref{StabilityCondition} is satisfied (top row), and also for the general case of larger step sizes (bottom row). We show the error graphs for the cases of uncorrelated noise (first column) and correlated noise (second column). We simulate the schemes with step sizes $h=2^0,\ldots,2^{-13}$. Errors are calculated using a reference solution, simulated with the EM scheme of step size $h_\mathrm{R}=2^{-16}$.

\begin{figure}[tbhp]
\centering
\subfloat[error graphs, EM stability condition not satisfied, uncorrelated noise]{\includegraphics[width=0.3\linewidth,trim=2 0 0 0,clip]{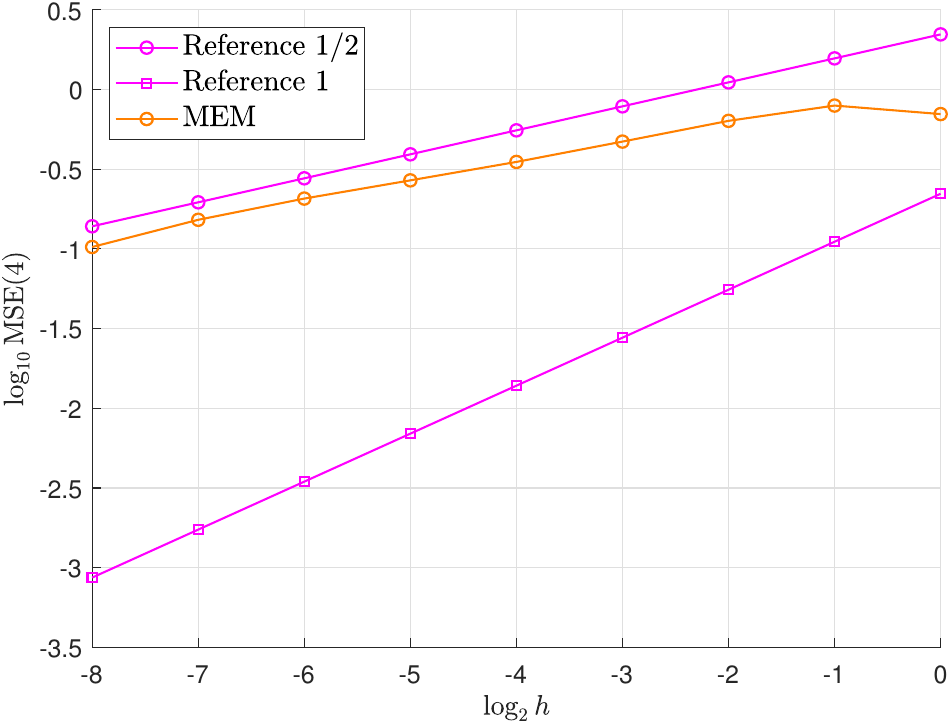}}\hspace{1cm}
\subfloat[error graphs, EM stability condition not satisfied, correlated noise]{\includegraphics[width=0.3\linewidth,trim=2 0 0 0,clip]{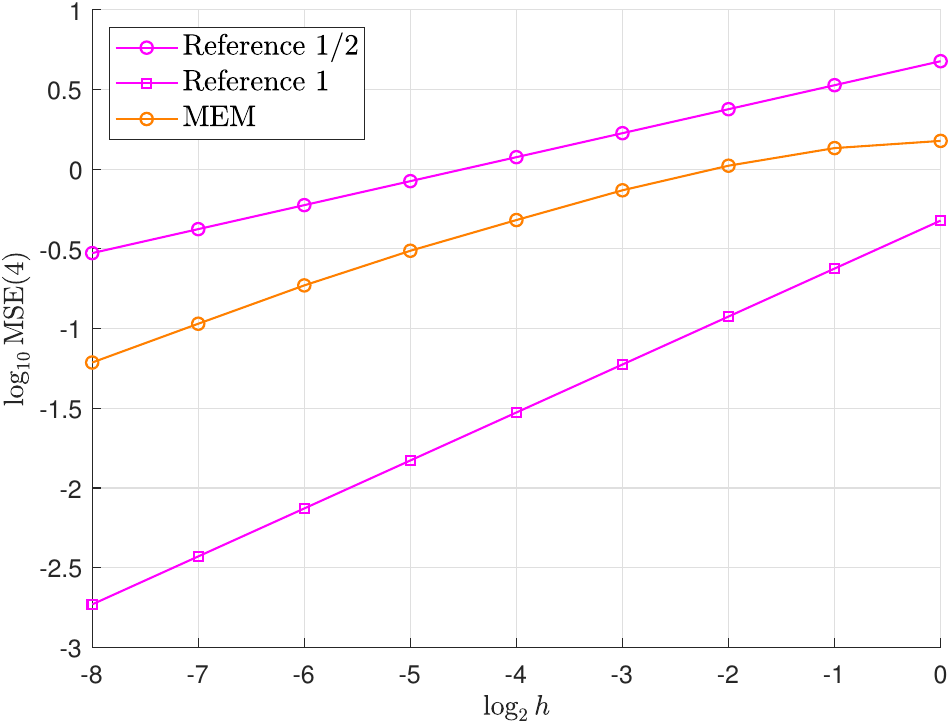}}
\par\bigskip
\subfloat[error graphs, EM stability condition satisfied, uncorrelated noise]{\includegraphics[width=0.3\linewidth,trim=2 0 0 0,clip]{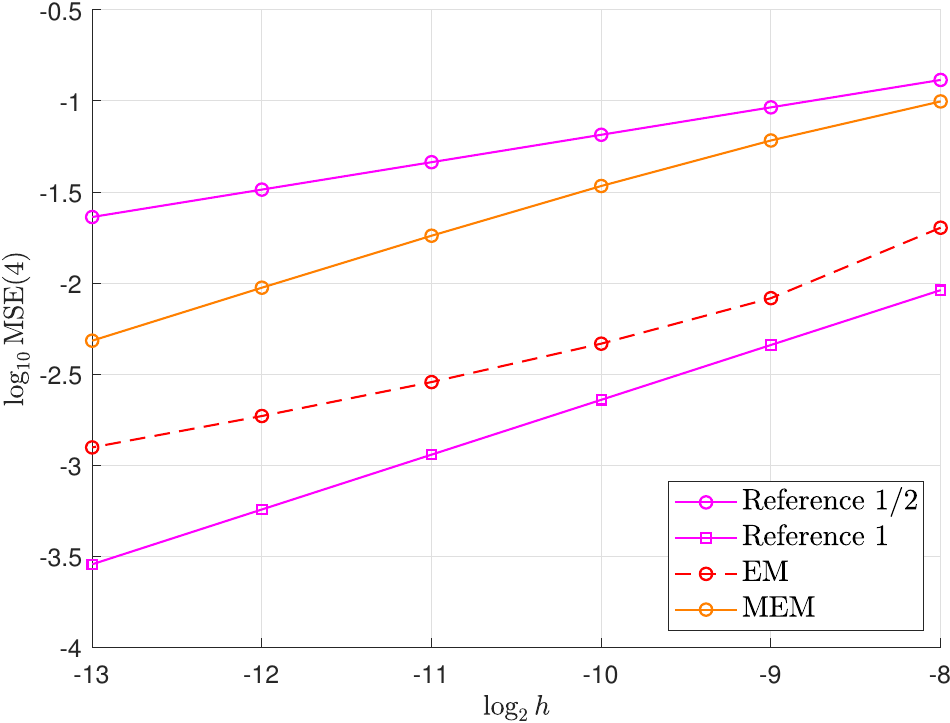}}\hspace{1cm}
\subfloat[error graphs, EM stability condition satisfied, correlated noise]{\includegraphics[width=0.3\linewidth,trim=2 0 0 0,clip]{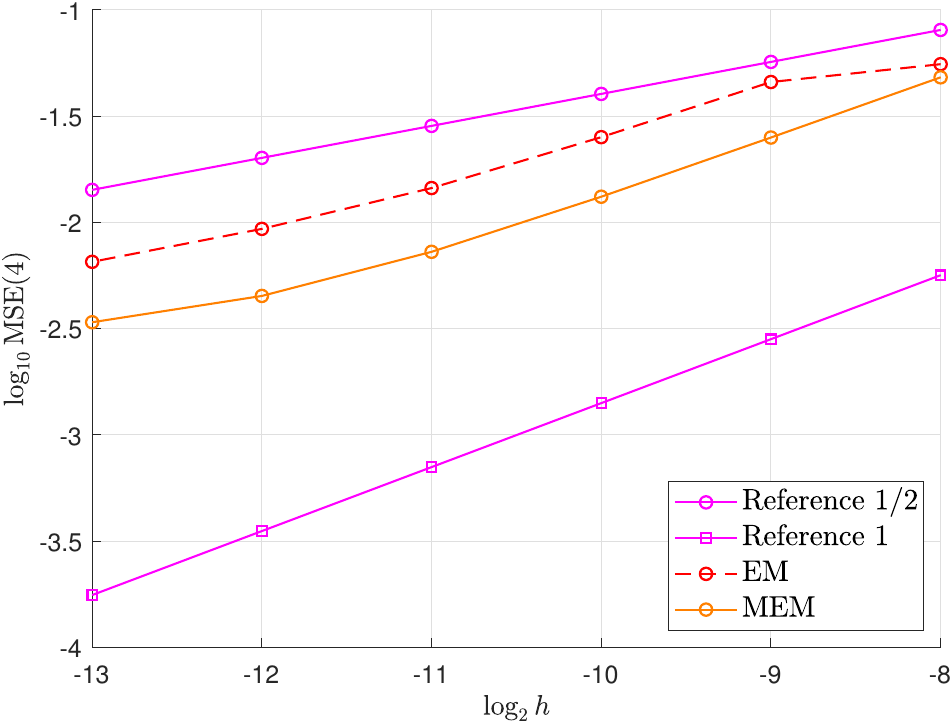}}
\caption{Error graphs for the high-dimensional SDDE \cref{BigMatrixSDDE}. The plots are separated according to whether the EM stability condition \cref{StabilityCondition} is satisfied. A large spatial discretization corresponds to \cref{StabilityCondition} not being satisfied. These graphs show that when \cref{StabilityCondition} is satisfied, both the EM and MEM schemes achieve OoC $1/2$. When the condition is not satisfied and the EM scheme is inappropriate to use, the MEM scheme retains its order of convergence, providing an ideal alternative to the EM scheme.}
\label{B}
\end{figure}

\Cref{B} shows the error graphs of the EM and MEM approximations for \cref{BigMatrixSDDE} in the cases of uncorrelated noise (first column) and correlated noise (second column). In each column, we show the error graphs across two rows. We see the EM scheme performing better (in terms of reduced error) in the case of uncorrelated noise, while the MEM scheme is superior with correlated noise, due to the matrices in \cref{BigMatrixSDDE}. The first row of \Cref{B} shows the errors in the schemes for the larger step sizes, $h=2^0,\ldots,2^{-8}$, and the second row shows the errors for the smaller step sizes, $h=2^{-8},\ldots,2^{-13}$. In the case of the larger step sizes, the condition \cref{StabilityCondition} is breached, so the EM scheme is inappropriate to use, but here the MEM scheme displays OoC-$1/2$ behavior, showing that the MEM scheme is suitable to use in place of the EM scheme for these larger (temporal) step sizes. This situation of \cref{StabilityCondition} being breached corresponds to high spatial discretization. We also see that when \cref{StabilityCondition} is satisfied, in the case of smaller step sizes, then either numerical scheme achieves OoC $1/2$. That is, when both the EM and MEM schemes are numerically stable, then either method is viable. To determine which method we should use in this case, we may consider other issues, such as computational efficiency.

\begin{figure}[tbhp]
\centering
\subfloat[sample solution, $c=0$]{\includegraphics[width=0.3\linewidth,trim=2 0 0 20,clip]{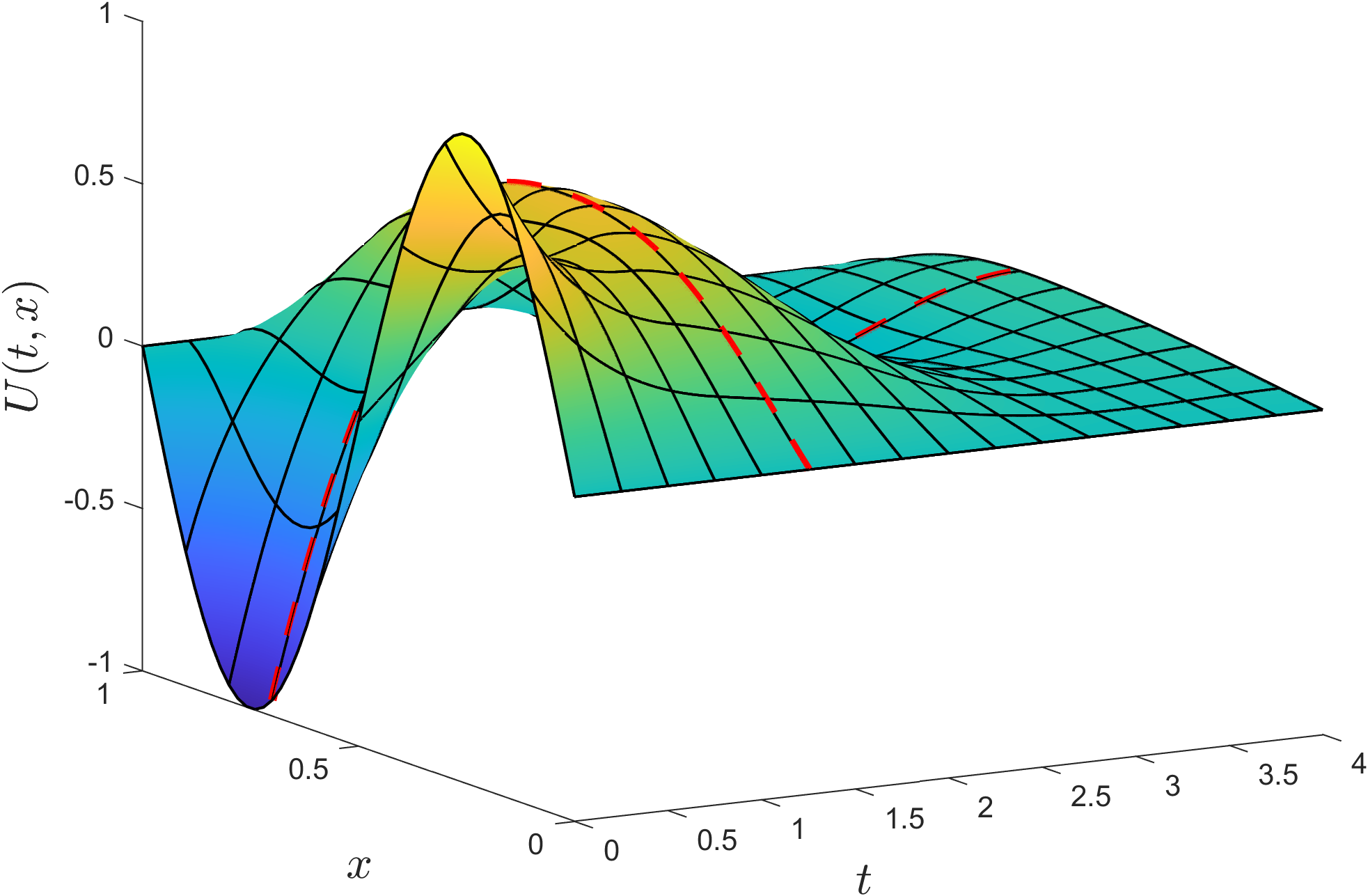}}\hfill
\subfloat[sample solution, uncorrelated noise]{\includegraphics[width=0.3\linewidth,trim=2 0 0 20,clip]{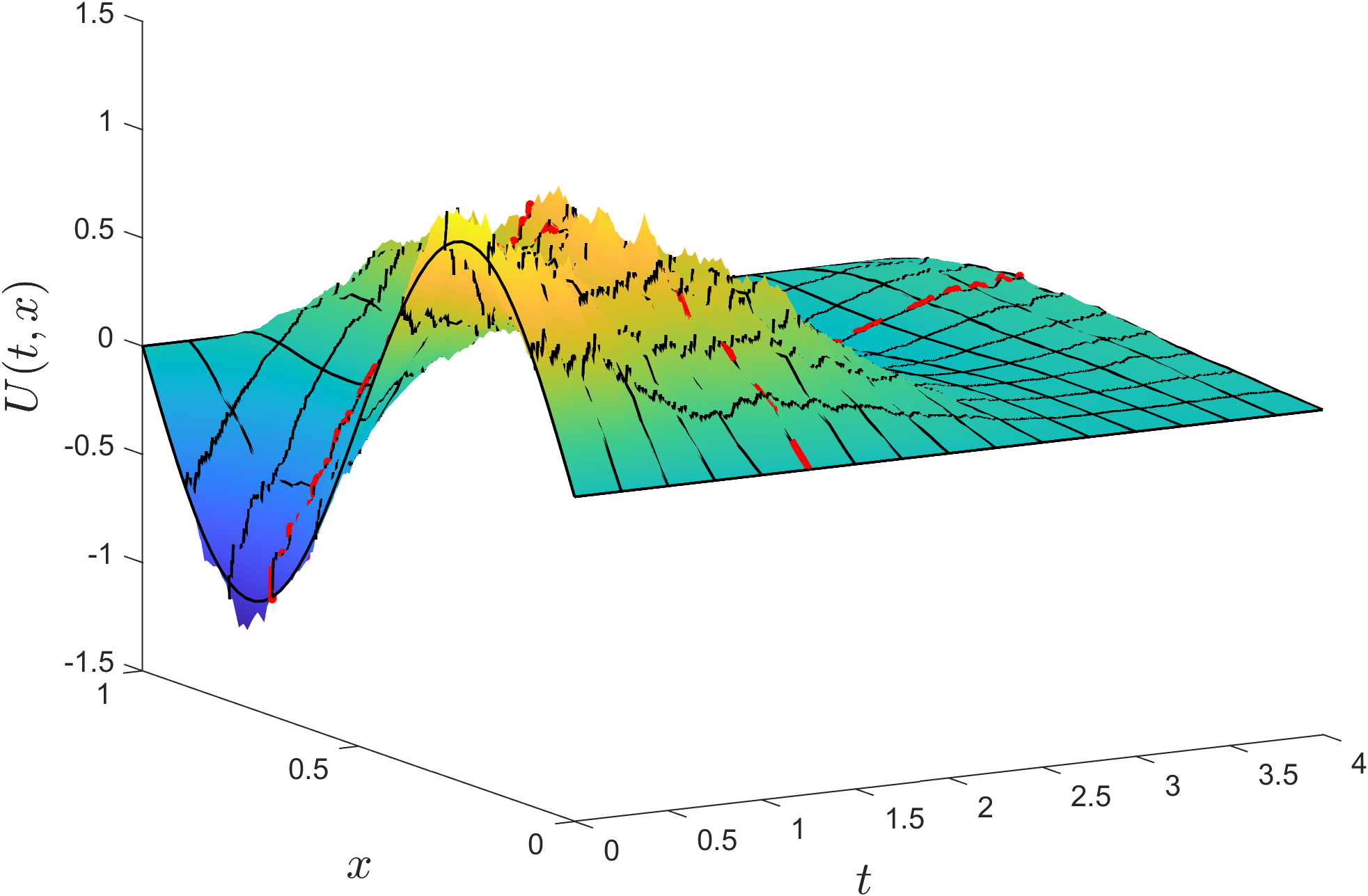}}\hfill
\subfloat[sample solution, correlated noise]{\includegraphics[width=0.3\linewidth,trim=2 0 0 20,clip]{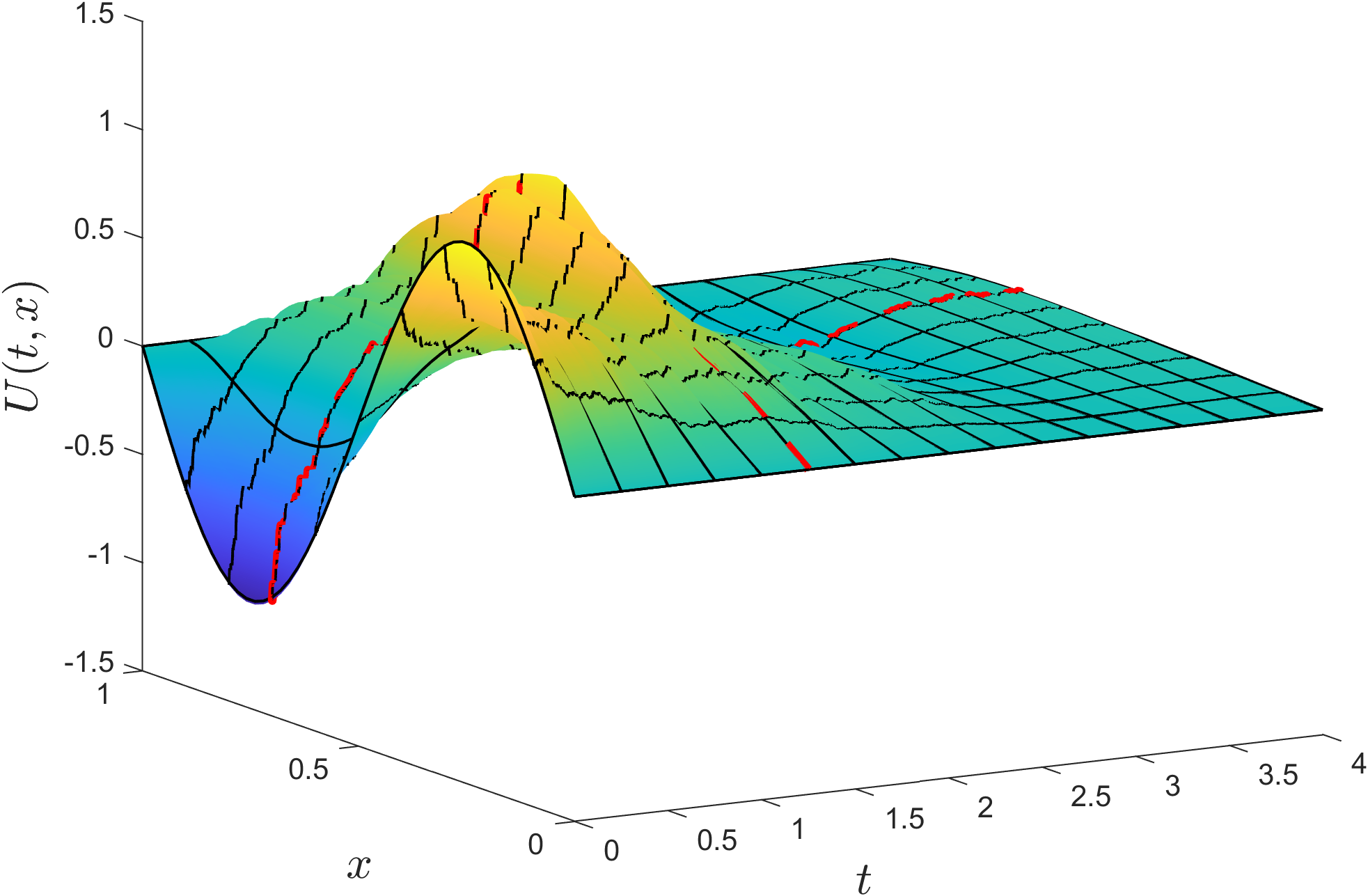}}
\par\bigskip
\subfloat[heat map of solution, $c=0$]{\includegraphics[width=0.3\linewidth,trim=2 0 0 0,clip]{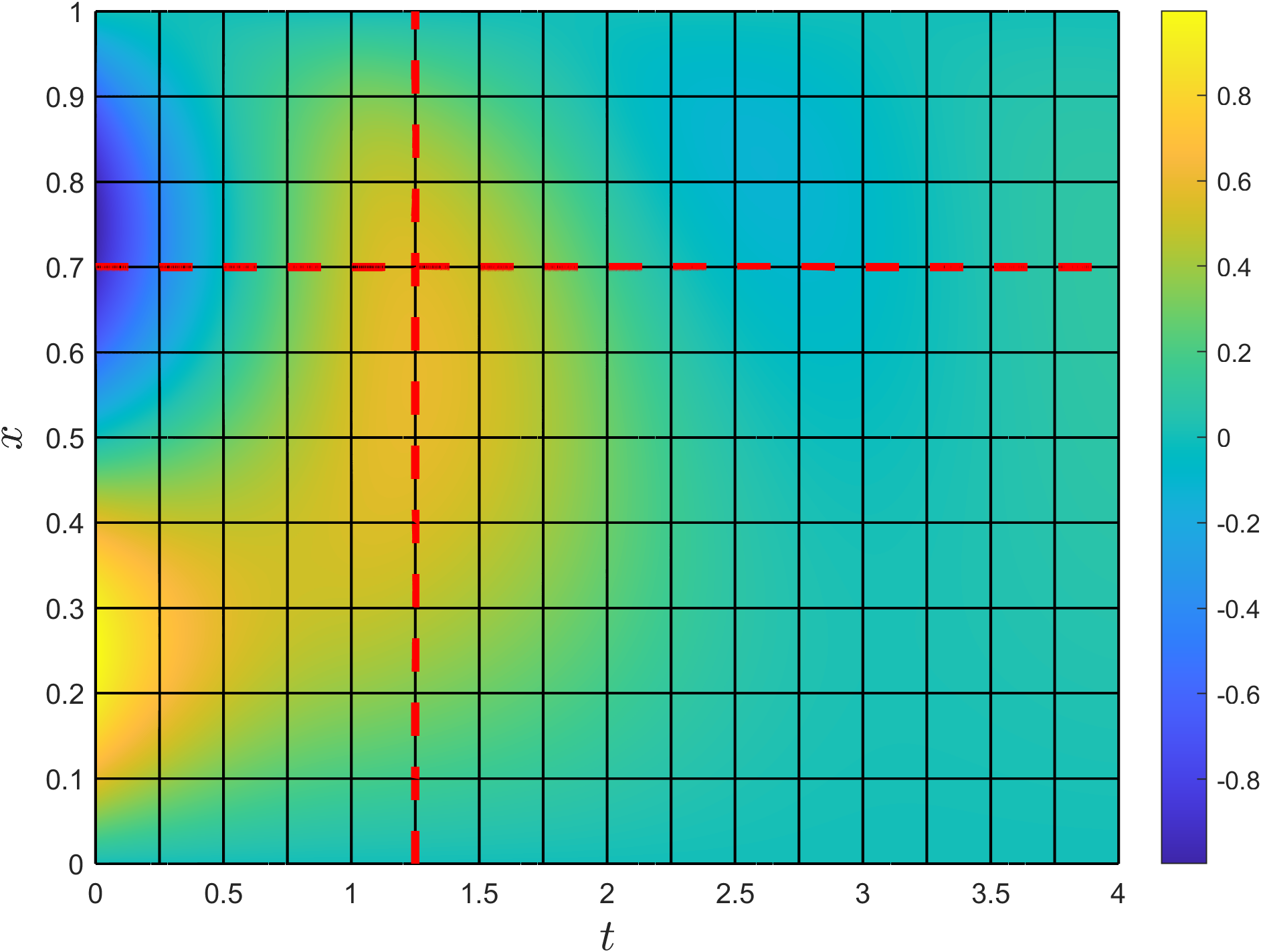}}\hfill
\subfloat[heat map of solution, uncorrelated noise]{\includegraphics[width=0.3\linewidth,trim=2 0 0 0,clip]{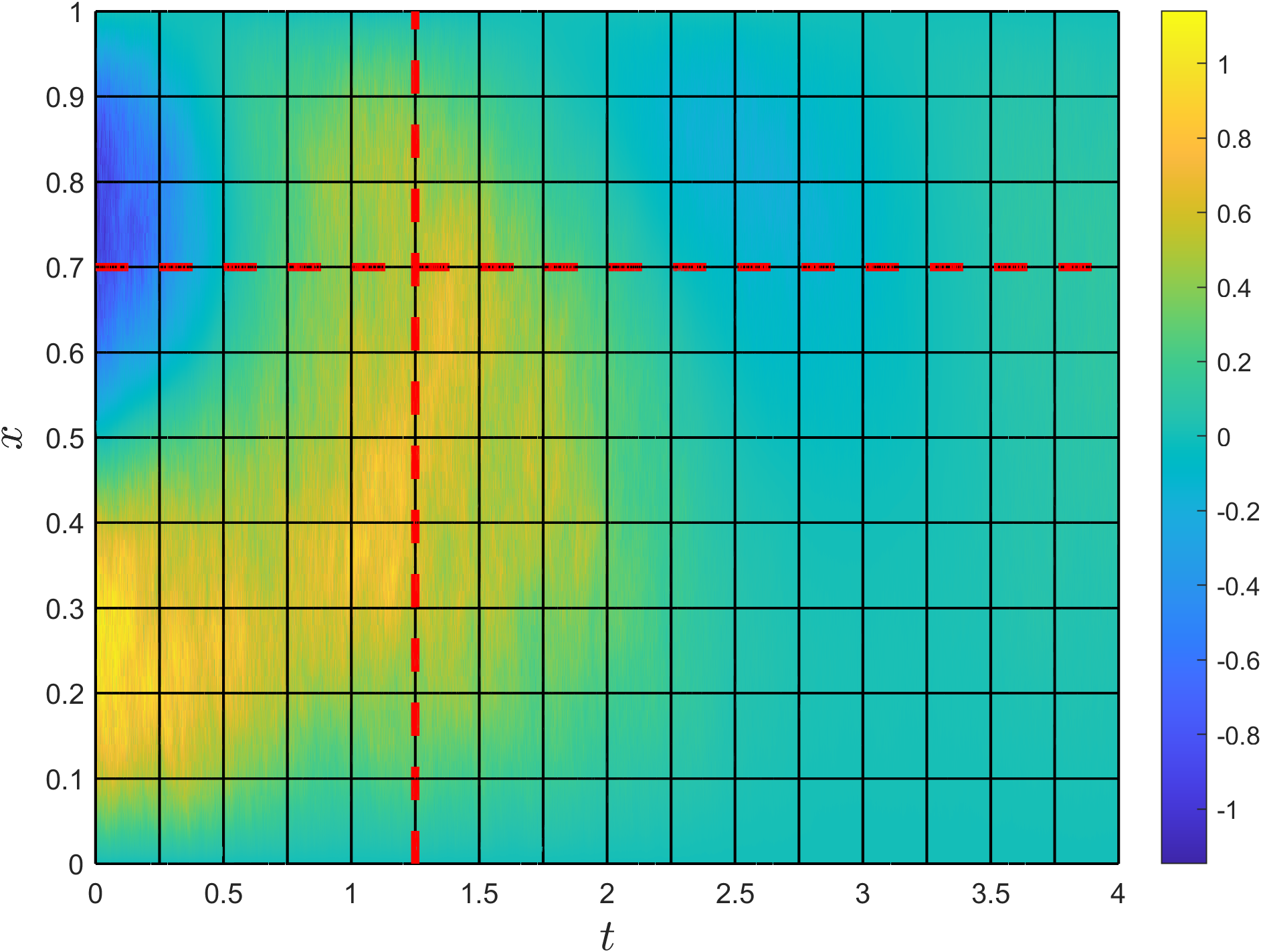}}\hfill
\subfloat[heat map of solution, correlated noise]{\includegraphics[width=0.3\linewidth,trim=2 0 0 0,clip]{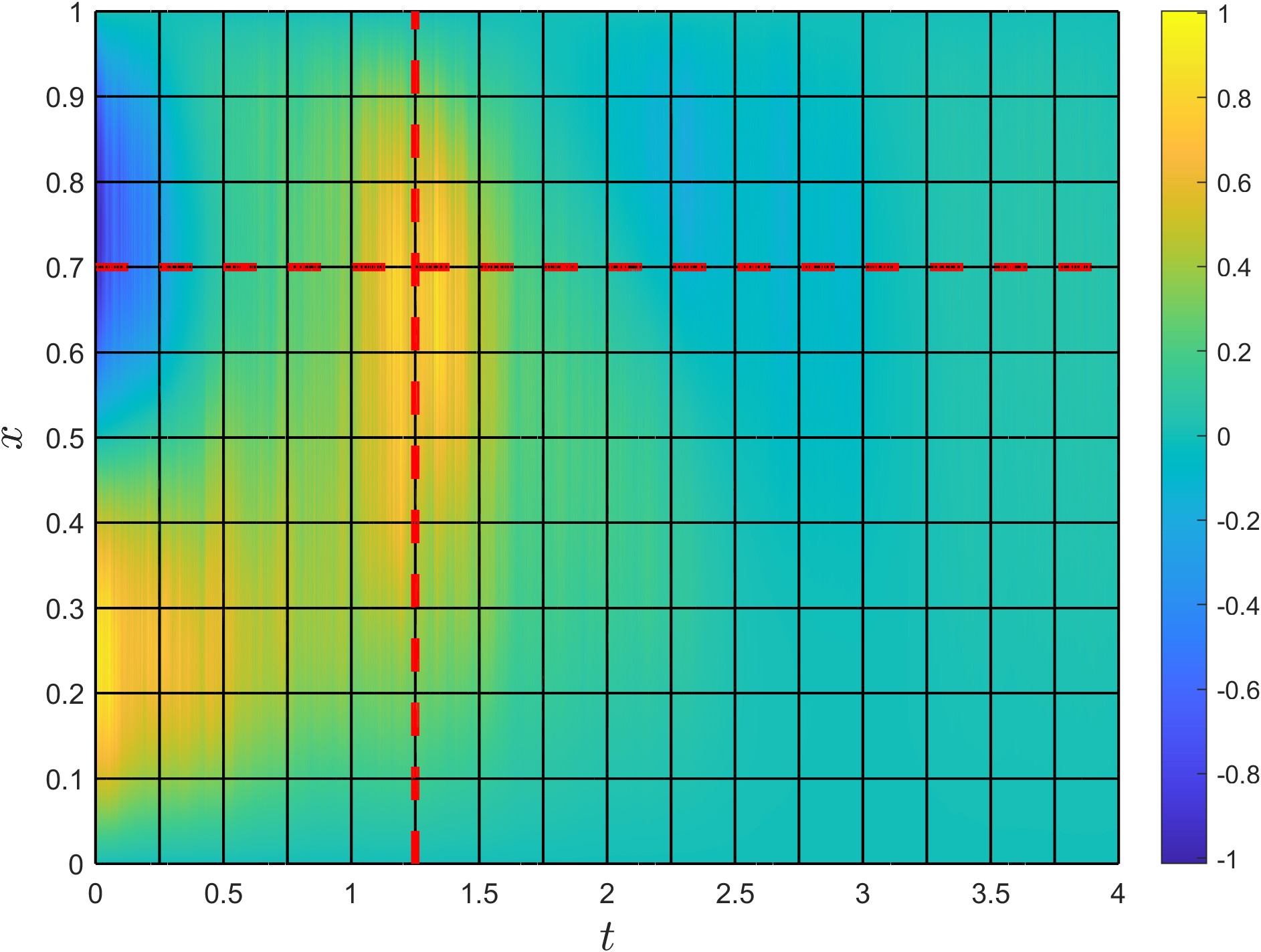}}
\par\bigskip
\subfloat[cross section ($x=0.7$), $c=0$, comparing cooling with ($\tau=1$) and without ($\tau=0$) delay]{\includegraphics[width=0.3\linewidth,trim=2 0 0 0,clip]{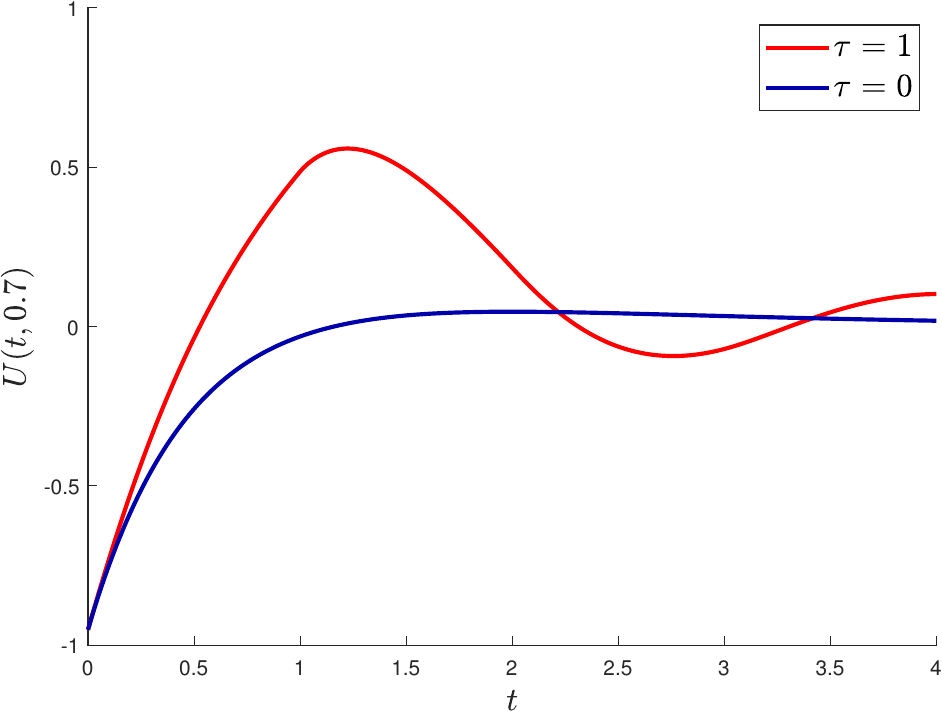}}\hfill
\subfloat[cross section ($x=0.7$), uncorrelated noise, comparing cooling with and without delay]{\includegraphics[width=0.3\linewidth,trim=2 0 0 0,clip]{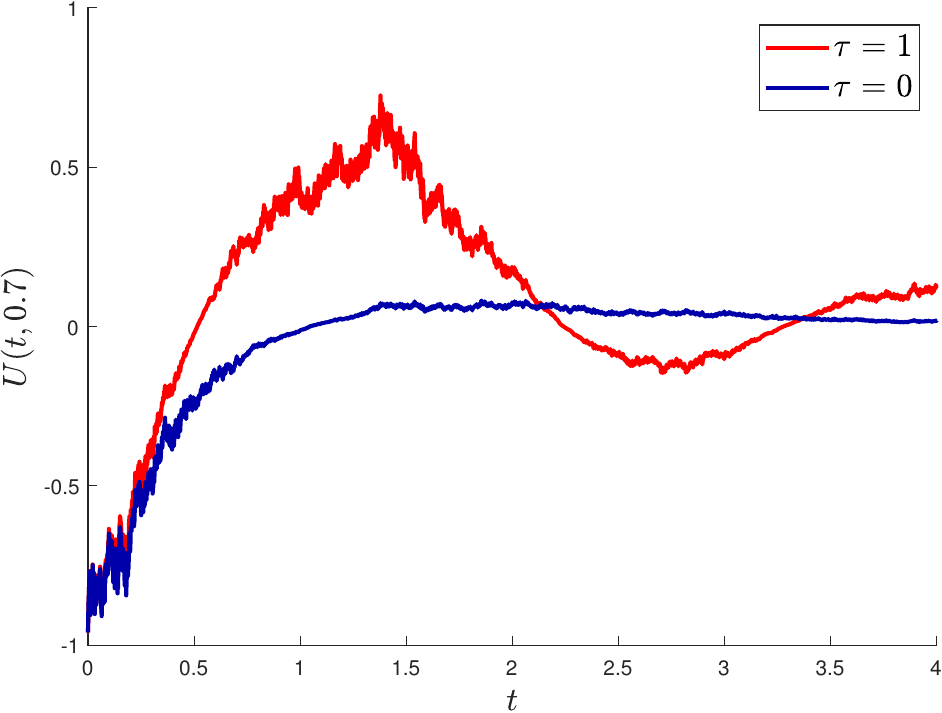}}\hfill
\subfloat[cross section ($x=0.7$), correlated noise, comparing cooling with and without delay]{\includegraphics[width=0.3\linewidth,trim=2 0 0 0,clip]{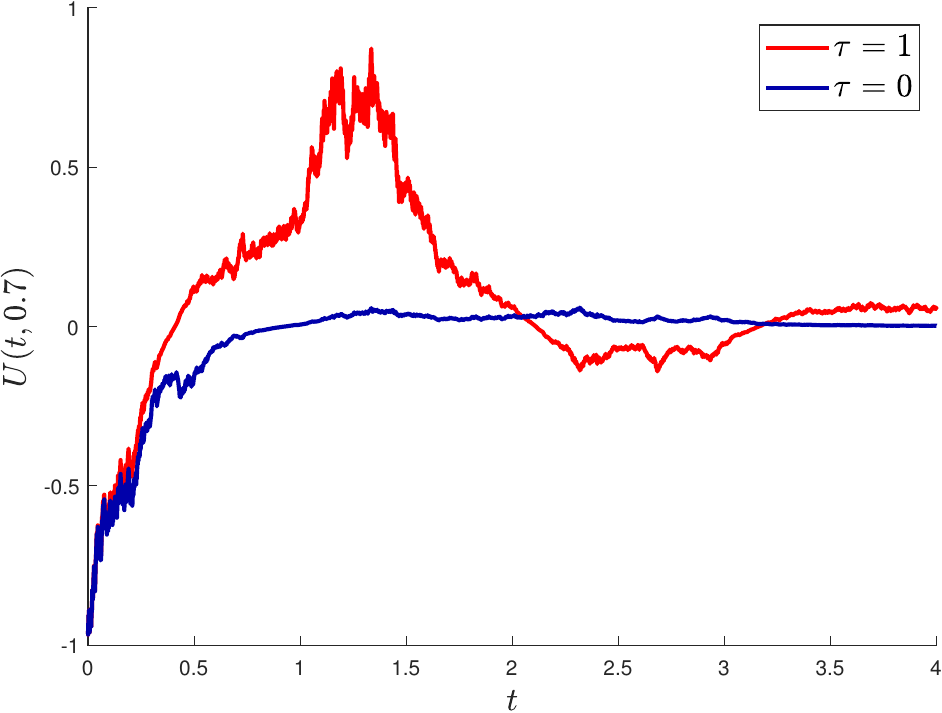}}
\par\bigskip
\subfloat[cross section ($t=1.25$), $c=0$]{\includegraphics[width=0.3\linewidth,trim=2 0 0 0,clip]{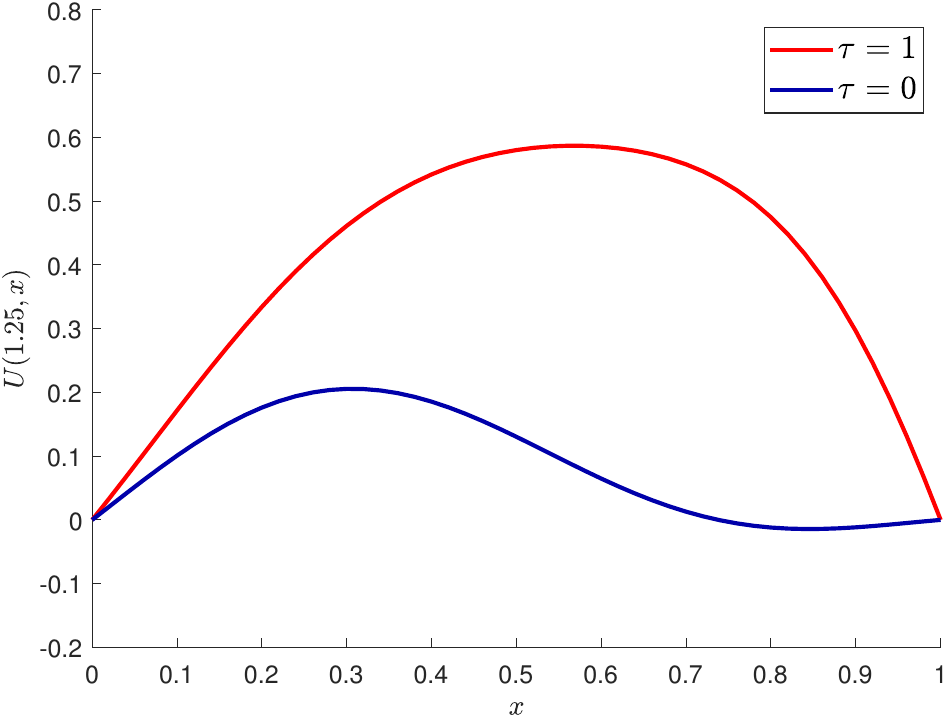}}\hfill
\subfloat[cross section ($t=1.25$), uncorrelated noise]{\includegraphics[width=0.3\linewidth,trim=2 0 0 0,clip]{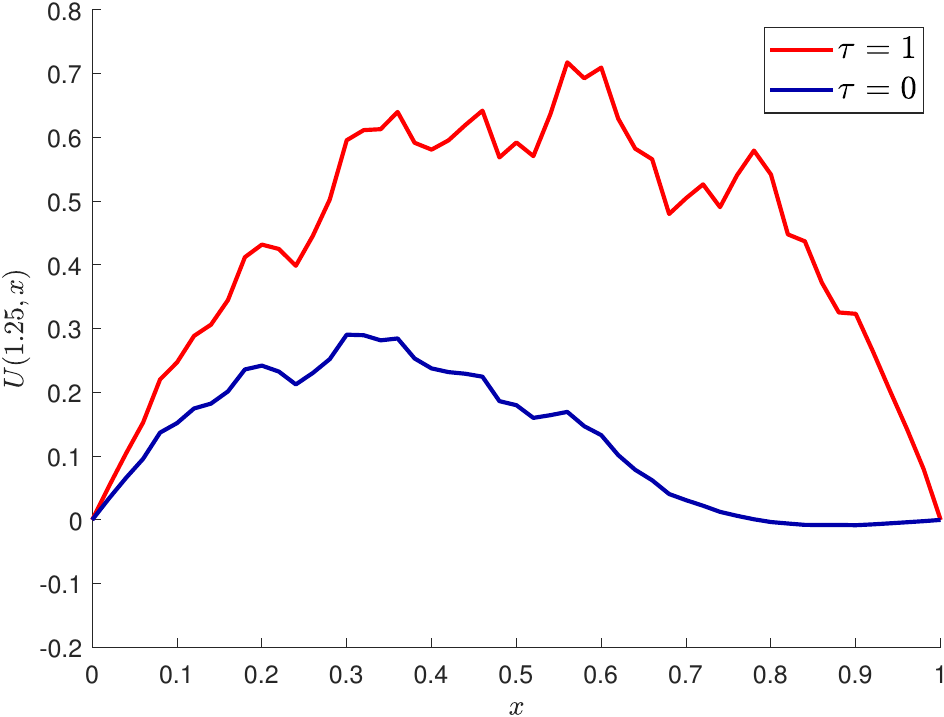}}\hfill
\subfloat[cross section ($t=1.25$), correlated noise]{\includegraphics[width=0.3\linewidth,trim=2 0 0 0,clip]{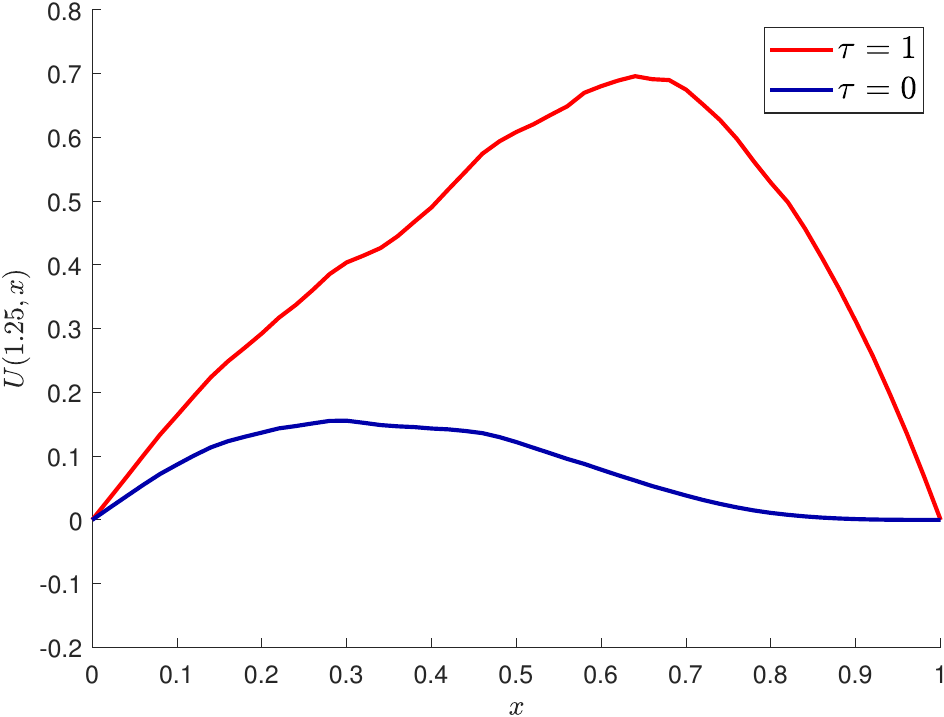}}
\par\bigskip
\caption{Numerical solutions of the stochastic heat equation with delayed cooling, for the cases of no noise, uncorrelated noise, and also correlated noise, including cross sections comparing each solution with the non-delayed solution.}
\label{C}
\end{figure}

Computation times of matrix exponentials becomes significant as the dimension increases, but research is underway to improve these simulation times. Future applications of the Magnus schemes (to SPDDEs or otherwise) may incorporate techniques to reduce computation times for matrix exponentials, such as those given by Moler and Van Loan \cite{MolerVanloan1978,MolerVanLoan2003}. For example, such techniques include Krylov methods, as described by Saad \cite{Saad1992}. Future researchers simulating our SPDDE may exploit the tridiagonal structure \cref{SPDDEA0} of $A_0$, which also becomes symmetric if we omit the $d^\mathrm{th}$ row and column. Software for simulating the matrix exponential of such a large, sparse matrix is given by Sidje \cite{Sidje1998}, as well as by Burrage, Hegland, Macnamara, and Sidje \cite{Burrage2006}.

In addition to investigating numerical performance, we also use the Magnus schemes to simulate sample solutions of \cref{Equation3}, which we now demonstrate. In \Cref{C}, we show simulation of the SPDDE solution, for the deterministic case with no noise (first column), simulation with uncorrelated noise (second column), and simulation using correlated noise (third column). For each of these cases, we show the numerical SPDDE solution (first row), as well as a heat map of the same solution (second row), where we indicate the cross sections $x=0.7$ and $t=1.25$ with dashed lines. At the cross sections at $x=0.7$ (third row) and $t=1.25$ (fourth row), we show both the SPDDE solution (with delay $\tau=1$) superimposed with the solution of the same equation but without delay ($\tau=0$). In the deterministic case (the first column), we are able to view the effects of time-delayed cooling (of the heat equation), compared with instantaneous cooling. For example, from \Cref{C} (g), we see that instantaneous cooling leads to monotonically decaying (in time) solution (along the $x=0.7$ cross section) as it approaches equilibrium, while if the cooling is delayed then the solution both approaches the equilibrium temperature and is also cooled beyond this point. Our selection of parameters cause the solution still to converge to equilibrium, but with sufficiently large parameters (such as $r_a$ and $r_b$), it is possible for the delayed solution to diverge. Although we do not discuss this topic further in this paper, the study of stability for these solutions may form future research. In the stochastic settings of this example, we remark that our delay is only present in the drift term of \cref{Equation3}, so the stochasticity of the solutions is only enhanced by the immediate value ($cU(t,x)$, in \cref{Equation3}), but if we include delay dependence in the volatility terms then this also alters the solution trajectories, and may also affect stability. In our setting, from comparing the affects of uncorrelated and correlated noise, the solutions converge in time similarly to the deterministic case (third row). At a fixed time (fourth row), uncorrelated noise yields clearly uncorrelated values at distinct spatial points, while spatial correlation in the noise causes neighboring values to cluster more tightly.

As a final comment on this SPDDE application, we remark that our investigation is only a beginning into the study of SPDDE simulations, and future work on simulating SPDDE solutions may compare computation times between the Taylor and Magnus schemes. The Magnus methods provide a simulation technique that is applicable for fine spatial discretization, but the tradeoff is the computation time required to simulate approximations of large matrix exponentials. That is, when simulating the SPDDE of \hyperref[Example4]{Example~\ref*{Example4}} with small $\Delta x$, the EM scheme requires significantly smaller temporal step size, due to \cref{StabilityCondition}, but for the MEM scheme, a small $\Delta x$ leads to larger matrix coefficients. In order to implement the Magnus schemes, we must simulate the exponentials of these large matrices.

%
%

\end{example}

\section{Conclusions}\label{Section5}


In this paper, we introduced Magnus numerical schemes to simulate SDDE solutions. We constructed these schemes by extending existing SDE Magnus techniques to include constant time delays. A key motivation for this extension was the simulation of SPDDEs such as the stochastic heat equation with delayed cooling \cref{Equation3}. This contrasts with the traditional Euler–Maruyama scheme,
which is unsuitable for simulating \cref{Equation3} due to the restrictive stability condition \cref{StabilityCondition}.


We presented the Magnus schemes in \cref{Section2}, alongside the regular Taylor (EM and Milstein) schemes. We produced our Magnus schemes in \cref{Section3}, by extending the construction of SDE Magnus schemes by Yang et al.\ \cite{BurrageBurrage2021} to the case of SDDEs. This involved decomposing the solution $X$ of the semilinear SDDE \cref{Equation1} into the product $X=X_HX_R$, where $X_H$ solves the homogeneous linear SDE \cref{dXHd}, and $X_R$ satisfies the inhomogeneous integral equation \cref{XRd}. We then applied Magnus expansions to approximate $X_H$, and stochastic Taylor expansions to approximate $X_R$. In \cref{Section4}, we demonstrated that the proposed schemes achieve their theoretical orders of convergence across a range of examples, including linear and semilinear SDDEs, along with an SPDDE model \cref{Equation3}.


Our Magnus schemes find purpose in specific situations, such as the SPDDE problem in \hyperref[Example4]{Example~\ref*{Example4}}. In low-dimensional problems (\hyperref[Example1]{Examples~\ref*{Example1}}, \ref{Example2}, \ref{Example3}), their accuracy is comparable to the Taylor schemes, but they incur greater computational cost due to the matrix exponential evaluations. Consequently, when both approaches are stable and applicable, Taylor methods are typically more efficient. However, as shown in \hyperref[Example4]{Example~\ref*{Example4}}, the Magnus methods remain stable and accurate even when the Taylor methods fail, such as in simulations requiring large spatial dimension or coarse temporal discretization.


Future research directions include a more detailed comparison of computational costs between Magnus and Taylor schemes, particularly as spatial dimension increases. For example, in \hyperref[Example4]{Example~\ref*{Example4}}, when the stability condition \cref{StabilityCondition} is satisfied using a sufficiently small time step $h$, it would be valuable to determine whether Magnus schemes remain computationally feasible. For larger $h$, the Magnus methods offer a practical means to simulate SPDDEs with fine spatial discretization. Although we considered only one SPDDE model here, the approach could be extended to more general stochastic parabolic PDEs with delays. The computation times of the Magnus methods for SPDDEs will also improve as methods for simulating matrix exponentials improve, whether it be from development of analytic techniques or improvements in computer processor speeds.

\section*{Acknowledgments}
We thank the anonymous reviewers for their constructive comments and helpful suggestions, which improved the quality of this manuscript.

\appendix


\section{Calculation results}\label{AppendixProofsforConvergenceAnalysis}

This appendix states technical lemmas used throughout the convergence analysis. Full proofs are provided in the supplementary material; see \cref{ProofofNewLemma2,ProofofLemmaPhiInverseLemma,ProofofNewLemma1,ProofofNewLemma3}. These results support the error expansion and convergence arguments for the MEM and MM schemes introduced in Section~\ref{Section3}.

\begin{lemma}\label{LemmaPhiInverse}
Suppose $\sigma_b\leqslant t_l<t_{l+1}\leqslant\sigma_{b+1}$. The SDE
\begin{align}
\mathrm{d}Z_{t_l}(t)&=\bigg(-A_0+\sum_{i=1}^mA_i^2\bigg)Z_{t_l}(t)\,\mathrm{d}t+\sum_{j=1}^m(-A_j)Z_{t_l}(t)\,\mathrm{d}W_j(t),\quad t\in[t_l,t_{l+1}],\label{PhiInverseSDE}
\end{align}
with $Z_{t_l}(t_l)=I_d$, has solution $Z_{t_l}(t)=\exp(-\Omega(t_l,t))$.
\end{lemma}

\begin{lemma}\label{NewLemma1}
For each $i,j=1,\ldots,m$, there is a constant $c>0$ such that the following inequalities hold for any times $t_l,t_{l+1}\in[\sigma_b,\sigma_{b+1}]$ with $t_l < t_{l+1}$:
\begin{align}
\mathbb{E}\bigg[\abs{\int_{t_l}^{t_{l+1}}\int_{t_l}^sg_j(u,X(u),\phi_{\sigma_b}(u))\,\mathrm{d}u\,\mathrm{d}W_j(s)}^2\bigg]&\leqslant c\,h_l^2;\label{NewLemma1Bound5}\\
\mathbb{E}\bigg[\bigg|\int_{t_l}^{t_{l+1}}\int_{t_l}^sA_j\bigg([A_iX(t_l)+g_i(t_l,X(t_l),\phi_{\sigma_b}(t_l))]\hspace{1.5cm}&\label{NewLemma1Bound6}\\
-[A_iX(u)+g_i(u,X(u),\phi_{\sigma_b}(u))]\bigg)\,\mathrm{d}W_i(u)\,\mathrm{d}W_j(s)\bigg|^2\bigg]&\leqslant c\,h_l^2;\quad\textrm{and}\nonumber\\
\mathbb{E}\bigg[\abs{\int_{t_l}^{t_{l+1}}\int_{t_l}^sA_j[A_0X(u)+f(u,X(u),\phi_{\sigma_b}(u))]\,\mathrm{d}u\,\mathrm{d}W_j(s)}^2\bigg]&\leqslant c\,h_l^2,\label{NewLemma1Bound7}
\end{align}
where $h_l=t_{l+1}-t_l$.
\end{lemma}

\begin{lemma}\label{NewLemma2}
There exists a constant $c>0$ such that the integrals in \cref{Isum} satisfy
\begin{align}
I^{(1)}(t_l,t_{l+1})&=R_1+\sum_{j=1}^mg_j(t_l,X(t_l),\phi_{\sigma_b}(t_l))\,\Delta W_j\label{I1}\\
&\quad+\sum_{i=1}^m\sum_{j=1}^m\bigg(\nabla_xg_j(t_l,X(t_l),\phi_{\sigma_b}(t_l))\nonumber\\
&\quad\quad\hspace{1cm}\cdot[A_iX(t_l)+g_i(t_l,X(t_l),\phi_{\sigma_b}(t_l))]\bigg)\,I_{ij}(t_l,t_{l+1})\nonumber\\
&\quad+\sum_{k=1}^K\mathbb{I}(t_l\geqslant\tau_k)\sum_{i=1}^m\sum_{j=1}^m\bigg(\nabla_{x_{\tau_k}}g_j(t_l,X(t_l),\phi_{\sigma_b}(t_l))\nonumber\\
&\quad\quad\hspace{1cm}\cdot[A_iX(t_l^{\tau_k})+g_i(t_l^{\tau_l},X(t_l^{\tau_l}),\phi_{\sigma_b}(t_l^{\tau_l}))]\bigg)\,I_{ij}^{\tau_k}(t_l,t_{l+1}),\nonumber\\
I^{(2)}(t_l,t_{l+1})&=\sum_{j=1}^m\int_{t_l}^{t_{l+1}}\left(\int_{t_l}^s\underline{A}_0Z_{t_l}(u)\,\mathrm{d}u\right)g_j(s,X(s),\phi_{\sigma_b}(s))\,\mathrm{d}W_j(s)+R_2,\label{I2}\\
I^{(3)}(t_l,t_{l+1})&=-\sum_{i=1}^m\sum_{j=1}^mA_iZ(t_l)g_j(t_l,X(t_l),\phi_{\sigma_b}(t_l))\,I_{ij}(t_l,t_{l+1})+R_3,\label{I3}
\end{align}
with remainders $R_1,R_2,R_3$ satisfying the bounds $\mathbb{E}[\abs{R_1}^2]\leqslant c\,h_l^{3/2}$, $\mathbb{E}[\abs{R_2}^2]\leqslant c\,h_l^2,$ and $\mathbb{E}[\abs{R_3}^2]\leqslant c\,h_l^2$, where $h_l=(t_{l+1}-t_l)$.
\end{lemma}

\begin{lemma}\label{NewLemma3}
For $\sigma_b\leqslant t_n<t_{n+1}\leqslant\sigma_{b+1}$, the conditional expectations of $X_H$, $\exp\left(\Omega^{[1]}\right)$, and $\exp\left(\Omega^{[2]}\right)$, at $t_{n+1}$, given $X_H(t_n)$, satisfy
\begin{align}
&\mathbb{E}[X_H(t_{n+1})|X_H(t_n)]=X_H(t_n)\exp(A_0\,h_n),\label{EXHconditional}\\
&\mathbb{E}\left[\exp\left(\Omega^{[k]}(t_n,t_{n+1})\right)\!\mid\!\exp\left(\Omega^{[k]}(t_n,t_n)\right)\!=\!I_d\right]=I_d+A_0\,h_n+\frac{\hat{A}_0^2\,h_n^2}{2}+R_k,\label{EOmega1conditional}
\end{align}
for both $k=1,2$, where $\hat{A}_0=A_0-\frac{1}{2}\sum_{j=1}^mA_j^2$, where each $R_k$ denotes a (constant) matrix multiple of $h_n^2$ or higher-order terms.
\end{lemma}



\section{Karhunen--Lo\`eve results}\label{AppendixSPDE}

This section states the Karhunen--Lo\`eve Theorem applied to decompose a $Q$-Wiener process as a countable combination of independent Wiener processes. These results are applied extensively in \hyperref[Example4]{Example~\ref*{Example4}}. The results \cref{QWienerProcess}, \cref{QKLTheorem}, and \cref{ItoIntegralofQWienerProcess}, below, are respectively from pages 436, 437, and 445, of the text by Lord, Powell, and Shardlow \cite{LordPowellShardlow2014}, restricted to the one-dimensional spatial domain $[0,1]$.

\begin{definition}\label{QWienerProcess}A stochastic field $W^c=(W^c(t,x))_{(t,x)\in[0,T]\times[0,1]}$ is a \emph{$Q$-Wiener process} if:
\begin{enumerate}
  \item For each $x\in[0,1]$, $W^c(0,x)=0$, almost surely.
  \item For each $x\in[0,1]$, $W^c(t,x)$ is continuous with $t$, almost surely.
  \item $W$ is an adapted process and, for each $x\in[0,1]$ and any $0\leqslant s<t\leqslant u<v$, $W^c(t,x)-W^c(s,x)$ and $W^c(v,x)-W^c(u,x)$ are independent, and
  \item There exists a covariance function $Q:[0,1]\times[0,1]\to[0,\infty)$ such that for all $0\leqslant s<t$ and $x\in[0,1]$, $W^c(t,x)-W^c(s,x)\sim\mathcal{N}(0,(t-s)Q(x,x))$.
\end{enumerate}
\end{definition}

\begin{theorem}\label{QKLTheorem}Let $Q:[0,1]\times[0,1]\to[0,\infty)$ be a (spatial) covariance function and let $W^c$ be a $Q$-Wiener process, where the eigenvalues and corresponding eigenfunctions of $Q$ are $(\lambda_j,\phi_j)_{j\in\mathbb{N}}$, so that for each $j\in\mathbb{N}$,
\begin{align*}
\int_0^1&Q(x,y)\,\phi_j(y)\,\mathrm{d}y=\lambda_j\phi_j(x),\quad x\in[0,1],\\
&Q(x,y)=\sum_{j=1}^\infty\lambda_j\phi_j(x)\phi_j(y),\quad(x,y)\in[0,1]\times[0,1].
\end{align*}
Almost surely, the process $W^c$ satisfies
\begin{equation}W^c(t,x)=\sum_{j=1}^\infty\sqrt{\lambda_j}\phi_j(x)W_j(t),\quad(t,x)\in[0,\infty)\times[0,1],\label{QWienerFullExpansion}\end{equation}
where $W_1,W_2,\ldots$ are independent, adapted Wiener processes.
\end{theorem}

\begin{definition}\label{ItoIntegralofQWienerProcess}Using the notation of \cref{QWienerProcess} and \cref{QKLTheorem}, the \emph{It\^o integral} of $U=(U(t,x))_{(t,x)\in[0,\infty)\times[0,1]}$, with respect to $W^c$, is defined as
\begin{equation}\int_0^tU(s,x)\,\mathrm{d}W^c(s)=\sum_{j=1}^\infty\int_0^tU(s,x)\sqrt{\lambda_j}\phi_j(x)\,\mathrm{d}W_j(s),\quad(t,x)\in[0,\infty)\times[0,1].\label{UItoIntegral}\end{equation}
\end{definition}

\bibliographystyle{siamplain}
\bibliography{Bibliography}

\pagebreak

\begin{center}
\begin{Huge}\textbf{SUPPLEMENTARY\\MATERIAL}\end{Huge}
\end{center}

\section[Alternative proof of existence-uniqueness]{Alternative proof of \Cref{PLS}}\label{AlternativeProofofPLS}

\begin{proof}
Kloeden and Shardlow \cite{KloedenShardlow2012} provide an existence-uniqueness result for the continuous-delays SDDE
\begin{equation}\mathrm{d}X(t)=\underline{f}(t,(X(s))_{s\in[-\tau,t]})\,\mathrm{d}t+\sum_{j=1}^m\underline{g}_j(t,(X(s))_{s\in[-\tau,t]})\,\mathrm{d}W_j(t),\;\,t\in[0,T],\label{SDEwithContinuousDelayFunctions}\end{equation}
with functions $\underline{f},\underline{g}_1,\ldots,\underline{g}_m$ defined at time $t$ and also the entire histories of stochastic processes on $[-\tau,t]$. The existence-uniqueness result for \cref{SDEwithContinuousDelayFunctions} reduces to the same result as ours by restricting the dependence on the past path $(X(s))_{s\in[-\tau,t]}$ to only $X(s)$ at the times $s=t,t-\tau_1,\ldots,t-\tau_K$.
\end{proof}

\section[Proof of Milstein SDDE expansion]{Proof of \Cref{ItoTaylorKuchlerPlaten}}\label{ProofofItoTaylorKuchlerPlaten}

\begin{proof}
This result follows by applying the expansion by K\"uchler and Platen \cite{KuchlerPlaten2000}, who applied the stochastic Taylor expansion (see, for example, Kloeden and Platen \cite{KloedenPlaten1999}) to the single-delay equation
\begin{align}
\mathrm{d}X(t)&=f(t,X(t),X(t-\tau))\,\mathrm{d}t+\sum_{j=1}^mg_j(t,X(t),X(t-\tau))\,\mathrm{d}W_j(t),\quad t\in[0,T],\label{SDDEKis1}\\
X(t)&=\phi(t),\quad t\in[-\tau,0].\nonumber
\end{align}
We summarize this proof for the case of \cref{SDDEKis1}, before then extending to $K$ delays. For each Bellman interval $[\alpha\tau,(\alpha+1)\tau]$ (for nonnegative integer $\alpha$), the It\^o expansion is applied to the system of equations (see equation (10.6) in \cite{KuchlerPlaten2000})
\begin{align}
X(t_{n+1}-q\tau)&=X(t_n-q\tau)+\int_{t_n}^{t_{n+1}}f(s-q\tau,X(s-q\tau),X(s-(q+1)\tau))\,\mathrm{d}s\nonumber\\
&\quad+\sum_{j=1}^m\int_{t_n}^{t_{n+1}}f(s-q\tau,X(s-q\tau),X(s-(q+1)\tau))\,\mathrm{d}W_j(s-q\tau),\label{XtnqSystem}
\end{align}
for each $q=0,\ldots,\alpha$, from which it follows that $X$ is of the form
\begin{align}
X(t_{n+1})&=X(t_n)+f(t_n,X(t_n),X(t_n^{\tau}))\,h_n+\sum_{j=1}^mg_j(t_n,X(t_n),X(t_n^{\tau}))\,\Delta W_j(t_n,t_{n+1})\nonumber\\
&\;+\sum_{i=1}^m\sum_{j=1}^m\nabla_xg_j(t_n,X(t_n),X(t_n^{\tau}))g_i(t_n,X(t_n),X(t_n^{\tau}))\,I_{ij}(t_n,\!t_{n+1})+R\label{DelayTaylorExpansionKis1}\\
&\;+\mathbb{I}(t_n\!\geqslant\!\tau)\!\sum_{i=1}^m\sum_{j=1}^m\nabla_{x_{\tau}}g_j(t_n,X(t_n),X(t_n^{\tau}))g_i(t_n^{\tau},X(t_n^{\tau}),X(t_n^{\tau,\tau}))\,I_{ij}^{\tau}(t_n,\!t_{n+1}),\nonumber
\end{align}
where $R$ contains only multiples of triple or higher-degree integrals, whose integrators are all the Wiener processes, $W_j$, $j=1,\ldots,m$, or the delayed Wiener processes, $W_j^{\alpha_1\tau_1}$, defined by \[W_j^{\alpha\tau}(t)=W_j(t-\alpha\tau),\quad t\in[\alpha\tau,T],\quad\alpha=1,2,\ldots.\]
The condition $h_n<\tau_\mathrm{min}=\tau$ combined with the independent increments of the Wiener processes means that all of
\[W_j,W_j^{\alpha\tau},\quad j=1,\ldots,m,\quad\alpha=1,2,\ldots\]
are independent processes, from which we calculate $\mathbb{E}[\abs{R}^2]<ch_n^{3/2}$ for constant $c>0$. In the case of \cref{Equation1} with $K$ delays, in the Bellman interval $[\sigma_b,\sigma_{b+1}]$, $\sigma_b=\alpha_k\tau_k$ for some $\alpha_k=0,1,\ldots,N_k-1$ and $k=1,\ldots,K$, so that by forming the system
\begin{align}
X(t_{n+1}-q\tau_k)&=X(t_n-q\tau_k)+\int_{t_n}^{t_{n+1}}f(s-q\tau_k,X(s-q\tau_k),\phi_{\sigma_b}(s-(q+1)\tau_k))\,\mathrm{d}s\nonumber\\
&\;+\sum_{j=1}^m\int_{t_n}^{t_{n+1}}f(s-q\tau_k,X(s-q\tau_k),\phi_{\sigma_b}(s-(q+1)\tau_k))\,\mathrm{d}W_j(s-q\tau_k),\label{XtnqSystemMultipleDelays}
\end{align}
($q=0,\ldots,\alpha_k$) in place of \cref{XtnqSystem}, applying \cref{DelayTaylorExpansionKis1} to \cref{XtnqSystemMultipleDelays} gives \cref{DelayTaylorExpansion} for $t_{n+1}<(\alpha_k+1)\tau_k$.

In order to extend the result to the whole Bellman interval $[\sigma_b,\sigma_{b+1}]$ rather than $[\sigma_b,\sigma_b+\tau_k]=[\alpha_k\tau_k,(\alpha_k+1)\tau_k]$, we consider separately the cases of whether or not $[\sigma_b,\sigma_{b+1}]$ is contained in $[\sigma_b,\sigma_b+\tau_k]$, and if not then we extend the system \cref{XtnqSystemMultipleDelays} further ahead, for $q>\alpha_k$, until the time $\sigma_{b+1}$ is passed.

If $\sigma_{b+1}\leqslant\sigma_b+\tau_k$ then the result is already shown. However, if $(\alpha_k+1)\tau_k<\sigma_{b+1}$ then we append to the system \cref{XtnqSystemMultipleDelays} the equations
\begin{align}
X(t_{n+1}-q\tau_k)&=X(t_n-q\tau_k)+\int_{t_n}^{t_{n+1}}f(s-q\tau_k,X(s-q\tau_k),\phi_{\sigma_b}(s-(q+1)\tau_k))\,\mathrm{d}s\nonumber\\
&\;+\sum_{j=1}^m\int_{t_n}^{t_{n+1}}f(s-q\tau_k,X(s-q\tau_k),\phi_{\sigma_b}(s-(q+1)\tau_k))\,\mathrm{d}W_j(s-q\tau_k),\label{XtnqSystemExtended}
\end{align}
($q=\alpha_k+1,\ldots,\alpha_{k'}$) where $\alpha_{k'}$ is an integer that satisfies $\sigma_{b+1}\leqslant\alpha_{k'}\tau_k$. We then use \cref{XtnqSystemExtended} to derive \cref{DelayTaylorExpansion} for all $\sigma_b\leqslant t_n\leqslant t_{n+1}\leqslant\alpha_{k'}\tau_k$, and in particular, for $\sigma_b\leqslant t_n\leqslant t_{n+1}\leqslant \sigma_{b+1}$.
\end{proof}

\section[Proof of exp(-Omega) SDE]{Proof of \cref{LemmaPhiInverse}}\label{ProofofLemmaPhiInverseLemma}

\begin{proof}
\Cref{PhiInverseSDE} is of the form
\begin{equation}\mathrm{d}Z_{t_l}(t)=\underline{A}_0Z_{t_l}(t)\,\mathrm{d}t+\sum_{j=1}^m\underline{A}_jZ_{t_l}(t)\,\mathrm{d}W_j(t),\label{Zoft}\end{equation}
where we recall $\underline{A}_0=\left(-A_0+\sum_{j=1}^mA_j^2\right)$ and $\underline{A}_j=-A_j$ for $j=1,\ldots,m$. Using \cref{Omega}, \cref{Zoft} admits the Magnus expansion
\begin{align}
Z_{t_l}(t)&=\exp\bigg(\bigg(\underline{A}_0-\frac{1}{2}\sum_{j=1}^m\underline{A}_j^2\bigg)\,(t-t_l)+\sum_{j=1}^m\underline{A}_j\,\Delta W_j(t_l,t)\label{ZMagnusForm}\\
&\hspace{2cm}+\frac{1}{2}\sum_{i=0}^m\sum_{j=i+1}^m[\underline{A}_i,\underline{A}_j](I_{ji}(t_l,t)-I_{ij}(t_l,t))+R\bigg)\nonumber\\
&=\exp\bigg(-\bigg(A_0-\frac{1}{2}\sum_{j=1}^mA_j^2\bigg)\,(t-t_l)-\sum_{j=1}^mA_j\,\Delta W_j(t_l,t)\nonumber\\
&\hspace{2cm}-\frac{1}{2}\sum_{i=0}^m\sum_{j=i+1}^m[A_i,A_j](I_{ji}(t_l,t)-I_{ij}(t_l,t))-R\bigg),\nonumber
\end{align}
where $R$ is a matrix multiple of third-degree (and higher) iterated integrals. The equality \cref{ZMagnusForm} is shown by calculating
\begin{align*}
\frac{1}{2}\sum_{i=0}^m\sum_{j=i+1}^m[\underline{A}_i,\underline{A}_j](I_{ji}(t_l,t)-I_{ij}(t_l,t))&=\frac{1}{2}\sum_{i=0}^m\sum_{j=0}^m[\underline{A}_i,\underline{A}_j]I_{ji}(t_l,t)\\
&=-\frac{1}{2}\sum_{i=0}^m\sum_{j=i+1}^m[A_i,A_j](I_{ji}(t_l,t)-I_{ij}(t_l,t)),
\end{align*}
and similar for higher-degree terms. From \cref{Omega}, the expansion \cref{ZMagnusForm} is the matrix exponential of $-\Omega(t_l,t)$, so that $Z_{t_l}(t)=\exp(-\Omega(t_l,t))$.
\end{proof}

\section[Proof of inequalities lemma]{Proof of \cref{NewLemma1}}\label{ProofofNewLemma1}

\begin{proof}
We first show the inequalities
\begin{align}
\mathbb{E}\left[\abs{g_j(t,X(t),\phi_{\sigma_b}(t))}^2\right]&\leqslant c,\label{NewLemma1Bound1}\\
\mathbb{E}\left[\abs{X(t)}^4\right]&\leqslant c,\label{NewLemma1Bound2}\\
\mathbb{E}\left[\abs{g_j(t,X(t),\phi_{\sigma_b}(t))}^4\right]&\leqslant c,\quad\textrm{and}\label{NewLemma1Bound3}\\
\mathbb{E}\left[\abs{X(t_{l+1})-X(t_l)}^2\right]&\leqslant c\,h_l,\label{NewLemma1Bound4}
\end{align}
for $t\in[\sigma_b,\sigma_{b+1}]$. We then use these to show \Cref{NewLemma1Bound5,NewLemma1Bound6,NewLemma1Bound7}. By \cref{PLS}, $\mathbb{E}[\abs{X(t)}^2]$ is bounded, for all $t$. Combining this with the linear-growth bound \cref{LinearGrowthBound} shows that
\begin{equation*}
\mathbb{E}\left[\abs{g_j(t,X(t),\phi_{\sigma_b}(t))}^2\right]\leqslant\mathbb{E}\bigg[L_j^{(1)}\bigg(1+\abs{X(t)}^2+\sum_{k=1}^K\abs{X(t-\tau_k)}^2\bigg)\bigg]
\end{equation*}
is bounded, establishing \cref{NewLemma1Bound1}. From our assumption that $\mathbb{E}[\abs{\phi(t)}^4]<\infty$ for all $t\in[-\tau,0]$, it follows that there is a constant $c>0$ such that
\[\mathbb{E}\left[\abs{X(t)}^4\right]\leqslant\left(1+\mathbb{E}\left[\abs{\phi(0)}^4\right]\right)\exp(c(T+\tau)),\quad t\in[\sigma_b,\sigma_{b+1}],\]
from \cite[page 136]{KloedenPlaten1999}, which shows \cref{NewLemma1Bound2}. We then calculate
\begin{align*}
\mathbb{E}\left[\abs{g_j(t,X(t),\phi_{\sigma_b}(t))}^4\right]\leqslant\left(L_j^{(1)}\right)^2&\bigg(1+2(K+1)\sup_{t\in[-\tau,T]}\mathbb{E}\left[\abs{X(t)}^2\right]\\
&\hspace{1cm}+(K+1)\sup_{t\in[-\tau,T]}\mathbb{E}\left[\abs{X(t)}^4\right]\bigg),
\end{align*}
which is bounded and shows \cref{NewLemma1Bound3}. Using \cref{NewLemma1Bound1}, it follows that
\begin{align*}
\mathbb{E}\left[\abs{X(t_{l+1})-X(t_l)}^2\right]&=\mathbb{E}\bigg[\bigg|\int_{t_l}^{t_{l+1}}[A_0X(u)+f(u,X(u),\phi_{\sigma_b}(u))]\,\mathrm{d}u\\
&\quad\quad+\sum_{j=1}^m\int_{t_l}^{t_{l+1}}[A_jX(u)+g_j(u,X(u),\phi_{\sigma_b}(u))]\,\mathrm{d}W_j(u)\bigg|^2\bigg]\\
&\leqslant\int_{t_l}^{t_{l+1}}\mathbb{E}\left[\abs{A_0X(u)+f(u,X(u),\phi_{\sigma_b}(u))}^2\right]\,\mathrm{d}u\\
&\quad\quad+\sum_{j=1}^m\int_{t_l}^{t_{l+1}}\mathbb{E}\left[\abs{A_jX(u)+g_j(u,X(u),\phi_{\sigma_b}(u))}^2\right]\,\mathrm{d}u\\
&\leqslant c_0(t_{l+1}-t_l)+\sum_{j=1}^mc_j(t_{l+1}-t_l),
\end{align*}
for constants $c_0,\ldots,c_m>0$. \Cref{NewLemma1Bound4} follows, with $c=c_0+\cdots+c_m$.

By applying the It\^o isometry, the Cauchy--Schwarz inequality, swapping the expectation and the integrals with the Fubini--Tonelli Theorem, and then using \cref{NewLemma1Bound1}, it follows that
\begin{align*}
\mathbb{E}&\bigg[\abs{\int_{t_l}^{t_{l+1}}\left(\int_{t_l}^sg_j(u,X(u),\phi_{\sigma_b}(u))\,\mathrm{d}u\right)\,\mathrm{d}W_j(s)}^2\bigg]\\
&\leqslant\int_{t_l}^{t_{l+1}}\int_{t_l}^s\mathbb{E}\left[\abs{g_j(u,X(u),\phi_{\sigma_b}(u))}^2\right]\,\mathrm{d}u\,\mathrm{d}s\leqslant c(t_{l+1}-t_l)^2,
\end{align*}
for constant $c>0$, showing \cref{NewLemma1Bound5}. By applying the It\^o isometry twice, along with the spatial Lipschitz bound \cref{LipschitzBound},
\begin{align*}
\mathbb{E}&\bigg[\bigg|\int_{t_l}^{t_{l+1}}\int_{t_l}^sA_j\big([A_iX(t_l)+g_i(t_l,X(t_l),\phi_{\sigma_b}(t_l))]\\
&\quad\quad\quad\quad-[A_iX(u)+g_i(u,X(u),\phi_{\sigma_b}(u))]\big))\,\mathrm{d}W_i(u)\,\mathrm{d}W_j(s)\bigg|^2\bigg]\\
&\leqslant\int_{t_l}^{t_{l+1}}\mathbb{E}\bigg[\int_{t_l}^s\abs{A_jA_i(X(t_l)-X(u))}^2\,\mathrm{d}u\\
&\quad\quad\quad\quad+\int_{t_l}^s\abs{A_j(g_i(t_l,X(t_l),\phi_{\sigma_b}(t_l))-g_i(u,X(u),\phi_{\sigma_b}(u)))}^2\,\mathrm{d}u\bigg]\,\mathrm{d}s\\
&\leqslant\int_{t_l}^{t_{l+1}}\int_{t_l}^s\mathbb{E}\bigg[\abs{A_jA_i(X(t_l)-X(u))}^2\\
&\quad\quad\quad\quad+\abs{A_jL_j^{(2)}\bigg(\abs{X(t_1)-X(u)}+\sum_{k=1}^K\abs{X(t_l-{\tau_k})-X(u-{\tau_k})}\bigg)}^2\bigg]\,\mathrm{d}u\,\mathrm{d}s\\
&\leqslant c(t_{l+1}-t_l)^2,
\end{align*}
for constant $c>0$, showing \cref{NewLemma1Bound6}. Similarly to showing \cref{NewLemma1Bound5}, we calculate
\begin{align*}
\mathbb{E}&\bigg[\abs{\int_{t_l}^{t_{l+1}}\int_{t_l}^sA_j[A_0X(u)+f(u,X(u),\phi_{\sigma_b}(u))]\,\mathrm{d}u\,\mathrm{d}W_j(s)}^2\bigg]\\
&\leqslant\int_{t_l}^{t_{l+1}}\int_{t_l}^s\mathbb{E}\left[\abs{A_j[A_0X(u)+f(u,X(u),\phi_{\sigma_b}(u))]}^2\right]\,\mathrm{d}u\,\mathrm{d}s\leqslant c(t_{l+1}-t_l)^2,
\end{align*}
for constant $c>0$, showing \cref{NewLemma1Bound7}.
\end{proof}

\section[Proof of integral expressions]{Proof of \cref{NewLemma2}}\label{ProofofNewLemma2}

\begin{proof}
In order to simplify notation, we find it convenient to write $g_j(s)=g_j(s,X(s),\phi_{\sigma_b}(s))$ for each $j=0,1,\ldots,m$, where $g_0=f$. The first integral of \cref{Isum} is
\begin{align*}
I^{(1)}(t_l,t_{l+1})&=\sum_{j=1}^m\int_{t_l}^{t_{l+1}}g_j(s,X(s),\phi_{\sigma_b}(s))\,\mathrm{d}W_j(s)\\
&=\int_{t_l}^{t_{l+1}}[A_0X(s)+f(s,X(s),\phi_{\sigma_b}(s))]\,\mathrm{d}s\\
&\quad+\sum_{j=1}^m\int_{t_l}^{t_{l+1}}[A_jX(s)+g_j(s,X(s),\phi_{\sigma_b}(s))]\,\mathrm{d}W_j(s)\\
&\quad-\int_{t_l}^{t_{l+1}}[A_0X(s)+f(s,X(s),\phi_{\sigma_b}(s))]\,\mathrm{d}s-\sum_{j=1}^m\int_{t_l}^{t_{l+1}}A_jX(s)\,\mathrm{d}W_j(s).
\end{align*}
By applying the Taylor expansion, \cref{ItoTaylorKuchlerPlaten}, we calculate
\begin{align}
I^{(1)}&(t_l,t_{l+1})=X(t_{l+1})-X(t_l)\nonumber\\
&\hspace{1.5cm}-\int_{t_l}^{t_{l+1}}[A_0X(s)+f(s,X(s),\phi_{\sigma_b}(s))]\,\mathrm{d}s-\sum_{j=1}^m\int_{t_l}^{t_{l+1}}A_jX(s)\,\mathrm{d}W_j(s)\nonumber\\
&=[A_0X(t_l)+f(t_l,X(t_l),\phi_{\sigma_b}(t_l))]\,(t_{l+1}-t_l)\nonumber\\
&\quad+\sum_{j=1}^m[A_jX(t_l)+g_j(t_l,X(t_l),\phi_{\sigma_b}(t_l))]\,\Delta W_j(t_l,t_{l+1})\nonumber\\
&\quad+\sum_{i=1}^m\sum_{j=1}^m[A_j+\nabla_xg_j(t_l,X(t_l),\phi_{\sigma_b}(t_l))]\nonumber\\
&\quad\quad\cdot[A_iX(t_l)+g_i(t_l,X(t_l),\phi_{\sigma_b}(t_l))]\,I_{ij}(t_l,t_{l+1})\nonumber\\
&\quad+\sum_{k=1}^K\mathbb{I}(t_l\geqslant\tau_k)\sum_{i=1}^m\sum_{j=1}^m\nabla_{x_{\tau_k}}g_j(t_l,X(t_l),\phi_{\sigma_b}(t_l))\nonumber\\
&\quad\quad\cdot[A_iX(t_l^{\tau_k})+g_i(t_l^{\tau_k},X(t_l^{\tau_k}),\phi_{\sigma_b}(t_l^{\tau_k}))]\,I_{ij}^{\tau_k}(t_l,t_{l+1})\nonumber\\
&\quad+R^{(1)}-\int_{t_l}^{t_{l+1}}[A_0X(s)+f(s,X(s),\phi_{\sigma_b}(s))]\,\mathrm{d}s-\sum_{j=1}^m\int_{t_l}^{t_{l+1}}A_jX(s)\,\mathrm{d}W_j(s)\nonumber\\
&=\sum_{j=1}^mg_j(t_l,X(t_l),\phi_{\sigma_b}(t_l))\,\Delta W_j\nonumber\\
&\quad+\sum_{i=1}^m\sum_{j=1}^m\nabla_xg_j(t_l,X(t_l),\phi_{\sigma_b}(t_l))[A_iX(t_l)+g_i(t_l,X(t_l),\phi_{\sigma_b}(t_l))]\,I_{ij}(t_l,t_{l+1})\nonumber\\
&\quad+\sum_{k=1}^K\mathbb{I}(t_l\geqslant\tau_k)\sum_{i=1}^m\sum_{j=1}^m\nabla_{x_{\tau_k}}g_j(t_l,X(t_l),\phi_{\sigma_b}(t_l))\nonumber\\
&\quad\quad\cdot[A_iX(t_l^{\tau_k})+g_i(t_l^{\tau_k},X(t_l^{\tau_k}),\phi_{\sigma_b}(t_l^{\tau_k}))]\,I_{ij}^{\tau_k}(t_l,t_{l+1})\nonumber\\
&\quad+R^{(1)}+F+G,\nonumber
\end{align}
where $\mathbb{E}[\abs{R^{(1)}}^2]<c(t_{l+1}-t_l)^{3/2}$ for constant $c>0$,
\begin{align*}
F&=f(t_l,X(t_l),\phi_{\sigma_b}(t_l))\,(t_{l+1}-t_l)+A_0X(t_l)\,(t_{l+1}-t_l)\\
&\quad-\int_{t_l}^{t_{l+1}}[A_0X(s)+f(s)]\,\mathrm{d}s,\quad\textrm{and}\\
G&=-\sum_{j=1}^m\int_{t_l}^{t_{l+1}}A_jX(s)\,\mathrm{d}W_j(s)+\sum_{j=1}^mA_jX(t_l)\,\Delta W_j\nonumber\\
&\quad+\sum_{i=1}^m\sum_{j=1}^mA_j[A_iX(t_l)+g_i(t_l,X(t_l),\phi_{\sigma_b}(t_l))]\,I_{ij}(t_l,t_{l+1}).
\end{align*}
Using the Lipschitz bound \cref{LipschitzBound} and \cref{NewLemma1}, we calculate
\begin{align*}
\mathbb{E}\left[\abs{F}^2\right]&\leqslant\mathbb{E}\bigg[\int_{t_l}^{t_{l+1}}\abs{A_0(X(t_l)-X(s))}^2\\
&\quad\quad+\abs{f(t_l,X(t_l),\phi_{\sigma_b}(t_l))-f(s,X(s),\phi_{\sigma_b}(s))}^2\,\mathrm{d}s\bigg]\\
&\leqslant\mathbb{E}\bigg[\int_{t_l}^{t_{l+1}}\abs{A_0(X(t_l)-X(s))}^2\\
&\quad\quad+\abs{L_j^{(2)}\bigg((X(t_l)-X(s))+\sum_{k=1}^K(X(t_l-\tau_k)-X(s-\tau_k))\bigg)}^2\,\mathrm{d}s\bigg]\\
&\leqslant c_1(t_{l+1}-t_l)^2,
\end{align*}
for constant $c_1>0$, while
\begin{align*}
G&=\sum_{j=1}^m\sum_{i=1}^mA_j[A_iX(t_l)+g_i(t_l,X(t_l),\phi_{\sigma_b}(t_l))]\,I_{ij}(t_l,t_{l+1})\\
&\quad-\sum_{j=1}^m\int_{t_l}^{t_{l+1}}A_j(X(s)-X(t_l))\,\mathrm{d}W_j(s)\\
&=\sum_{j=1}^m\sum_{i=1}^mA_j[A_iX(t_l)+g_i(t_l,X(t_l),\phi_{\sigma_b}(t_l))]\,I_{ij}(t_l,t_{l+1})\\
&\quad-\sum_{j=1}^m\int_{t_l}^{t_{l+1}}A_j\bigg(\int_{t_l}^s[A_0X(u)+f(u)]\,\mathrm{d}u\\
&\quad\quad\quad\quad+\sum_{i=1}^m\int_{t_l}^s[A_iX(u)+g_i(u)]\,\mathrm{d}W_i(u)\bigg)\,\mathrm{d}W_j(s)\\
&=\sum_{j=1}^m\sum_{i=1}^m\int_{t_l}^{t_{l+1}}\int_{t_l}^sA_j\bigg([A_iX(t_l)+g_i(t_l,X(t_l),\phi_{\sigma_b}(t_l))]\\
&\quad\quad\quad\quad-[A_iX(u)+g_i(u)]\bigg)\,\mathrm{d}W_i(u)\,\mathrm{d}W_j(s)\\
&\quad-\sum_{j=1}^m\int_{t_l}^{t_{l+1}}\int_{t_l}^sA_j[A_0X(u)+f(u)]\,\mathrm{d}u\,\mathrm{d}W_j(s),\\
\end{align*}
so that $\mathbb{E}[\abs{G}^2]<c_2(t_{l+1}-t_l)^2$, for constant $c_2>0$, from \cref{NewLemma1Bound6}.

As for the second integral of \cref{Isum}, we apply the H\"older inequality, along with the fact that $Z$ is a geometric Brownian motion with known moments, to compute
\begin{align*}
\mathbb{E}&\left[\abs{I^{(2)}(t_l,t_{l+1})}^2\right]\leqslant\sum_{j=1}^m\mathbb{E}\bigg[\abs{\int_{t_l}^{t_{l+1}}\left(\int_{t_l}^s\underline{A}_0Z_{t_l}(u)\,\mathrm{d}u\right)g_j(s,X(s),\phi_{\sigma_b}(s))\,\mathrm{d}W_j(s)}^2\bigg]\\
&=\sum_{j=1}^m\int_{t_l}^{t_{l+1}}\mathbb{E}\bigg[\abs{\left(\int_{t_l}^s\underline{A}_0Z_{t_l}(u)\,\mathrm{d}u\right)g_j(s,X(s),\phi_{\sigma_b}(s))}^2\bigg]\,\mathrm{d}s\\
&\leqslant\sum_{j=1}^m\int_{t_l}^{t_{l+1}}\left(\mathbb{E}\bigg[\norm{\int_{t_l}^s\underline{A}_0Z_{t_l}(u)\,\mathrm{d}u}^4\bigg]\mathbb{E}\left[\abs{g_j(s,X(s),\phi_{\sigma_b}(s))}^4\right]\right)^{1/2}\,\mathrm{d}s,
\end{align*}
where $\norm{\cdot}$ denotes the operator norm (although this result holds for any consistent norm). Using \cref{NewLemma1Bound3} and the fourth moment of geometric Brownian motion, we establish \cref{I2}.

The third integral of \cref{Isum} is
\begin{equation*}
I^{(3)}(t_l,t_{l+1})=\sum_{j=1}^m\int_{t_l}^{t_{l+1}}\bigg(\sum_{i=1}^m\int_{t_l}^s\underline{A}_iZ_{t_l}(u)\,\mathrm{d}W_i(u)\bigg)g_j(s,X(s),\phi_{\sigma_b}(s))\,\mathrm{d}W_j(s),
\end{equation*}
which we expand as
\begin{align}
I^{(3)}(t_l,t_{l+1})&=\sum_{j=1}^m\sum_{i=1}^m\int_{t_l}^{t_{l+1}}\bigg(\int_{t_l}^s\underline{A}_iZ_{t_l}(u)\,\mathrm{d}W_i(u)\,\nonumber\\
&\hspace{3cm}\cdot[g_j(s,X(s),\phi_{\sigma_b}(s))-g_j(t_l,X(t_l),\phi_{\sigma_b}(t_l))]\bigg)\,\mathrm{d}W_j(s)\nonumber\\
&\quad-\sum_{j=1}^m\sum_{i=1}^m\underline{A}_i\int_{t_l}^{t_{l+1}}\int_{t_l}^sZ_{t_l}(u)\,\mathrm{d}W_i(u)\,g_j(t_l,X(t_l),\phi_{\sigma_b}(t_l))\,\mathrm{d}W_j(s).\label{I3TwoParts}
\end{align}
By substituting the result of \cref{LemmaPhiInverse} into \cref{I3TwoParts}, it follows that
\begin{align}
I^{(3)}(t_l,t_{l+1})&=\sum_{j=1}^m\sum_{i=1}^m\underline{A}_i\int_{t_l}^{t_{l+1}}\bigg(\int_{t_l}^sZ_{t_l}(u)\,\mathrm{d}W_i(u)\,\nonumber\\
&\hspace{3cm}\cdot[g_j(s,X(s),\phi_{\sigma_b}(s))-g_j(t_l,X(t_l),\phi_{\sigma_b}(t_l))]\bigg)\,\mathrm{d}W_j(s)\nonumber\\
&\quad+\sum_{j=1}^m\sum_{i=1}^m\underline{A}_iZ(t_l)g_j(t_l,X(t_l),\phi_{\sigma_b}(t_l))\,I_{ij}(t_l,t_{l+1})\label{I3ThreeParts}\\
&\quad+\sum_{j=1}^m\sum_{i=1}^m\int_{t_l}^{t_{l+1}}\int_{t_l}^s\bigg[\int_{t_l}^u\underline{A}_0Z(v)\,\mathrm{d}v\nonumber\\
&\hspace{2cm}+\sum_{q=1}^m\int_{t_l}^u\underline{A}_q\,\mathrm{d}W_q(v)\bigg]\,\mathrm{d}W_i(u)\,\mathrm{d}W_j(s)\,\cdot g_j(t_l,X(t_l),\phi_{\sigma_b}(t_l)).\nonumber
\end{align}
Applying the It\^o isometry, the Lipschitz bound for $g_j$, along with the H\"older inequality shows that the first term of \cref{I3ThreeParts} satisfies the mean-square bound
\begin{align*}
\mathbb{E}&\bigg[\bigg|\sum_{j=1}^m\sum_{i=1}^mA_i\int_{t_l}^{t_{l+1}}\bigg(\int_{t_l}^sZ_{t_l}(u)\,\mathrm{d}W_i(u)\,\\
&\hspace{4cm}\cdot[g_j(s,X(s),\phi_{\sigma_b}(s))-g_j(t_l,X(t_l),\phi_{\sigma_b}(t_l))]\bigg)\,\mathrm{d}W_j(s)\bigg|^2\bigg]\\
&\leqslant\sum_{j=1}^m\sum_{i=1}^mc(t_{l+1}-t_l)^2,
\end{align*}
for constant $c>0$, while the third term of \cref{I3ThreeParts} satisfies
\begin{align*}
\mathbb{E}&\bigg[\bigg|\sum_{j=1}^m\sum_{i=1}^m\int_{t_l}^{t_{l+1}}\int_{t_l}^s\bigg[\int_{t_l}^u\underline{A}_0Z(v)\,\mathrm{d}v\\
&\hspace{3cm}+\sum_{q=1}^m\int_{t_l}^u\underline{A}_q\,\mathrm{d}W_q(v)\bigg]\,\mathrm{d}W_i(u)\,\mathrm{d}W_j(s)\,g_j(t_l,X(t_l),\phi_{\sigma_b}(t_l))\bigg|^2\bigg]\\
&\leqslant\sum_{j=1}^m\sum_{i=1}^m\int_{t_l}^{t_{l+1}}\int_{t_l}^s\bigg(\mathbb{E}\bigg[\bigg\Vert A_i\int_{t_l}^u\underline{A}_0Z(v)\,\mathrm{d}v+A_i\sum_{q=1}^m\int_{t_l}^u\underline{A}_q\,\mathrm{d}W_q(v)\bigg\Vert^4\bigg]\\
&\hspace{4cm}\times\mathbb{E}\left[\abs{g_j(t_l,X(t_l),\phi_{\sigma_b}(t_l))}^4\right]\bigg)^{1/2}\,\mathrm{d}u\,\mathrm{d}s.
\end{align*}
Combining these shows
\begin{equation}I^{(3)}(t_l,t_{l+1})=-\sum_{j=1}^m\sum_{i=1}^mA_iZ(t_l)g_j(t_l,X(t_l),\phi_{\sigma_b}(t_l))\,I_{ij}(t_l,t_{l+1})+R^{(3)},\end{equation}
where $\mathbb{E}[\abs{R^{(3)}}^2]\leqslant c(t_{l+1}-t_l)^2$ for constant $c>0$.
\end{proof}

\section[Taylor expansions for Magnus approximations]{Proof of \cref{NewLemma3}}\label{ProofofNewLemma3}

\begin{proof}
Calculating the expectation of the integral form of \cref{dXHd} gives
\begin{equation}\mathbb{E}X_H(t)=\mathbb{E}X_H(t_n)+\mathbb{E}\int_{t_n}^tA_0X_H(s)\,\mathrm{d}s.\label{EXHIntegralForm}\end{equation}
The integral \cref{EXHIntegralForm} has known solution \cref{EXHconditional}. Using the Multinomial Theorem,
\begin{align*}
\exp\left(\Omega^{[1]}(t_n,t_{n+1})\right)&=I_d+\hat{A}_0\,h+\sum_{j=1}^mA_j\,\Delta W_j(t_n,t_{n+1})+\frac{\left(\hat{A}_0\,h\right)^2}{2}\\
&\quad+\frac{1}{2}\sum_{j=1}^m\left(\hat{A}_0\,A_j+A_j\,\hat{A}_0\right)\,h\,\Delta W_j(t_n,t_{n+1})\\
&\quad+\frac{1}{2}\sum_{i=1}^m\sum_{j=1}^mA_i\,A_j\,\Delta W_i(t_n,t_{n+1})\,\Delta W_j(t_n,t_{n+1})\\
&\quad+\sum_{z=2}^\infty\sum_{k_0+\cdots+k_m=z}\!\frac{z!}{k_0!\cdots k_m!}\!\left(\!\hat{A}_0\,h\!\right)^{k_0}\prod_{j=1}^m\left(A_j\,\Delta W_j(t_n,t_{n+1})\right)^{k_j}.
\end{align*}
Due to the independence of the Wiener processes, it follows that
\begin{align}
&\mathbb{E}\left[\exp\left(\Omega^{[1]}(t_n,t_{n+1})\right)\mid\exp\left(\Omega^{[1]}(t_n,t_n)\right)=I_d\right]\nonumber\\
&\quad=I_d+\hat{A}_0\,h+\frac{1}{2}\bigg(\hat{A}_0^2\,h^2+\sum_{j=1}^m A_j^2\,h\bigg)+R_1=I_d+A_0\,h+\frac{\left(\hat{A}_0\,h\right)^2}{2}+R_1,\label{LemmaEexpOmega1}
\end{align}
where
\[R_1=\sum_{z=2}^\infty\sum_{k_0+\cdots+k_m=z}\frac{z!}{k_0!\,\cdots\,k_m!}\left(\hat{A}_0\,h\right)^{k_0}\prod_{j=1}^mA_j^{k_j}\mathbb{E}\left[\Delta W_j(t_n,t_{n+1})^{k_j}\right].\]
The moments of the Wiener increments are known and given by
\[\mathbb{E}\left[\Delta W_j(t_n,t_{n+1})^{2l+1}\right]=0\quad\textrm{and}\quad\mathbb{E}\left[\Delta W_j(t_n,t_{n+1})^{2l}\right]=\frac{(2l)!\,h^l}{l!\,2^l},\quad l=0,1,\ldots,\]
which we substitute into \cref{LemmaEexpOmega1} to show \cref{EOmega1conditional}. Finally, we note that
\begin{equation}\mathbb{E}[I_{ji}-I_{ij}]=2\mathbb{E}[I_{ji}]\begin{cases}=0&\textrm{if }i=j\textrm{ or }i=0\textrm{ of }j=0,\\ <ch^2&\mathrm{otherwise},\end{cases}\label{EIji}\end{equation}
for constant $c>0$ (see \cite[page 191]{KloedenPlaten1999}). With this, we calculate
\begin{align*}
\mathbb{E}&\left[\exp\left(\Omega^{[2]}(t_n,t_{n+1})\right)\mid\exp\left(\Omega^{[2]}(t_n,t_n)\right)=I_d\right]\\
&=\sum_{z=0}^\infty\frac{1}{z!}\mathbb{E}\bigg[\hat{A}_0\,h+\sum_{j=1}^mA_j\,\Delta W_j(t_n,t_{n+1})+\frac{1}{2}\sum_{i=0}^m\sum_{j=i+1}^m[A_i,A_j]\,(I_{ji}-I_{ij})\bigg]^z\\
&=I_d+A_0\,h+\frac{\left(\hat{A}_0\,h\right)^2}{2}+R_2,
\end{align*}
where $R_2$ is a (constant) matrix multiple of $h^2$ and higher-order terms, showing the result.
\end{proof}

\section[Componentwise It\^o expansion]{An expression for the product of negative Magnus exponential and function}\label{ProofofNewLemma4}

\begin{lemma}\label{NewLemma4}
For $s\in[t_n,t_{n+1}]$, the value $E_{t_n}^s=\exp(-\Omega(t_n,s))y(s)$ is given by
\begin{align}
E_{t_n}^s&=y(t_n)+\int_{t_n}^s\exp(-\Omega(t_n,u))\nabla y(u)\!\cdot\![A_0X(u)+f(u)]\,\mathrm{d}u\nonumber\\
&\;+\sum_{j=1}^m\int_{t_n}^s\exp(-\Omega(t_n,u))\nabla y(u)\!\cdot\![A_jX(u)+g_j(u)]\,\mathrm{d}W_j(u)\label{ExpNegOmegatns}\\
&\;+\frac{1}{2}\sum_{j=1}^m\int_{t_n}^s[A_jX(u)+g_j(u)]^\intercal\exp(-\Omega(t_n,u))\nabla\nabla y(u)\!\cdot\![A_jX(u)+g_j(u)]\,\mathrm{d}u\nonumber\\
&\;+\int_{t_n}^s\underline{A}_0\exp(-\Omega(t_n,u))y(u)\,\mathrm{d}u+\sum_{j=1}^m\int_{t_n}^s(-A_j)\exp(-\Omega(t_n,u))y(u)\,\mathrm{d}W_j(u)\nonumber\\
&\;+\sum_{j=1}^m\int_{t_n}^s(-A_j)\exp(-\Omega(t_n,u))\nabla y(u)\!\cdot\![A_jX(u)+g_j(u)]\,\mathrm{d}u,\nonumber
\end{align}
for each $y=f,\tilde{f},g_1,\ldots,g_m$, where $\nabla\nabla y$ denotes the Hessian matrix of $y$.
\end{lemma}

\begin{proof}
We establish a componentwise expression for $E_{t_n}^s=\exp(-\Omega(t_n,s))y(s)$, by applying the It\^o product rule (also known as integration by parts for It\^o processes). Writing $X(u)=(X_1(u),\ldots,X_d(u))^\intercal$, $y=(y_1,\ldots,y_d)^\intercal$, $g_j=(g_{1,j},\ldots,g_{d,j})^\intercal$, $\exp(-\Omega(t_n,s))=(e_{ab}(s))_{a,b=1,\ldots,d}$, $\underline{A}_0=(\underline{A}_0^{ab})_{a,b=1,\ldots,d}$, and $A_j=(A_j^{ab})_{a,b=1,\ldots,d}$ (for $j=0,1,\ldots,m$), in component and matrix form, the $a^\mathrm{th}$ component of $E_{t_n}^s$ is
\begin{align}
\sum_{b=1}^de_{ab}(s)y_b(s)&=\sum_{b=1}^d\bigg(e_{ab}(t_n)y_b(t_n)+\int_{t_n}^se_{ab}(u)\,\mathrm{d}y_b(u)\nonumber\\
&\quad\quad+\int_{t_n}^sy_b(u)\,\mathrm{d}e_{ab}(u)+\int_{t_n}^s\,\mathrm{d}\left[e_{ab}(u),y_b(u)\right]\bigg)\nonumber\\
&=y_a(t_n)+\sum_{b=1}^d\bigg[\int_{t_n}^se_{ab}(u)\bigg(\sum_{i=1}^d\frac{\partial y_b(u)}{\partial x_i}\,\mathrm{d}X_i(u)\label{BDA}\\
&\quad\quad\quad\quad\quad+\frac{1}{2}\sum_{i=1}^d\sum_{j=1}^d\frac{\partial^2y_b(u)}{\partial x_i\,\partial x_j}\,\mathrm{d}\left[X_i(u),X_j(u)\right]\bigg)\nonumber\\
&\quad+\int_{t_n}^sy_b(u)\bigg(\underline{A}_0^{ab}e_{ab}(u)\,\mathrm{d}u+\sum_{j=1}^m\left(-A_j^{ab}\right)e_{ab}(u)\,\mathrm{d}W_j(u)\bigg)\nonumber\\
&\quad+\int_{t_n}^s\bigg(\sum_{i=1}^d\frac{\partial y_b(u)}{\partial x_i}\,\mathrm{d}X_i(u)+\frac{1}{2}\sum_{i=1}^d\sum_{j=1}^d\frac{\partial^2y_b(u)}{\partial x_i\,\partial x_j}\,\mathrm{d}\left[X_i(u),X_j(u)\right]\bigg)\nonumber\\
&\quad\quad\quad\times\bigg(\underline{A}_0^{ab}e_{ab}(u)\,\mathrm{d}u+\sum_{j=1}^m\left(-A_j^{ab}\right)e_{ab}(u)\,\mathrm{d}W_j(u)\bigg)\bigg].\nonumber
\end{align}
We calculate
\[\mathrm{d}[X_i(u),X_j(u)]=\sum_{c=1}^d\sum_{l=1}^m\left(A_l^{ic}X_c(u)+g_{i,l}(u)\right)\left(A_l^{jc}X_c(u)+g_{j,l}(u)\right)\,\mathrm{d}u,\]
which substituted into \cref{BDA} yields
\begin{align}
\sum_{b=1}^de_{ab}(s)y_b(s)&=y_a(t_n)+\sum_{b=1}^d\int_{t_n}^se_{ab}(u)\sum_{i=1}^d\frac{\partial y_b(u)}{\partial x_i}\left[\left(\sum_{c=1}^dA_0^{ic}X_c(u)\right)+f_i(u)\right]\,\mathrm{d}u\nonumber\\
&\;+\sum_{b=1}^d\int_{t_n}^se_{ab}(u)\sum_{i=1}^d\frac{\partial y_b(u)}{\partial x_i}\sum_{j=1}^m\left[\left(\sum_{c=1}^dA_j^{ic}X_c(u)\right)+g_{i,j}(u)\right]\,\mathrm{d}W_j(u)\label{BEA}\\
&\;+\sum_{b=1}^d\int_{t_n}^se_{ab}(u)\frac{1}{2}\sum_{i=1}^d\sum_{j=1}^d\frac{\partial^2y_b(u)}{\partial x_i\,\partial x_j}\sum_{c=1}^d\sum_{l=1}^m\left(A_l^{ic}X_c(u)+g_{i,l}(u)\right)\nonumber\\
&\hspace{6cm}\times\left(A_l^{jc}X_c(u)+g_{j,l}(u)\right)\,\mathrm{d}u\nonumber\\
&\;+\sum_{b=1}^d\int_{t_n}^sy_b(u)\underline{A}_0^{ab}e_{ab}(u)\,\mathrm{d}u+\sum_{b=1}^d\int_{t_n}^sy_b(u)\sum_{j=1}^m(-A_j^{ab})e_{ab}(u)\,\mathrm{d}W_j(u)\nonumber\\
&\;+\sum_{b=1}^d\int_{t_n}^s\!\sum_{i=1}^d\frac{\partial y_b(u)}{\partial x_i}\!\sum_{j=1}^m\left[\left(\!\sum_{c=1}^dA_j^{ic}X_c(u)\!\right)+g_{i,j}(u)\right]\!\left(-A_j^{ab}\right)e_{ab}(u)\,\mathrm{d}u.\nonumber
\end{align}
Writing \cref{BEA} in matrix form shows \cref{ExpNegOmegatns}.
\end{proof}

\section[Proof of MEM errors]{Proof of \cref{MEMConditions}}\label{ProofofMEMConditions}

\begin{proof}[Proof of \cref{MEMConditions}]
We first show the bound \cref{WeakBound}. Conditional on $Y_n=X(t_n)=x$, we express
\begin{equation}Y_{n+1}-X(t_n)=D_1+D_2+D_3,\label{DifferenceExpanded}\end{equation}
where
\begin{align*}
D_1&=\exp\left(\Omega^{[1]}(t_n,t_{n+1})\right)x-\exp\left(\Omega(t_n,t_{n+1})\right)x,\\
D_2&=\exp\left(\Omega^{[1]}(t_n,t_{n+1})\right)\int_{t_n}^{t_{n+1}}\tilde{f}(t_n,X(t_n),\phi_{\sigma_b}(t_n))\,\mathrm{d}s\\
&\quad-\exp\left(\Omega(t_n,t_{n+1})\right)\int_{t_n}^{t_{n+1}}\exp\left(-\Omega(t_n,s)\right)\tilde{f}(s,X(s),\phi_{\sigma_b}(s))\,\mathrm{d}s,\\
D_3&=\exp\left(\Omega^{[1]}(t_n,t_{n+1})\right)\sum_{j=1}^m\int_{t_n}^{t_{n+1}}g_j(t_n,X(t_n),\phi_{\sigma_b}(t_n))\,\mathrm{d}W_j(s)\\
&\quad-\exp\left(\Omega(t_n,t_{n+1})\right)\sum_{j=1}^m\int_{t_n}^{t_{n+1}}\exp\left(-\Omega(t_n,s)\right)g_j(s,X(s),\phi_{\sigma_b}(s))\,\mathrm{d}W_j(s).
\end{align*}
To simplify notation, we write $\mathbb{E}_x[\cdot]=\mathbb{E}[\cdot|Y_n=X(t_n)=x]$. Using \cref{NewLemma3}, we evaluate
\begin{equation*}
\mathbb{E}_x[D_1]=\left(R_1-\sum_{z=3}^\infty\frac{(A_0\,h)^z}{z!}\right)x,
\end{equation*}
showing
\begin{equation}\abs{\mathbb{E}_x[D_1]}\leqslant M\abs{x}h^2\leqslant M\left(1+\abs{x}^2\right)^{1/2}\,h^2,\label{D1WeakMEM}\end{equation}
for constant $M>0$. The term $D_2$ is re-expressed by subtracting and adding the value $\exp(\Omega(t_n,t_{n+1}))\int_{t_n}^{t_{n+1}}\tilde{f}(t_n)\,\mathrm{d}s$, to become
\begin{align}
D_2&=\left[\exp(\Omega^{[1]}(t_n))-\exp(\Omega(t_n))\right]\int_{t_n}^{t_{n+1}}\tilde{f}(t_n)\,\mathrm{d}s\label{D2Difference}\\
&\quad-\exp(\Omega(t_n,t_{n+1}))\int_{t_n}^{t_{n+1}}\exp(-\Omega(t_n,s))\tilde{f}(s)-\tilde{f}(t_n)\,\mathrm{d}s.
\end{align}
Similarly to above, \cref{NewLemma3} shows that the first term of \cref{D2Difference} satisfies
\begin{equation*}
\abs{\mathbb{E}_x\left[\left(\exp\left(\Omega^{[1]}(t_n,t_{n+1})\right)-\exp(\Omega(t_n,t_{n+1}))\right)\int_{t_n}^{t_{n+1}}\tilde{f}(t_n)\,\mathrm{d}s\right]}\leqslant ch^3,
\end{equation*}
for constant $c>0$. As for the second term of \cref{D2Difference}, we use \cref{ExpNegOmegatns} to show
\[\abs{\mathbb{E}_x\left[-\exp(\Omega(t_n,t_{n+1}))\int_{t_n}^{t_{n+1}}\left(\exp(-\Omega(t_n,s))\tilde{f}(s)-\tilde{f}(t_n)\right)\,\mathrm{d}s\right]}\leqslant ch^2,\]
for constant $c>0$. From these, it follows that
\begin{equation}\abs{\mathbb{E}_x[D_2]}\leqslant M\abs{x}h^2\leqslant M\left(1+\abs{x}^2\right)^{1/2}\,h^2,\label{D2WeakMEM}\end{equation}
for constant $M>0$. The third term $D_3$ is treated similarly to $D_2$, where we subtract and add
\[\sum_{j=1}^m\exp(\Omega(t_n,t_{n+1}))\int_{t_n}^{t_{n+1}}g_j(t_n)\,\mathrm{d}W_j(s)\]
to express
\begin{align}
D_3&=\left[\exp\left(\Omega^{[1]}(t_n,t_{n+1})\right)-\exp(\Omega(t_n,t_{n+1}))\right]\sum_{j=1}^m\int_{t_n}^{t_{n+1}}g_j(t_n)\,\mathrm{d}W_j(s)\nonumber\\
&\quad-\exp(\Omega(t_n,t_{n+1}))\sum_{j=1}^m\int_{t_n}^{t_{n+1}}\left(\exp(-\Omega(t_n,s))g_j(s)-g_j(t_n)\right)\,\mathrm{d}W_j(s).\label{D3}
\end{align}
Using Taylor expansions, we find that the first term of \cref{D3} satisfies
\begin{align*}
&\bigg|\mathbb{E}_x\bigg[\left(\exp\left(\Omega^{[1]}(t_n,t_{n+1})\right)-\exp(\Omega(t_n,t_{n+1}))\right)\sum_{j=1}^m\int_{t_n}^{t_{n+1}}g_j(t_n)\,\mathrm{d}s\bigg]\bigg|\\
&\quad\leqslant\bigg|\mathbb{E}\bigg[\left(\frac{-1}{2}\sum_{i=1}^mA_i^2\,h+R\right)\sum_{j=1}^mg_j(t_n)\int_{t_n}^{t_{n+1}}\,\mathrm{d}W_j(s)\bigg]\bigg|\leqslant Mh^2,
\end{align*}
where $M>0$ is constant and the remainder $R$ satisfies $\abs{\mathbb{E}[R]}\leqslant ch^2$ for another constant $c>0$. The second term of \cref{D3} is treated similarly to the case for $D_2$, using \cref{NewLemma4}, which shows
\begin{align}
\mathbb{E}_x&\bigg[-\exp(\Omega(t_n,t_{n+1}))\sum_{j=1}^m\int_{t_n}^{t_{n+1}}\exp(-\Omega(t_n,s))g_j(s)-g_j(t_n)\,\mathrm{d}W_j(s)\bigg]\nonumber\\
&=\mathbb{E}_x\bigg[-\exp(\Omega(t_n,t_{n+1}))\sum_{j=1}^m\int_{t_n}^{t_{n+1}}\int_{t_n}^sF(u)\,\mathrm{d}u,\label{BAA}\\
&\hspace{3cm}+\sum_{i=1}^m\int_{t_n}^sG_i(u)\,\mathrm{d}W_i(u)\,\mathrm{d}W_j(s)\bigg],\nonumber
\end{align}
for processes $F,G_1,\ldots,G_m$ satisfying
\[\sup_{u\in[t_n,t_{n+1}]}\mathbb{E}_x\left[\abs{F(u)}^2\right]<\infty\quad\textrm{and}\quad\sup_{u\in[t_n,t_{n+1}]}\mathbb{E}_x\left[\abs{G_j(u)}^2\right]<\infty,\quad j=1,\ldots,m.\]
Applying the H\"older inequality shows \cref{BAA} has magnitude bounded by a constant multiple of $h^2$. Therefore
\begin{equation}\mathbb{E}_x[D_3]\leqslant M\left(1+\abs{x}^2\right)^{1/2}\,h^2,\label{D3WeakMEM}\end{equation}
for constant $M>0$. By combining \cref{D1WeakMEM}, \cref{D2WeakMEM}, and \cref{D3WeakMEM}, it follows that \cref{WeakBound} holds, with $q_1=2$.

To show \cref{MSPart}, we again use the decomposition \cref{DifferenceExpanded} and \cref{NewLemma3}, with which we calculate
\begin{equation}\mathbb{E}_x\left[\abs{D_1}^2\right]^{1/2}=\mathbb{E}_x\bigg[\norm{\frac{h}{2}\sum_{i=1}^mA_i^2+R}^2\abs{x}^2\bigg]^{1/2}\leqslant M\abs{x}\,h,\label{D1MSMEM}\end{equation}
for constant $M>0$ and remainder $R$ consisting of higher-order (in $h$) terms. Using \cref{D2Difference}, we calculate
\begin{equation}\mathbb{E}_x\left[\abs{D_2}^2\right]^{1/2}\leqslant c_1 h^{3/2} + c_2 h^2 \leqslant Mh^{3/2},\label{D2MSMEM}\end{equation}
for constant $M>0$. We also calculate
\begin{equation}\mathbb{E}_x\left[\abs{D_3}^2\right]^{1/2}\leqslant Mh,\label{D3MSMEM}\end{equation}
for constant $M>0$, similarly. By combining \cref{D1MSMEM}, \cref{D2MSMEM}, and \cref{D3MSMEM}, the bound \cref{MSPart} holds, with $q_2=1$.
\end{proof}

\section[Proof of MM errors]{Proof of \cref{MMConditions}}\label{ProofofMMConditions}

\begin{proof}[Proof of \cref{MMConditions}]
We again let
\begin{equation}Y_{n+1}-X(t_{n+1})=D_1+D_2+D_3\label{MMDifferenceExpanded},\end{equation}
conditional on $Y_n=X(t_n)=x$, where
\begin{align*}
D_1&=\left(\exp\left(\Omega^{[2]}(t_n,t_{n+1})\right)-\exp(\Omega(t_n,t_{n+1}))\right)x,\\
D_2&=\exp\left(\Omega^{[2]}(t_n,t_{n+1})\right)\int_{t_n}^{t_{n+1}}\tilde{f}(t_n)\,\mathrm{d}s\\
&\quad-\exp(\Omega(t_n,t_{n+1}))\int_{t_n}^{t_{n+1}}\exp(-\Omega(t_n,s))\tilde{f}(s)\,\mathrm{d}s,\\
D_3&=\exp\left(\Omega^{[2]}(t_n,t_{n+1})\right)\sum_{j=1}^m\bigg\{\int_{t_n}^{t_{n+1}}g_j(t_n)\,\mathrm{d}W_j(s)\\
&\hspace{3cm}+\sum_{i=1}^m\left(\nabla_xg_j(t_n)[A_iY_n+g_i(t_n)]-A_ig_j(t_n)\right)\,I_{ij}(t_n,t_{n+1})\\
&\hspace{3cm}+\sum_{k=1}^K\mathbb{I}(t_n\geqslant\tau_k)\sum_{i=1}^m\nabla_{\tau_k}g_j(t_n)[A_iY_n^{\tau_k}+g_i(t_n^{\tau_k})]\,I_{ij}^{\tau_k}(t_n,t_{n+1})\bigg\}\\
&\hspace{1cm}-\exp(\Omega(t_n,t_{n+1}))\sum_{j=1}^m\int_{t_n}^{t_{n+1}}\exp(-\Omega(t_n,s))g_j(s)\,\mathrm{d}W_j(s).
\end{align*}
Using \cref{NewLemma3}, we calculate
\begin{equation}\abs{\mathbb{E}_x[D_1]}\leqslant M\abs{x}\,h^2,\label{D1WeakMM}\end{equation}
for constant $M>0$. The remainder of this proof is similar to that for \cref{MEMConditions}, and shows
\begin{equation}\abs{\mathbb{E}_x[D_2]}\leqslant M_2\,h^2,\label{D2WeakMM}\end{equation}
and
\begin{equation}\abs{\mathbb{E}_x[D_3]}\leqslant M_3\,h^2,\label{D3WeakMM}\end{equation}
for constants $M_2,M_3>0$. Combining \cref{D1WeakMM}, \cref{D2WeakMM}, and \cref{D3WeakMM} shows \cref{WeakBound}, with $q_1=2$.

To show \cref{MSPart}, we again use \cref{MMDifferenceExpanded} and \cref{NewLemma3}, to calculate
\begin{align*}
\mathbb{E}\left[\abs{D_1}^2\right]^{1/2}&\leqslant M_1\abs{x}\,h^{3/2},\\
\mathbb{E}\left[\abs{D_2}^2\right]^{1/2}&\leqslant M_2\,h^{3/2},\\
\mathbb{E}\left[\abs{D_3}^2\right]^{1/2}&\leqslant M_3\,h^{3/2},
\end{align*}
for constants $M_1,M_2,M_3>0$, which shows the result, with $q_2=3/2$.
\end{proof}

\section*{Link to Matlab scripts}
The MATLAB R2024a scripts used to generate the figures in this paper are available at:\\
\url{https://github.com/Mitchell-Griggs/MagnusSDDEsScripts}

\end{document}